\documentclass[12pt]{article}

\usepackage{
   amsthm,
   amsmath,
   amssymb,
   graphicx,
   hyperref,
  mathtools,
  fullpage
   }
\usepackage[shortlabels]{enumitem}

\usepackage{mathrsfs}
\usepackage[margin=0pt]{subcaption}
\usepackage{stmaryrd}
\usepackage{tikz-cd}
\usepackage{needspace}

\usepackage[
backend=biber,
style=alphabetic,giveninits=true
]{biblatex}
\theoremstyle{plain}
\newtheorem{thm}{Theorem}[section]
\newtheorem{lemma}[thm]{Lemma}
\newtheorem{propn}[thm]{Proposition}
\newtheorem{cor}[thm]{Corollary}
\theoremstyle{definition}
\newtheorem{defn}[thm]{Definition}
\newtheorem{rmk}[thm]{Remark}
\newtheorem{ex}[thm]{Example}

\newtheorem*{thm*}{Theorem}

\renewcommand{\subset}{\subseteq}
\newcommand{\R}{\mathbb{R}}
\newcommand{\Z}{\mathbb{Z}}
\newcommand{\K}{\mathbb{K}}
\newcommand{\Hom}{\operatorname{Hom}}
\newcommand{\GermAut}{\operatorname{GermAut}}
\newcommand{\NullAut}{\operatorname{NullAut}}
\newcommand{\Aut}{\operatorname{Aut}}
\newcommand{\Diff}{\operatorname{Diff}}
\newcommand{\A}{\mathcal{A}}
\newcommand{\F}{\mathcal{F}}
\newcommand{\id}{\mathrm{id}}
\renewcommand{\Gauge}{\operatorname{Gauge}}
 
\newcommand{\nospacepunct}[1]{\makebox[0pt][l]{\,#1}}

\begin{document}

\title{On singular foliations tangent to a given hypersurface}
\author{Michael Francis}

\maketitle

\begin{abstract}
We consider a class of singular foliations in the sense of Androulidakis and Skandalis that we call transverse order $k$ foliations. These have a finite number of leaves: one hypersurface (the singular leaf) together with the components of its complement (open leaves). The positive integer parameter $k$ encodes the ``order of tangency'' of the leafwise vector fields to $L$. We show that a loop in the singular leaf induces a well-defined holonomy transformation at the level of $(k-1)$-jets. The resulting holonomy invariant can be used to give a complete classification of these foliations and obtain concrete descriptions of their associated groupoids and algebras.  
\end{abstract}



\section{Introduction}

Note that, except for minor changes, the contents of this article are duplicated in Chapters~1,~2 and 5 of the author's PhD dissertation \cite{Francis[PhD]}.

The interplay between foliation theory and operator theory is a significant aspect of Connes' noncommutative geometry program \cite{Connes[SURVEY]}.  A key construction in this area is  the C*-algebra $C^*(\F)$ of  a regular foliation $\F$. This begins with the construction, due to Winkelnkemper \cite{Winkelnkemper}, of a  (possibly non-Hausdorff) Lie groupoid $G(\F)$ called the \emph{holonomy groupoid} or \emph{graph} of $\F$.

Since many natural examples of foliations are not regular, but instead present singularities of some type or another, it is desirable to extend these constructions so that they also apply in singular cases. A number of authors have done work on this topic, see \cite{Pradines[1984]}, \cite{Pradines[1985]}, \cite{Pradines-Bigonnet}, \cite{Debord[2001a]}, \cite{Debord[2001b]}, \cite{AS[2007]}, \cite{AZ[2014]}. The most broadly applicable construction is the one given by Androulidakis and Skandalis in \cite{AS[2007]} and it is their approach that we are concerned with here.

We also follow \cite{AS[2007]} in understanding a foliation to be any locally finitely-generated $C^\infty(M)$-module $\F$ of compactly-supported vector fields on $M$ that is closed under Lie bracket. 
A leaf of $\F$ is then the set of points accessible from a given point by composing flows of vector fields in $\F$. By work of   Stefan and Sussmann (\cite{Stefan}, \cite{Sussmann}), the leaves of $\F$ constitute a partition of $M$ into immersed submanifolds (generalizing the Frobenius theorem).
A foliation is  \emph{regular} if all its leaves have the same dimension and \emph{singular} otherwise.  In the regular setting, the module of vector fields can be recovered from the partition, but this fails in the singular setting. In fact, varying the module $\F$ while keeping the partition the same  will be a prominent theme in this work.

In \cite{AS[2007]}, given any singular foliation $\F$, Androulidakis and Skandalis constructed the following objects:
\begin{enumerate}
\item A holonomy groupoid $G(\F)$.
In general, this is only  a topological groupoid, and its topology can be very wild. It is a Lie groupoid if and only if $\F$ is \emph{almost regular}, a hypothesis satisfied by all foliations studied in this article.
\item A smooth convolution algebra $\A(\F)$. In the almost regular case, $\A(\F)\cong C_c^\infty(G(\F))$, after choosing a smooth Haar system in order to make sense of convolution.
\item A C*-algebra $C^*(\F)$, obtained by completing $\A(\F)$.
In the almost regular case, this is the usual groupoid C*-algebra in the sense of \cite{Renault[BOOK]}.
\end{enumerate}

We  now give an informal discussion of holonomy to indicate how things change in the singular setting. Figure~\ref{fig:leaves} depicts the leaves of the following three foliations  of the cylinder $S^1 \times \R$, which we  regard as having coordinates $(x,y)$ where  $x$ is $\Z$-periodic:
\begin{align*}
\F\{\tfrac{d}{dx}+y\tfrac{d}{dy}\} &&\F\{\tfrac{d}{dx}+y^2\tfrac{d}{dy}\} && \F\{\tfrac{d}{dx},y\tfrac{d}{dy}\}.
\end{align*}
Here, the notation $\F\{X_1,\ldots,X_n\}$ refers to the foliation generated by a finite set of vector fields $X_1,\ldots,X_n$. 
The first two of these  are regular foliations with 1-dimensional leaves while the third is a singular foliation whose leaves are $S^1 \times \R_+$, $S^1 \times \{0\}$ and $S^1 \times \R_-$, where $\R_+\coloneqq (0,\infty)$, $\R_-\coloneqq(-\infty,0)$.

\begin{figure}[ht]
     \centering
     \begin{subfigure}[b]{0.32\textwidth}
         \centering
    \def\svgwidth{.4\textwidth}
    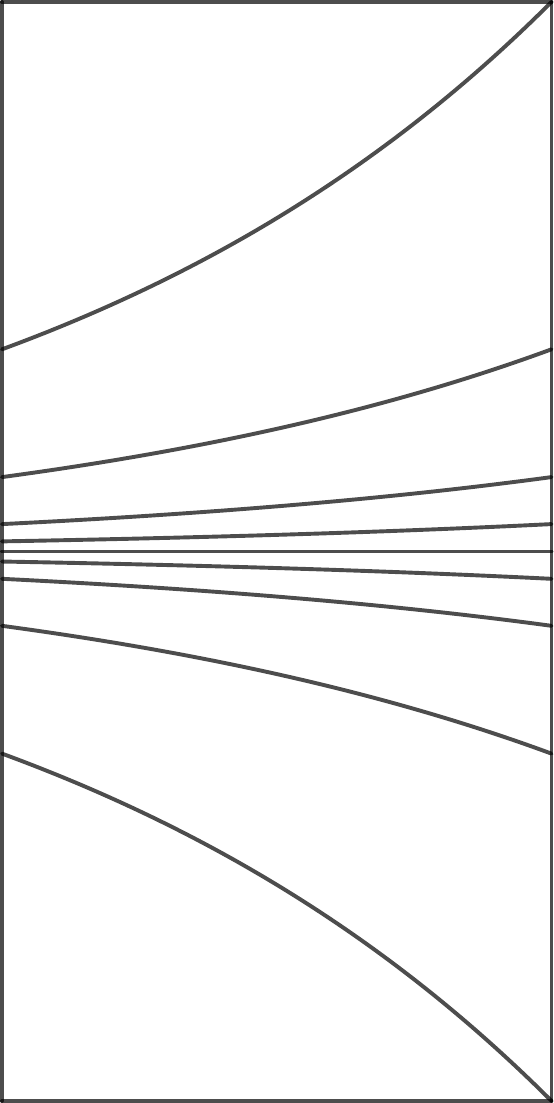
         \caption*{$\F\{\frac{d}{dx}+y\frac{d}{dy}\}$}
         \label{fig:leaves1}
     \end{subfigure}
     \hfill
     \begin{subfigure}[b]{0.32\textwidth}
         \centering
    \def\svgwidth{.4\textwidth}
    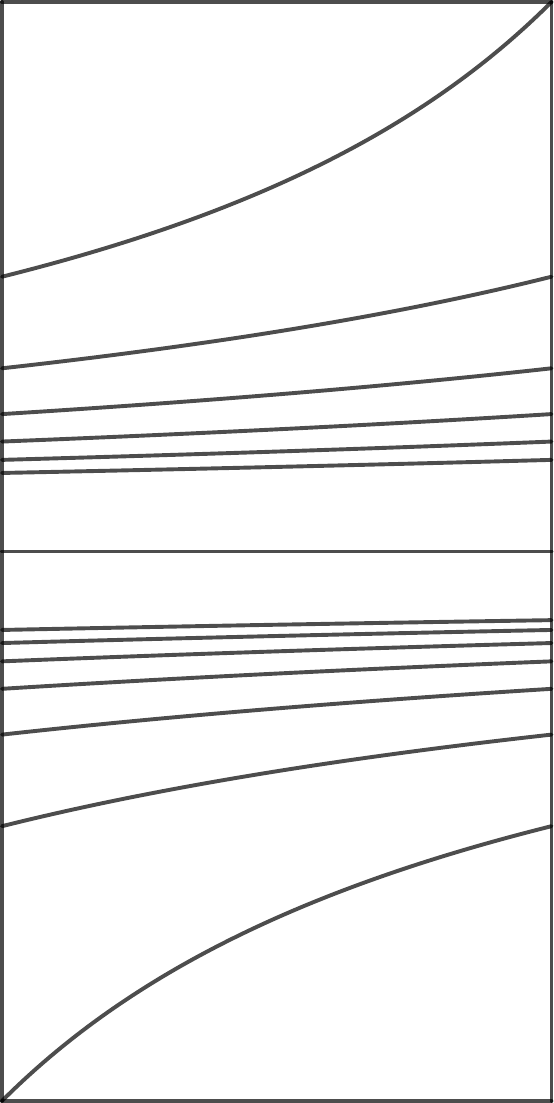
         \caption*{$\F\{\frac{d}{dx} + y^2 \frac{d}{dy}\}$}
         \label{fig:leaves2}
     \end{subfigure}
     \hfill
     \begin{subfigure}[b]{0.32\textwidth}
         \centering
    \def\svgwidth{.4\textwidth}
    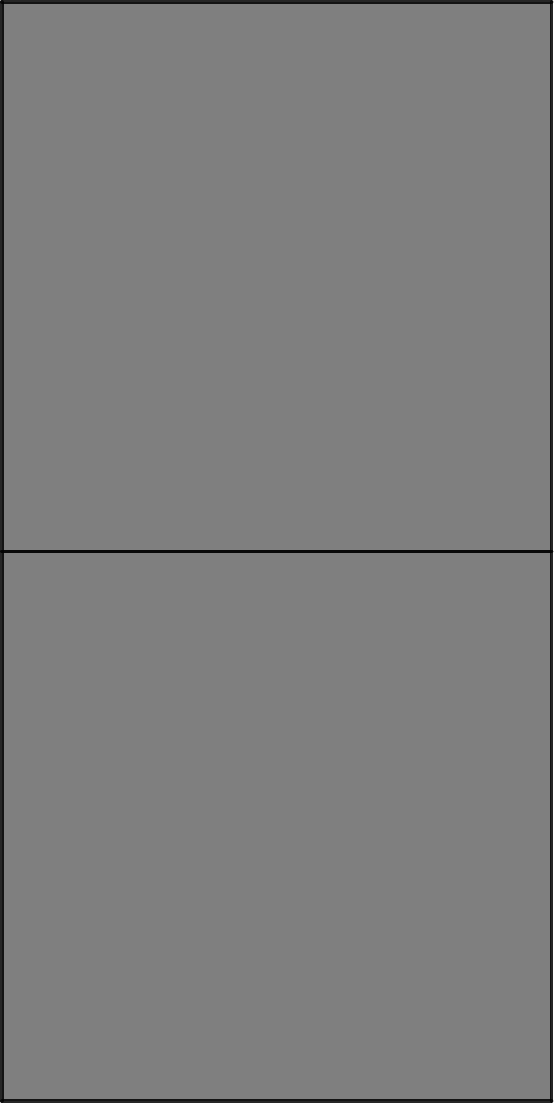
         \caption*{$\F\{\frac{d}{dx}, y\frac{d}{dy}\}$.}
         \label{fig:leaves3}
     \end{subfigure}
        \caption{Leaves of some  foliations of $S^1 \times \R$.}
        \label{fig:leaves}
\end{figure}

All three of these foliations determine nonsmooth equivalence relations. The issue becomes apparent when we restrict the leaf equivalence relations to the submanifold $T=\{0\} \times \R$ passing through   $p=(0,0)$. The  resulting subsets of $T\times T$ are depicted in Figure~\ref{fig:reln}.

\begin{figure}[ht]
     \centering
     \begin{subfigure}[b]{0.32\textwidth}
         \centering
    \def\svgwidth{.4\textwidth}
    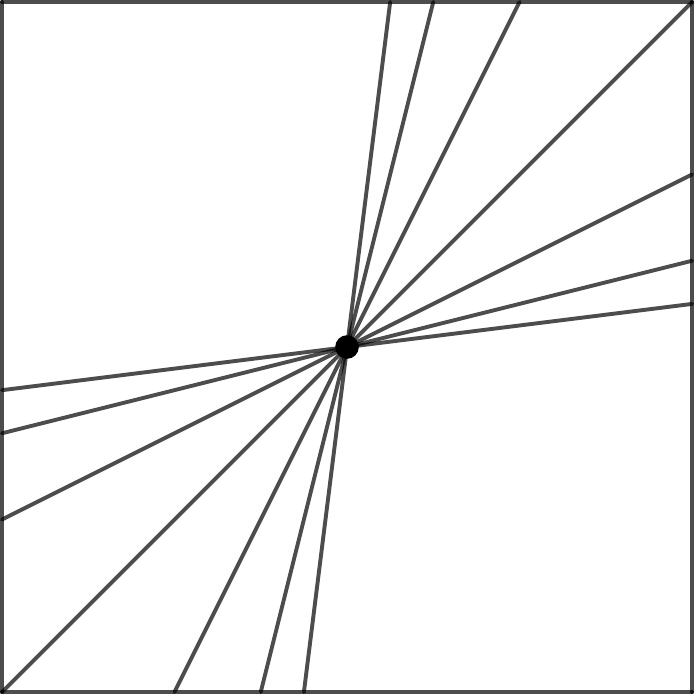
         \label{fig:reln1}
                  \caption*{$\F\{\frac{d}{dx}+y\frac{d}{dy}\}$}
     \end{subfigure}
     \hfill
     \begin{subfigure}[b]{0.32\textwidth}
         \centering
    \def\svgwidth{.4\textwidth}
    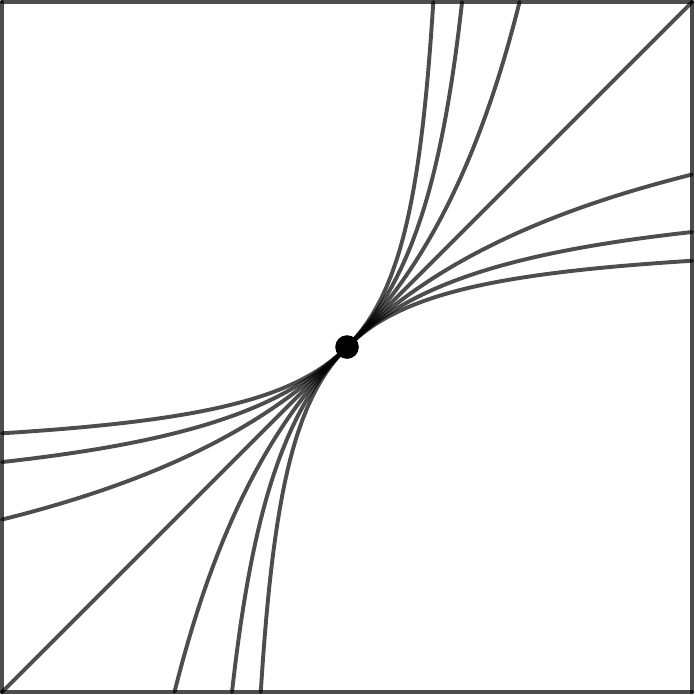
         \label{fig:reln2}
                  \caption*{$\F\{\frac{d}{dx} + y^2 \frac{d}{dy}\}$}
     \end{subfigure}
     \hfill
     \begin{subfigure}[b]{0.32\textwidth}
         \centering
    \def\svgwidth{.4\textwidth}
    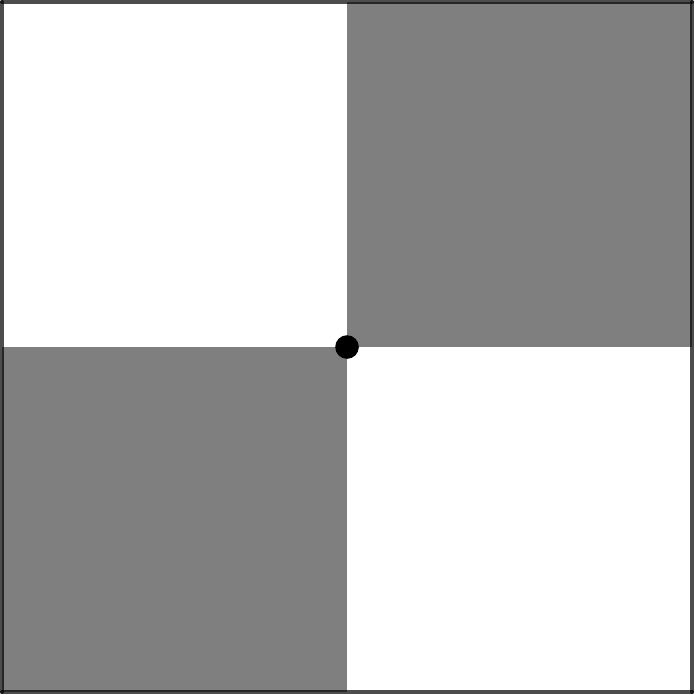
         \label{fig:reln3}
                  \caption*{$\F\{\frac{d}{dx}, y\frac{d}{dy}\}$.}
     \end{subfigure}
        \caption{Equivalence relations of some foliations of $S^1 \times \R$, restricted to $T=\{0\} \times \R$.}
        \label{fig:reln}
\end{figure}

Despite the singularities apparent in Figure~\ref{fig:reln}, the holonomy groupoid  of each of these foliations is smooth. In terms of these pictures,
what occurs is that the problematic point $(p,p)$ at the origin  gets blown up and replaced with the \emph{holonomy group} at $p$. For a regular foliation, given a point $p$ and a transversal submanifold $T$ through the leaf of $p$, the holonomy group at $p$ may be viewed as the discrete group consisting of all (germs of) diffeomorphisms of $T$ fixing $p$ which can be obtained using flows of vector fields in $\F$. For both of the regular foliations in Figure~\ref{fig:leaves},  this holonomy group is infinite cyclic, and it is easy to imagine how such a replacement can resolve the singularity. To put it more plainly, for each of the two regular foliations, the  equivalence relation in  Figure~\ref{fig:reln} is the union of the graphs of  a countable family of graphs of  diffeomorphisms which intersect only at the origin. After performing the blowup, one is left instead with the \emph{disjoint union} of these graphs.

A striking difference between the regular and singular  settings is that, whereas for regular foliations holonomy is purely a discrete phenomenon, for singular foliations one can also have \emph{continuous holonomy}. For the singular foliation  $\F\{\frac{d}{dx},y\frac{d}{dy}\}$ shown above, the group of diffeomorphism germs of $T$ which can be obtained using compositions of flows is infinite-dimensional. One perspective is that the work of Androulidakis and Skandalis  identifies the correct way to take a quotient of this infinite-dimensional group and obtain a finite-dimensional Lie group which serves as the natural generalization of the usual holonomy group.

For the foliation $\F\{\frac{d}{dx},y\frac{d}{dy}\}$ above, the holonomy group at $p=(0,0)$ may be identified with the Lie group of linear, orientation-preserving diffeomorphisms of $T$ and, in particular, it is isomorphic to $\R$.  However, this is just one of many foliations   of $S^1 \times \R$ whose leaves are $S^1 \times \R_+$, $S^1 \times \{0\}$ and $S^1 \times \R_-$. With the exception of some pathological examples, the holonomy group at $p$ of any such foliation is naturally realized, for some positive integer $k$ encoding the ``transverse order'' of the foliation, as a one-dimensional subgroup of the group $J^k$ of $k$-jets of   diffeomorphisms of $\R$ which fix $0$. Explicitly, 
\[J^k = \{ a_1y  + a_2 y^2 + \ldots a_\ell y^k: a_i \in\R , a_1 \neq 0 \} \] under the operation ``compose and truncate''.  Some examples of holonomy groups which can occur are tabulated in Table~\ref{table}.

\begin{table}[htbp]
\begin{tabular}{l|l|l}
Foliation 
& Holonomy group  & Ambient group \\  \hline 
$\F\{\tfrac{d}{dx},  y\tfrac{d}{dy}\}$ & 
$\{ e^t y : t \in \R\} $ & $ J^1$\\
$\F\{\tfrac{d}{dx},   y^2 \tfrac{d}{dy}\}$ & $\{y + ty^2 : t \in \R\} $ & $ J^2$ \\
$\F\{ \tfrac{d}{dx}+y\tfrac{d}{dy},  y^2 \tfrac{d}{dy}\}$ & $\{e^ny + ty^2 : n \in \Z, t \in \R  \} $ & $ J^2$ \\
$\F\{\tfrac{d}{dx} + y^2 \tfrac{d}{dy}, y^4 \tfrac{d}{dy}\}$ & $\{y + ny^2 + n^2 y^3 + t y^4 : n \in \Z, t \in \R\} $ & $ J^4$
\end{tabular}
\caption{Holonomy groups at $p=(0,0)$ of several different foliations of $S^1 \times \R$, all of which have  leaves  $S^1 \times \R_+$, $S^1 \times \{0\}$ and $S^1 \times \R_-$.}
\label{table}
\end{table}

The precise details of how $(p,p)$ is blown up into a copy of the holonomy group at $p$ also depends on two natural orderings of the group $J^k$, associated to the positive and negative half lines. As Figure~\ref{fig:strangeblups} shows, the topological possibilities for the blowup space are actually quite rich, especially given how simple the  leaf space of these foliations is. The last two surfaces, for example,   are not homeomorphic, as can be seen by counting the number of topological ends.

 \begin{figure}[htbp]
     \centering
          \hfill
     \begin{subfigure}[b]{0.16\textwidth}
         \centering
    \def\svgwidth{.8\textwidth}
    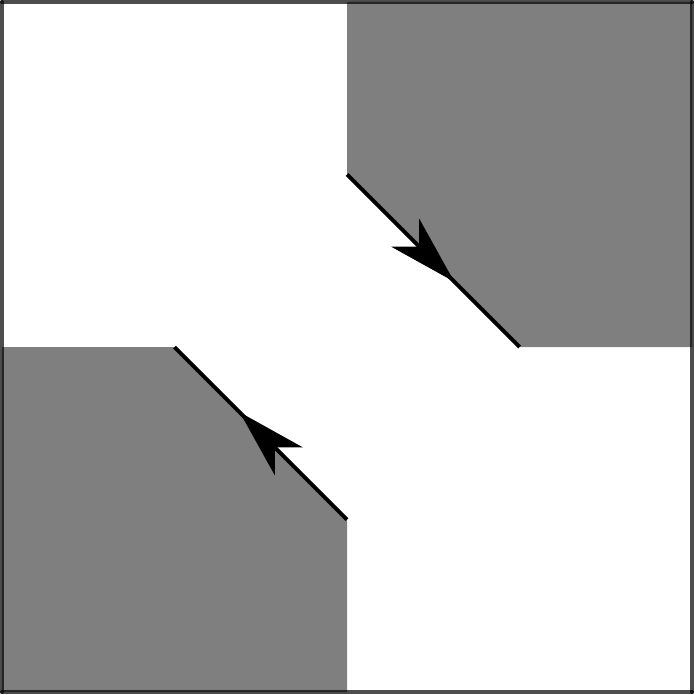
         \caption*{$\F\{\frac{d}{dx},  y\frac{d}{dy}\}$}
         \label{fig:blup1}
     \end{subfigure}
     \hfill
     \begin{subfigure}[b]{0.16\textwidth}
         \centering
    \def\svgwidth{.8\textwidth}
    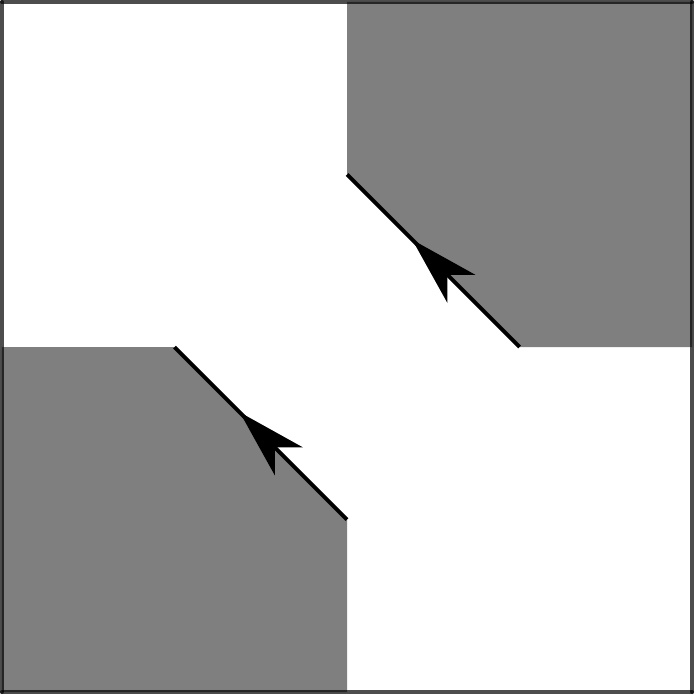
         \caption*{$\F\{\frac{d}{dx},   y^2 \frac{d}{dy}\}$}
         \label{fig:blup2}
     \end{subfigure}
               \hfill
     \begin{subfigure}[b]{0.24\textwidth}
         \centering

    \def\svgwidth{.5333\textwidth}
    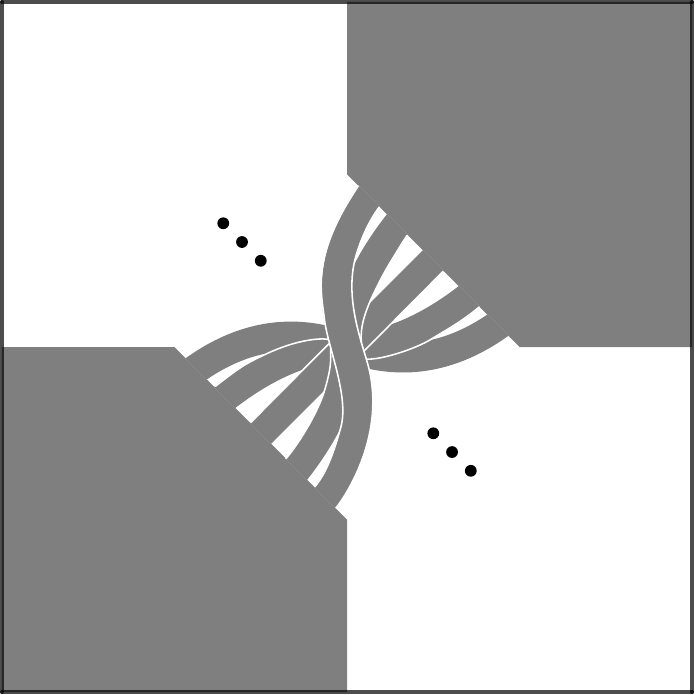
         \caption*{$\F\{ \frac{d}{dx}+y\frac{d}{dy},  y^2 \frac{d}{dy}\}$}
         \label{fig:blup3}
     \end{subfigure}
     \hfill
     \begin{subfigure}[b]{0.24\textwidth}
         \centering
    \def\svgwidth{.5333\textwidth}
    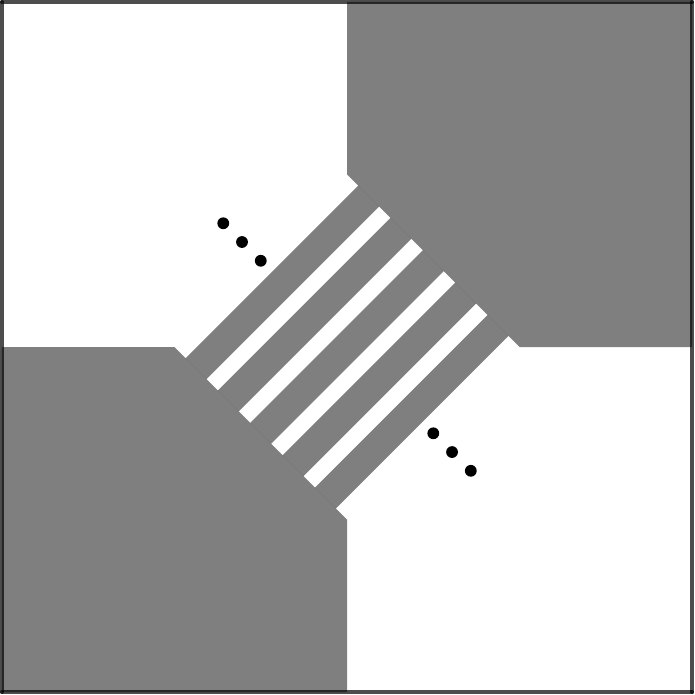
         \caption*{$\F\{\frac{d}{dx} + y^2 \frac{d}{dy}, y^4 \frac{d}{dy}\}$}
         \label{fig:blup4}
     \end{subfigure}
               \hfill \text{ }
        \caption{Holonomy groupoids of several  foliations of $S^1 \times \R$ whose leaves are  $S^1 \times \R_+$, $S^1 \times \{0\}$ and $S^1 \times \R_-$, restricted to   $T=\{0\} \times \R$.}
        \label{fig:strangeblups}
\end{figure}

The following family of one-dimensional foliations  appeared in \cite{AS[2007]}, Example~1.3~(3) and  received a detailed analysis in \cite{Francis[1dim]}. Discounting examples whose construction involves the use of bump functions, these are all of the singular foliations of $\R$ having leaves $\R_-$, $\R_+$ and $\{0\}$. 

\begin{defn}\label{FkRdef}
Let the coordinate function of $\R$ be $y$. For each positive integer $k$, we denote by $\F^k_\R$ the singular foliation of $\R$ singly-generated by $y^k \frac{\partial}{\partial y}$. That is, $\F^k_\R$ is all compactly supported vector fields on $\R$ which vanish to $k$th order at the origin.
\end{defn}

In this article, we analyze a class of   singular foliations which we call \emph{transverse  order $k$ foliations} (Definition~\ref{def:transordk}) which are the generalizations to higher codimension of the above foliations $\F^k_\R$. Suppose $\F$ is a foliation of a connected manifold $M$ whose leaves consist of a single codimension-1 submanifold $L \subset M$ together with the components of $M \setminus L$. The total number of leaves is therefore either two or three. Thanks to a splitting principle for singular foliations  (\cite{AS[2007]} Proposition~1.12, \cite{AZ[2014]} Proposition~1.2), the local structure of such a foliation is completely determined by a  foliation of $\R$ modelling   the transverse structure of the foliation near the leaf $L$. If this transverse foliation is  $\F^k_\R$, we say that $(M,\F)$ is  a foliation of \emph{transverse order $k$} (Definition~\ref{def:transordk}).  The foliations of $S^1 \times \R$ with singular leaf $S^1 \times \{0\}$ that appeared earlier in Table~\ref{table} were examples of transverse order $k$ foliations. The purpose of this article 
is to classify transverse order $k$ foliations and provide explicit descriptions their groupoids and algebras.

If  $M$ and $L$ are given, there is a unique   foliation $\F$ of transverse order $k=1$ whose singular leaf is $L$,  namely the collection of all compactly-supported vector fields which are tangent along $L$. However, when $k \geq 2$, the structure of transverse order $k$ foliations  becomes much more interesting. For this reason, we will generally assume $k$ is an integer $\geq 2$ in this article.

If $X_1,\ldots, X_n$  are smooth vector fields on a manifold $M$ such that $[X_i,X_j]$ is a $C^\infty(M)$-linear combination of $X_1,\ldots,X_n$  for all $i,j$, we use the notation 
\[ \F\{X_1,\ldots,X_n\} \coloneqq \mathrm{span}_{C_c^\infty(M)}\{X_1,\ldots,X_n\} \]
for the foliation they generate.

\begin{ex}\label{planeex}
Consider $\R^2$ with usual coordinates $(x,y)$. Then 
\begin{align*} \F\left\{ y^2 \tfrac{\partial}{\partial y}, \tfrac{\partial}{\partial x} \right\}  && \F  \left\{ y^2 \tfrac{\partial}{\partial y}, \tfrac{\partial}{\partial x} + y \tfrac{\partial}{\partial y} \right\} \end{align*}
are distinct transverse order $2$ foliations with singular leaf the $x$-axis.  Although these foliations are distinct, they are still isomorphic to each other; the  pushforward of the first foliation by the  diffeomorphism  $(x,y) \mapsto (x,e^xy)$ is the second foliation.
\end{ex}

The above example exposes the basic point that the  module of vector fields on $M$ which are ``tangent to $L$ to order $2$'' is not in fact well-defined   when  working up to smooth coordinate changes.  To get nonisomorphic examples of transverse order $2$ foliations with the same leaves, one needs the singular leaf to not be simply-connected.

\begin{ex}
In the previous example, consider $x$  as a $\Z$-periodic coordinate so that the space is the cylinder $S^1 \times \R$. Since the generating vector fields  are invariant under horizontal translation,  they also define  transverse order 2 foliations   on $S^1 \times \R$ with singular leaf the equator $S^1 \times \{0\}$. These foliations are  not isomorphic to each other and this  article  will provide  a framework explaining this nonisomorphism. Briefly, whereas for the first foliation the holonomy around the equator  is trivial, for the second, the holonomy, suitably interpreted,  is multiplication by $e$.
\end{ex}

Suppose $(M,\F)$ is a foliation of transverse order  $k \geq 2$ with singular leaf $L$. We give a brief overview of the notion of holonomy to be introduced. Let $c$ be a path in $L$ from a point $x$ to a point $x'$ and fix small, one-dimensional transversals $T$ and $T'$ at the endpoints. If $\F$ were a regular foliation of codimension-one, the classical notion of holonomy would assign to $c$  a diffeomorphism germ $T \to T'$ sending $x \mapsto x'$. This does not  occur in the  case at hand, but it turns out one does have a well-defined holonomy mapping at the level of $(k-1)$-jets (this is related to the picture of the holonomy groupoid obtained in \cite{AZ[2014]}). In particular, taking $x=x'$, $T=T'$ and choosing an  identification of $T$ with $\R$, one obtains from this assignment a homomorphism 
\begin{align}\label{hom} \gamma: \pi_1(L,x) \to J^{k-1}, \end{align}
where $J^r$ denotes the group of $r$-jets at $0$ of diffeomorphisms of $\R$ fixing $0$. Concretely, $J^r$ is the group of polynomials of the form $a_1y + \ldots + a_r y^r$ with $a_i \in \R$, $a_1 \neq 0$ under the operation ``compose and truncate''. There is a canonical quotient map $J^r \to J^{r-1}$ for all $r \geq 2$ whose kernel is the group $\R$, embedded in $J^r$ by way of $t \mapsto y + ty^r$. We obtain:
\begin{thm*}[Definition~\ref{invdef}]
The homomorphism \eqref{hom} is well-defined up to inner automorphisms of $J^{k-1}$ and gives rise to a holonomy invariant
\begin{align}\label{gamma} h(\F) \in [\pi_1(L),J^{k-1}] \end{align}
for the foliation. 
\end{thm*}
Here, if $A$ and $B$ are groups, we use $[A,B]$ to denote the quotient set of $\Hom(A,B)$ by the conjugation action of $B$, in the spirit of the similar notation frequently employed for homotopy classes of maps. Note it is not necessary to specify a basepoint for the fundamental group  in \eqref{gamma} because $L$ is connected, so its different fundamental groups are canonically isomorphic when working up to to inner automorphisms.

The   holonomy invariant \eqref{gamma}   is ``$L$-local'', in the sense that it only depends on the restriction of the foliation to a neighbourhood of  the singular leaf, and natural with respect to isomorphisms of transverse order $k$ foliations, in an appropriate way. 
In fact, it is a complete invariant for the structure of the foliation nearby to the singular leaf
\begin{thm*}[Theorem~\ref{holcompl}]
If $(M_i,\F_i)$ is a foliation of transverse order $k$ with singular leaf $L_i$ for $i=1,2$ and there is a diffeomorphism $\theta_0 : L_1 \to L_2$ carrying $h(\F_1)$ to $h(\F_2)$, then $\theta_0$ can be extended to a foliation-preserving diffeomorphism $\theta : U_1 \to U_2$, where $U_i$ is neighbourhood of $L_i$ in $M_i$.
\end{thm*}

Furthermore, the possible values of this holonomy invariant are exhausted.
\begin{thm*}[Theorem~\ref{holrange}]
 Given a connected manifold $L$ and a homomorphism $\pi_1(L) \to J^{k-1}$, there exists a transverse order $k$ singular foliation with singular leaf $L$ whose holonomy invariant is (the class of) the given homomorphism.
\end{thm*}

We formalize the above ideas about holonomy using principal bundles\footnote{Our principal bundles  will always be smooth, with structure group acting from the left.}. Given a transverse order $k\geq 2$ singular foliation $(M,\F)$ with singular leaf $L$, we construct a sequence of principal bundles $P^r(\F) \to L$, $r=1,2,\ldots,k$. The elements of $P^r(\F)$ are $r$-jets of certain submersions $M \to \R$  and may be thought of as dual versions of transversals. When $r\leq k-1$, the structure group of $P^r(\F)$ is $J^r_d$, the underlying discrete group of $J^r$. This captures the idea that the  $(k-1)$-jets of a transversals can be parallel transported along paths in $L$. This rigidity breaks down at $r=k$; the structure group of $P^k(\F)$ is $J^k_\R$, the one-dimensional Lie group structure on $J^k$  obtained by decomposing it into the fibers  of the  natural projection $J^k \to J^{k-1}$. The main applications of these principal bundles $P^r(\F)$ are as follows:

\begin{thm*}[Sections~\ref{principalbundles}, \ref{sec:gaugefull}, \ref{sec:gaugemin}]\Needspace{3\baselineskip} 
\leavevmode
\begin{enumerate}
\item The monodromy of the principal $J^{k-1}_d$-bundle $P^{k-1}(\F)$ is exactly the holonomy invariant $h(\F)$ of \eqref{gamma}.
\item The gauge groupoid of the principal $J^k_\R$-bundle $P^k(\F)$  reconstructs the holonomy groupoid of $\F$, restricted to $L$.
\item The monodromy of the principal $J^1_d$-bundle $P^1(\F)$ determines a flat connection on the conormal bundle of $L$ in $M$ (note that $J^1=\mathrm{GL}(1,\R)$) which it is appropriate to call the \emph{Bott connection}.
\end{enumerate}
\end{thm*}

In \cite{AS[2007]}, in addition to the (singular analog of the) usual holonomy groupoid $G(\F)$, the authors construct a \emph{full holonomy groupoid} $G_\textup{full}(\F)$. This is a ``big groupoid'' containing the usual (i.e. minimal) holonomy groupoid, as well as various intermediate groupoids.  It is helpful to work inside this big groupoid initially and later extract the usual one  as its  $s$-connected component.

The full holonomy groupoid of $\F^k_\R$ can be thought of as a smooth blowup of the singular equivalence relation $(\R\setminus\{0\})^2 \cup \{(0,0)\} \subset \R^2$  wherein the singular point $(0,0)$ is replaced by a copy of $J^k_\R$:
\[ G_\textup{full}(\F^k_\R) \cong (\R\setminus\{0\})^2 \cup J^k_\R. \] 
More generally, the full holonomy groupoid of any transversely order $k$ foliation $(M,\F)$ with singular leaf $L$ can be thought of as a smooth blowup of the singular equivalence relation $(M \setminus L)^2 \cup L^2 \subset M^2$ wherein the singular locus $L^2$ is replaced by a copy of the gauge groupoid of the principal $J^k_\R$-bundle $P^k(\F) \to L$:
\[ G_\textup{full}(\F) \cong (M\setminus L)^2 \cup \Gauge(P^k(\F)). \]
 The full holonomy groupoid of a transverse order $k$ foliation provides an interesting example of a topological space equipped with a smooth atlas  that is nearly, but not quite, a smooth manifold: 
 \begin{thm*}[Theorem~\ref{thm:topprop}]
 The topology of $G_\textup{full}(\F^k_\R)$  is Hausdorff, regular and separable, but not normal. 
 \end{thm*}
 There are known constructions for  ``manifolds'' with these properties (see \cite{Kunen-Vaughan}, Chapter~14), but it is interesting to encounter such a beast ``in the wild''.

The holonomy groupoid  of a transverse order $k\geq 2$ foliation $(M,\F)$ is the $s$-connected component of the full holonomy groupoid. We can give more concrete information: let $\gamma:\pi_1(L,x_0) \to J^{k-1}$ be any homomorphism representing the holonomy invariant $h(\F) \in [\pi_1(L), J^{k-1}]$.  Let $\Gamma$ be the range of $\gamma$, a countable subgroup of $J^{k-1}$, and let $\Gamma_\R$ be the preimage of $\Gamma$ by the natural projection $J^k \to J^{k-1}$. The relationships between these various groups are shown in the following diagram:
\[ \begin{tikzcd} 
\R \ar[r] & J^k_\R \ar[r] & J^{k-1}_d \\
\R \ar[u,equals] \ar[r] & \Gamma_\R \ar[u] \ar[r] & \Gamma \ar[u] \\
 & & \pi_1(L,x_0) \ar[u,"\gamma"] 
\end{tikzcd} \]
\begin{thm*}[Theorem~\ref{isotrope}]
The isotropy groups of  $G(\F)$ at points on the singular leaf are isomorphic to $\Gamma_\R$.
\end{thm*}
Meanwhile, the  restriction of $G(\F)$ to one of the open leaves (there are at most two) is a pair groupoid  so, applying standard results on groupoid C*-algebras, we  obtain information about the structure of the foliation C*-algebra. 
\begin{thm*}[Theorem~\ref{thm:extract}, Corollary~\ref{C*L}]
The foliation C*-algebra $C^*(\F)$ fits into an extension
\begin{align}\label{cstarext} 
0 \to I \to C^*(G(\F)) \to C^*(\Gamma_\R) \otimes \K \to 0.
\end{align}
where $\K$ is the C*-algebra of compact operators on a separable Hilbert and $I$ denotes either $\K$ or $\K \oplus \K$, according to whether $\F$ has one open leaf or two. 
\end{thm*}
It would be interesting to analyse the extension \eqref{cstarext}. The one-dimensional Lie group $\Gamma_\R$ is solvable, so this problem is likely to be tractable.

Since we are mainly concerned with what is happening near the singular leaf  $L$, it is often sufficient to consider the case where $M$ is the total space of a line bundle $\pi:E \to L$, with $L$ embedded in $E$ as the zero section. In this case, we already have a natural principal $J^r$-bundle $J^r(E,\R) \to L$, even without specifying a transverse order $k$-foliation. 
\begin{thm*}[Theorem~\ref{flatcorresp}]
There is a one-to-one correspondence between:
\begin{enumerate}
\item Flat connections on the principal $J^{k-1}$-bundle $J^{k-1}(E,\R)$
\item Singular foliations of transverse order $k$ on $E$ whose singular leaf is $L$ (embedded as the zero section)
\end{enumerate}
\end{thm*}
Once a flat connection on $J^{k-1}(\F)$  has been fixed, the resulting $J^{k-1}_d$-bundle structure on $J^{k-1}(E,\R)$ is canonically isomorphic to the principal $J^{k-1}_d$-bundle $P^{k-1}(\F)$. 

\vspace{\baselineskip}\noindent\textbf{Relation to other work.} What were introduced in \cite{Francis[PhD]}  as \emph{transverse order $k$ foliations} are quite similar to what Scott previously  introduced as \emph{$b^k$-manifolds} \cite{Scott}. A \emph{$b^k$-manifold} is an oriented, smooth manifold $M$ with an oriented hypersurface $L$ plus the data of the $(k-1)$-jet along $L$ of a smooth, positively-oriented defining function for $L$, i.e. an oriented submersion $p: \Omega \to \R$ with $p=f^{-1}(0)$, where  $\Omega$ is  a neighbourhood of $L$. Such a choice of $(k-1)$-jet determines an associated transverse order $k$ foliation by taking for $\F$ the collection of all compactly-supported vector fields $X$ such that $Xp$ vanishes to $k$th order along $L$, where $p$ is some defining function representing the given $(k-1)$-jet. However, this assignment is neither injective nor surjective when $k \geq 2$. Replacing $p$ by $2p$ changes its $(k-1)$-jet, but not its associated transverse order $k$ foliation. Furthermore, the existence of an inducing $(k-1)$-jet globally-defined on $L$ implies the triviality of the holonomy invariant studied in \cite{Francis[PhD]}, thus not  all transverse order $k$ foliations arise from Scott's setup. Indeed, $p$ is an $\F$-$\F^k_\R$-submersion (Definition~\ref{def:transordk}) defined on all of $L$ and yields a global section in the principal bundle $P^{k-1}(\F)$ (Section~\ref{principalbundles}) whose monodromy is the holonomy of $\F$.

In the recent preprint  \cite{Bischoff-},  Bischoff, del Pino and Witte introduce independently under the name \emph{$k$th order foliations} the  same objects  that \cite{Francis[PhD]} called \emph{transverse order $k$ foliations}, and also  consider  the case of arbitrary submanifolds in addition to that of hypersurfaces. The authors of \cite{Bischoff-} independently obtain a local classification theorem  essentially the same as one  in \cite{Francis[PhD]}. The other aspects of the theory of this class  of singular  foliations considered by the present author and by Bischoff, del Pino and Witte are rather complementary. For example, \cite{Francis[PhD]} emphasizes the holonomy groupoids and C*-algebras whereas \cite{Bischoff-} is more focused  on the Lie algebroid along the leaf, also considering  symplectic structures and going on to  develop and apply a cohomology theory  related to that in \cite{Scott}.

Even more recently, the preprint \cite{Fischer-} has appeared  in which  which Fischer and Laurent-Gengoux  carry the idea of classifying singular foliations with a given transverse model in a neighbourhood of given leaf by holonomy data much further than is done in either \cite{Francis[PhD]} or \cite{Bischoff-}. In \cite{Fischer-}, a complete classification is given, at least at the formal level, for any codimension and any transverse model.

\vspace{\baselineskip}\noindent\textbf{Structure of article.} We now briefly summarize the contents that follow this introduction.   Section~\ref{chap:prelim} is a rather long preliminary section that gathers for convenience various definitions and results, especially relating to Androulidakis and Skandalis's construction of the holonomy groupoid of a singular foliation. Furthermore, a description (Theorem~\ref{G/Nthm}) of the holonomy groupoid closely related to the description in  \cite{AZ[2014]} is obtained. In Section~\ref{jetsec}, we introduce the groups in which our holonomy invariant will take values. In Section~\ref{sec:fullFkR}, we study the full holonomy groupoid of the model foliation $\F^k_\R$ and its point-set topological properties (Theorem~\ref{thm:topprop}). In Section~\ref{sec:transordk}, we define precisely transverse order $k$ foliations (Definition~\ref{def:transordk}). In Section~\ref{localresultssection}, we determine which local transformations and submersions  are compatible with transverse order $k$ foliations (Theorem~\ref{local}). The rigidity in lower order Taylor expansions uncovered in this section is at the root of the holonomy invariants to be defined.  In Section~\ref{principalbundles}, we construct certain principal bundles and use their monodromy to precisely define our holonomy invariant (Definition~\ref{invdef}). In Section~\ref{sec:gaugefull}, we use the principal bundles of the preceding section to give a gauge groupoid description of the full holonomy groupoid of a transverse order $k$ foliation. In Section~\ref{sec:gaugemin} we describe  the holonomy groupoid (Theorem~\ref{thm:extract}) which sits inside the full holonomy groupoid as the \emph{$s$-connected component}.  In Section~\ref{sec:linbun}, we consider transverse order $k$ foliations in the particular case where the total space is a line bundle and the singular leaf is its zero section. In this case,  transverse order-$k$ foliations can be put into a natural correspondence with certain flat connections (Theorems~\ref{flatcorresp} and \ref{flatleafwise}). In Section~\ref{sec:complete}, we prove that our holonomy invariant is a complete invariant (Theorem~\ref{holcompl}). Finally, in Section~\ref{sec:ran}, we prove that our holonomy invariant takes on all of its possible values (Theorem~\ref{holrange}).

\section{Preliminaries}\label{chap:prelim}

In this lengthy section, we gather for ease of reference  various definitions and results which will be needed. We review the work  \cite{AS[2007]}, defining precisely what is meant by a (singular) foliation $(M,\F)$ and giving  the constructions of the holonomy groupoid $G(\F)$ as well as  the \emph{full holonomy groupoid} $G_\textup{full}(\F)$, a larger groupoid containing $G(\F)$. We also give a picture of the full holonomy groupoid as a groupoid of (equivalence classes of germs of) holonomy transformations that is  similar in spirit to the picture obtained in  \cite{AZ[2014]}, but slightly different because we do not make use of slices.   Proofs are omitted when they can be found in \cite{AS[2007]}. Readers already familiar with the literature on holonomy groupoids of singular foliations will most likely wish to skip this section and refer back to it whenever necessary.

\subsection{Manifolds}\label{smoothspaceterm}

In this article, a \textbf{smooth manifold} refers to a topological space equipped with a smooth atlas of some constant, finite dimension that is furthermore metrizable (or, equivalently, paracompact and Hausdorff). Many authors require their manifolds to be second-countable, but that assumption is not convenient here. In any event, the distinction is only relevant for highly disconnected manifolds; every smooth manifold in our sense is a (possibly uncountable) disjoint union of second-countable smooth manifolds.

We sometimes employ the term \textbf{smooth space} to refer to a topological space that is equipped  with a smooth atlas. Note that the same disclaimers which apply when one speaks of two manifolds being ``equal'' apply  to smooth spaces also. This is to say, an atlas for a smooth space is not really an innate part of its structure. Rather, one introduces the usual notion of equivalence of two atlases and works either with equivalence classes of atlases, or with the unique maximal atlas in each class.  Note the topology of a smooth space is uniquely determined by any atlas, so the former need not be specified in advance.

\subsection{Groupoids}

We shall tend to use calligraphic characters such as $\mathcal{G}$ to denote abstract groupoids (no topology) and reserve roman characters such as $G$ for topological groupoids.  For brevity, we often write  $\mathcal{G} \rightrightarrows X$ to indicate that $\mathcal{G}$ is an (abstract) groupoid with unit space $X$. We typically  denote the source and target projections of $\mathcal{G}$ by $s$ and $t$, respectively. Multiplication is performed from right to left so that, given $a,b \in \mathcal{G}$, the  product  $ab$   is defined if and only if $s(a)=t(b)$. The inversion map is denoted $\imath : \mathcal{G} \to \mathcal{G}$ or, frequently, just $a \mapsto a^{-1}$. We use (standard) notations such as $\mathcal{G}_x \coloneqq s^{-1}(x)$ and $\mathcal{G}^x  \coloneqq t^{-1}(x)$ for the source and target fibers.  We only speak of  morphisms between groupoids that have the same unit space and always require that the underlying map on the unit space  is the identity. Accordingly, the way in which  we understand quotients  of groupoids is constrained to the following:

\begin{defn}
A \textbf{normal subgroupoid}  of a groupoid  $\mathcal{G} \rightrightarrows X$ is a union $\mathcal{N} = \bigcup_{x \in M} \mathcal{N}_x$, where $\mathcal{N}_x$ is a subgroup of the isotropy group $\mathcal{G}_x^x$ such that, if $x,y \in X$, $a \in \mathcal{G}_x^y$, $b \in \mathcal{N}_x$, then $aba^{-1} \in \mathcal{N}_y$.
\end{defn}

\begin{lemma}
Let $\mathcal{N}$ be  normal subgroupoid of a groupoid $\mathcal{G}\rightrightarrows X$. Given $x,y \in X$ and $a,b \in \mathcal{G}_x^y$, put $a \approx b$ if and only if $a^{-1}b \in \mathcal{N}$. Then, $\approx$ is an equivalence relation and the groupoid operations of $\mathcal{G}$ descend in a well-defined way to give the quotient set $\mathcal{G}/\approx$ the structure of a groupoid on $X$. \qed
\end{lemma}

\begin{defn}
Given a normal subgroupoid $\mathcal{N}$ of a groupoid $\mathcal{G}$, the \textbf{quotient groupoid} $\mathcal{G}/\mathcal{N}$ is the groupoid $\mathcal{G}/\approx$ of the above lemma. 
\end{defn}

\subsection{Lie groupoids}

A \textbf{Lie groupoid} is a groupoid $G \rightrightarrows B$ where $G$ and $B$ are smooth manifolds, the source and target maps $s,t: G \to B$ are submersions and all structure maps are smooth. If $G$ is only a smooth space (see Section~\ref{smoothspaceterm}), we call $G \rightrightarrows B$ a \textbf{smooth groupoid}. One may refer to \cite{Mackenzie} for a detailed treatment of Lie groupoids.

\begin{ex}For any manifold $M$, the   \textbf{pair groupoid} refers to the Lie groupoid structure on  cartesian product $M \times M$ with  source projection  $\mathrm{pr}_2$, target projection $\mathrm{pr}_1$ and multiplication defined by $(x_3,x_2)(x_2,x_1)=(x_3,x_1)$ for all $x_1,x_2,x_3 \in M$.
\end{ex}

\begin{ex}
Given  a smooth action of a Lie group $H$ on a smooth manifold $M$,  the \textbf{transformation groupoid} $H \ltimes M$ is the Lie groupoid whose underlying manifold is  $H \times M$ with groupoid operations defined as follows:
\begin{align*}
\text{Source projection: } && &(h,x) \mapsto x \\
\text{Target projection: } &&  &(h,x) \mapsto h x \\
\text{Multiplication: } &&  &( h_2,h_1x) (h_1,x) = (h_2h_1,x)
\end{align*}
\end{ex}

We sometimes find it convenient to reverse the order of the factors; the transformation groupoid $M \rtimes H$ has underlying manifold $M \times H$ and its Lie groupoid structure is such that $(h,x) \mapsto (x,h): H \ltimes M \to M \rtimes H$ is a Lie groupoid isomorphism.

In this article, all principal bundles are assumed to be smooth, with structure group acting on the left. Every principal bundle determines a so-called gauge groupoid (also known as the Atiyah groupoid) as described below. This construction will play an important role in this article.

\begin{lemma}
Let $\pi:P \to B$ be a (smooth, left) principal $H$-bundle, where $H$ is a Lie group. Then, there is a unique Lie groupoid structure on the quotient manifold $(P \times P) / H$, where $H$ acts diagonally, whose operations are determined as follows:
\begin{align*}
\text{source projection:} && [q,p] \mapsto \pi(p) \\
\text{target projection:}&& [q,p] \mapsto \pi(q) \\
\text{multiplication:}  && [r,q][q,p] = [r,p].
\end{align*}
for all $p,q,r \in P$. Here $[q,p]$ denotes the class of $(q,p)$ in $(P \times P) / H$. \qed
\end{lemma}

\begin{defn}\label{gaugedef}
The \textbf{gauge groupoid} $\Gauge(P) \rightrightarrows B$ of a (smooth, left) principal bundle $P\to B$ is the Lie groupoid constructed in the above lemma.
\end{defn}

\begin{ex}\label{trivalgauge}
In the case of a trivial left $H$-bundle $B \times H$, it is easy to see that $(y,x,h) \mapsto [(y,1),(x,h)]$ defines an isomorphism $B \times B \times H \to \Gauge(B \times H)$. 
\end{ex}

Gauge groupoids are always transitive and, in fact, every transitive Lie groupoid $G \rightrightarrows B$ is isomorphic to a gauge groupoid; for any choice of $x \in B$, one has that $G^x$ is a principal $G^x_x$-bundle and $G \cong \Gauge(G^x)$.  
In \cite{Androulidakis[2004]}, the  correspondence between transitive Lie groupoids and principal bundles is extended in a way that takes extensions into account.

\subsection{Monodromy of flat bundles}\label{sec:monodromy}

Given two groups $A$ and $B$, we denote by $[A,B]$ the quotient   of the set $\Hom(A,B)$ by the conjugation action of $B$.

Let $M$ be a connected, smooth manifold and  let $\Gamma$ be a (discrete)  group. Given basepoints $x,y \in M$, there is a canonical bijection between the quotient sets $[\pi_1(M,x),\Gamma]$ and $[\pi_1(M,y),\Gamma]$. This is so because the isomorphisms  $\pi_1(M,x) \to \pi_1(M,y)$ determined by two different choices of paths from $x$ to $y$ only differ  by an inner automorphism of $\pi_1(M,x)$, and it follows that their induced  bijections $\mathrm{Hom}(\pi_1(M,x),\Gamma) \to \mathrm{Hom}(\pi_1(M,y),\Gamma)$   differ  by  an inner automorphism of $\Gamma$. It therefore makes sense to speak of the set $[\pi_1(M),G)]$ without specifying a choice of basepoint. In a similar vein, we have the following

\begin{propn}\label{[pi,Gamma]induce}
If  $\theta:M_1 \to M_2$ is a diffeomorphism of connected manifolds, then pushing forward loops  by $\theta$  determines a well-defined bijection
\[
\pushQED{\qed} 
\theta_* : [\pi_1(M_1),\Gamma] \to [\pi_1(M_2),\Gamma].\qedhere
\popQED
\]    
\end{propn}

In fact, the procedure of the above proposition makes 
\[ M \mapsto [\pi_1(M),\Gamma] \]
into a functor from the category of connected, smooth manifolds and diffeomorphisms to the category of sets and bijections.

Let $\pi:Q\to M$ be a smooth principal bundle with connected base manifold $M$ and discrete structure group $\Gamma$.\footnote{This is the same thing as as a normal covering space whose group of deck transformations has been identified with $\Gamma$.} Fix a point $q_0 \in Q$ and put $x_0 \coloneqq \pi(q_0)$. One may then define a group homomorphism \[ \gamma : \pi_1(M,x_0)\to\Gamma \] 
in the following way: given any  loop $c:[0,1]\to M$ based at $M$, let $\widetilde c:[0,1] \to Q$ be the unique lift of $c$ with $c(0)=q_0$ and define $\gamma([c])$ by: 
\[ \widetilde c(1) =\gamma([c]) \cdot q_0. \]
One may check this gives a well-defined homomorphism $\gamma$, moreover, that the class in $[\pi_1(M),\Gamma]$ of this homomorphism does not depend on the chosen point $q_0\in M$. For example, if $h \in \Gamma$ and $q'_0 = h \cdot q_0$, then the monodromy homomorphism $\gamma' : \pi_1(M,x_0)\to \Gamma$ determined by $q_0'$ satisfies $\gamma' = \mathrm{Ad}_h \circ \gamma$. It therefore makes sense to define:

\begin{defn}\label{def:monodromy}
Let $Q$ be a smooth principal bundle with connected base manifold $M$ and discrete structure group $\Gamma$. The \textbf{monodromy invariant} of $Q$ is the element $h(Q) \in [\pi_1(M),\Gamma]$ represented by the homomorphism $\pi_1(M,x_0)\to \Gamma$ constructed  above.
\end{defn}

It is a standard result that principal bundles with discrete structure group (or equivalently flat bundles)  are completely classified by their monodromy invariants. To be more precise, the following result holds:

\begin{thm}\label{thm:monodromy}
For $i=1,2$, let $M_i$ be a connected, smooth manifold and let $Q_i \to M_i$  be a smooth left principal bundle with discrete structure group $\Gamma$.
\begin{enumerate}
\item If $\theta:Q_1\to Q_2$ is a $\Gamma$-bundle isomorphism and $\theta_0 : M_1 \to M_2$ is the underlying map of the base, then $(\theta_0)_*(h(Q_1)) = h(Q_2)$. \item Conversely, if there exists a diffeomorphism $\theta_0:M_1\to M_2$ such that $(\theta_0)_*(h(Q_1)) = h(Q_2)$, then there exists a $\Gamma$-bundle isomorphism $\theta : Q_1 \to Q_2$  covering $\theta_0$.\qed
\end{enumerate}
\end{thm}

\subsection{Modules of  smooth sections}

Throughout the following,  $E$ is a smooth vector bundle over a smooth manifold $M$ and $\F$ is a $C^\infty(M)$-submodule of $C_c^\infty(M;E)$, the  smooth, compactly-supported sections of $E$.

\begin{defn}\label{fibersection}
The \textbf{fiber} of a submodule $\F \subset C_c^\infty(M;E)$ at $x \in M$ is the vector  space $A_x \F  \coloneqq \F/I_x \F$, where $I_x \subset C^\infty(M)$ denotes the ideal of functions vanishing at the point $x$. We denote the quotient map $\F \to A_x \F $ by $X \mapsto [X]_x$.
\end{defn}
From this definition, we have the relation
\begin{align*}
[fX]_x = f(x)[X]_x && f \in C^\infty(M), X \in \F.
\end{align*}
If $X \in \F$ has $X \in I_x \F$ for all $x \in M$, then $X=0$, so one may legitimately regard  $\F$ as a module of    sections of the bundle
\[ A\F \coloneqq \bigsqcup_{x \in M}A_x\F. \] 
However, it should be noted that $A\F$ need not have the structure of a smooth vector bundle and, indeed, the dimensions of its fibers may vary from point to point. Note also that, for each $x \in M$, the evaluation map $\F \to E_x$ contains $I_x \F$ in its kernel and therefore descends to a well-defined map $A_x \F  \to E_x$  on the fiber. Thus, the ``singular bundle'' $A\F$ is equipped with a ``singular bundle map'' $A\F \to E$. Furthermore, this bundle map sends $\F$, viewed as a module of sections of $A\F$, identically onto $\F$, viewed as a module of sections of $E$.

\begin{defn}A submodule $\F \subset C_c^\infty(M;E)$ is \textbf{finitely-generated} if there exist sections  $X_1,\ldots,X_n \in C^\infty(M;E)$ such that every $X \in \F$ has a representation $X = f_1X_1 + \ldots + f_n X_n$ with $f_i \in C_c^\infty(M)$ and \textbf{free of rank $\mathbf{n}$} if the latter representations are also unique.
\end{defn}

\begin{rmk}
Note that, in the above definition, the generators are not required to be compactly-supported. So, for example, we consider $\mathfrak{X}_c(\R^n) = C_c^\infty(\R^n;T\R^n)$, the $C^\infty(\R^n)$-module of compactly-supported vector fields on $\R^n$,  to be freely-generated by $\frac{\partial}{\partial x_1},\ldots, \frac{\partial}{\partial x_n}$, even though $\frac{\partial}{\partial x_i} \notin \mathfrak{X}_c(\R^n)$. 
\end{rmk}

\begin{defn}
Given an open set $U \subset M$, the  \textbf{restriction} of a submodule $\F \subset C_c^\infty(M;E)$ to $U$ is defined to be the $C^\infty(U)$-module $\F_U \subset C_c^\infty(U;E_U)$ given by $\F_U \coloneqq C_c^\infty(U) \F$.  We say that $\F$ is \textbf{locally finitely-generated} (resp. \textbf{locally free of rank $\mathbf{n}$}) if each point of $M$ belongs to some open set $U$ such that $\F_U$ is finitely-generated (resp. free of rank $n$).
\end{defn}

If $\F$ is locally finitely-generated, then each fiber $A_x \F $ is a finite-dimensional vector space. Furthermore,  the dimension of $A_x \F $ equals the minimum number of generators required for $\F_U$, when $U$ is any sufficiently small neighbourhood of $x$ (\cite{AS[2007]}, Proposition~1.5).

Let us now say a bit more about the locally free case.

\begin{propn}\label{locallyfree}
Let $\F$ be a locally finitely-generated $C^\infty(M)$-submodule of $C_c^\infty(M;E)$. Then, the following are equivalent:
\begin{enumerate}
\item $\dim(A_x\F) = k$ for all $x \in M$. 
\item $\F$ is locally free of rank $k$. 
\item  There exists a $k$-dimensional smooth vector bundle $A \to M$ and a vector bundle map $A \to E$ which is injective   over a dense subset of $M$ such that the image of the induced map $C_c^\infty(M;A) \to C_c^\infty(M;E)$ is $\F$. 
\end{enumerate}
Moreover, when these equivalent conditions hold, we may take $A = A\F$, equipped with the unique smooth structure for which $\F$, realized as a module of sections of $A\F$, coincides with $C_c^\infty(M;A\F)$. 
\end{propn}
\begin{proof}
If (2) holds, then (a version of) the Serre-Swan theorem gives that $A\F$ is a vector bundle with respect to the unique smooth structure for which $\F$, realized as a module of sections of $A\F$, coincides with $C_c^\infty(M;A\F)$. From this, it is simple to deduce (1), (3) and the ``moreover'' statement.

Suppose (3) holds. Note the  ``almost injectivity'' assumption on the bundle map $A \to E$ is equivalent to injectivity of the induced map   $C_c^\infty(M;A) \to C_c^\infty(M;E)$. Therefore, $\F$ is isomorphic to $C_c^\infty(M;A)$ as a $C^\infty(M)$-module and (2) follows. 

Finally, assume (1) holds. Let $U \subset M$ be  open and $X_1,\ldots, X_k \in C_c^\infty(U;E_U)$ be generators for $\F_U$. Suppose $f_1,\ldots,f_k \in C_c^\infty(U)$ have $f_1 X_1 +\ldots +f_k X_k = 0$. Then, for any $x \in U$, we get $f_1(x) [X_1]_x + \ldots f_k(x) [X_k]_x$. Since the $k$ vectors $[X_i]_x$ span the $k$-dimensional vector space $A_x \F $, they are a basis, and we obtain $f_1(x) = \ldots = f_k(x) = 0$. This shows that (2) holds. 
\end{proof}

\subsection{Foliations}

In this section, we precisely define what will be meant  by the word ``foliation'' and discuss related notions and constructions.

 \begin{defn}[\cite{AS[2007]}, Definition~1.1]
 A \textbf{foliation} $\F$ of a smooth manifold $M$ is a locally finitely-generated $C^\infty(M)$-module  of compactly-supported vector fields on $M$ that is furthermore stable under taking Lie brackets. 
 \end{defn}

 The choice to work with compactly-supported vector fields, though not totally essential, is convenient in several ways. Flows  of compactly-supported vector fields are automatically complete. Furthermore,  a compactly-supported vector field defined on an open subset can always be extended by zero to the whole space. Alternative approaches include  dropping the compact-support assumption altogether or  working  with the sheaf of  locally-defined vector fields. The  approach via compactly-supported vector fields is something of a  compromise between a fully global and fully local approach.

 \begin{defn}[\cite{AS[2007]}, Definition~1.7]\label{leafdef}
 A \textbf{leaf} of a foliation $(M,\F)$ is an orbit of $\exp(\F)$, the group of diffeomorphisms of $M$ generated by $\exp(X)$, $X \in \F$. 
 \end{defn}

By work of Stefan and Sussmann (\cite{Stefan}, \cite{Sussmann}), the leaves of a foliation $(M,\F)$ constitute a partition of $M$ into immersed submanifolds. See Section~1.3 of \cite{AS[2007]} for more detailed information on the leafwise smooth structure.\footnote{There is an unimportant  error in Remark~1.15~(1) of \cite{AS[2007]}. Let $N = \R$  and let $M =\R$ with the singular foliation singly-generated by $x \frac{d}{dx}$. Define $f : N \to M$ by $f(x) = x^3$. Then $(df_x)(T_xN) \subset F_{f(x)}$ is satisfied, but $f$ is not leafwise in the sense of \cite{AS[2007]}.}

\begin{ex}
If $A$ is any Lie algebroid  over a smooth manifold $M$, then the image of the map $C_c^\infty(M;A) \to \mathfrak{X}_c(M)$ induced by the anchor map is a   foliation of $M$. Presently, it seems not to be known whether in fact \emph{all}  foliations can (perhaps only locally) be obtained in this way. It was noted in \cite{AZ[2013]}, Proposition~1.3 that one can construct   foliations $(M,\F)$, with $M$ noncompact, such that $\sup_{x \in M} \dim(A_x \F ) = \infty$. Obviously such a foliation cannot be induced by a single Lie algebroid on $M$. However, it appears to not be known whether a   foliation whose fibers are bounded in dimension (which always happens  if $M$ is compact) must be induced by a Lie algebroid, nor does the local version of this question seem to be settled. The article \cite{LGLS} contains some partial work on this problem; see Proposition~4.33 therein.
\end{ex}

\begin{ex}
Specializing the above example, any Lie groupoid  $G\rightrightarrows M$  induces a foliation by way of its Lie algebroid. If $G$ is $s$-connected, the leaves of the foliation induced by $G$ are exactly the orbits of $G$.
\end{ex}

\begin{defn}
Let $(M,\F)$ be a foliation and let $x \in M$.
\begin{itemize}
\item The \textbf{fiber} of $\F$ at $x$ is the finite-dimensional vector space $A_x\F \coloneqq \F/I_x\F$ (this is a particular case of  Definition~\ref{fibersection}). 
\item The \textbf{tangent space} of $\F$ at $x$ is the finite dimensional vector space $T_x\F \coloneqq \{ X(x) : X \in \F\}$. 
\item The \textbf{isotropy Lie algebra} of $\F$ at $x$ is the kernel $\mathfrak{g}_x\F$ of the surjective linear map $A_x\F \to T_x\F$ that descends from the evaluation map $\F \to T_x\F$. 
\end{itemize}
\end{defn}
Consequent to these definitions, for each point $x \in M$, there is an exact sequence:
\[ 0 \to \mathfrak{g}_x\F \to A_x\F \to T_x\F \to 0. \] 
As the notation and terminology would suggest, $\mathfrak{g}_x$ is a Lie algebra. The bracket on $\F$ descends to a well-defined bracket on $\mathfrak{g}_x\F$.

\subsection{Regular and almost regular foliations}

\begin{defn}
A foliation $(M,\F)$ is called \textbf{regular} if the dimensions of its tangents spaces $T_x\F$, $x \in M$ are constant. A foliation which is not regular is said to be \textbf{singular}. A foliation is called \textbf{almost regular} if the dimensions of its fibers $A_x\F$, $x \in M$ are constant.
\end{defn}

A foliation is regular if and only if its leaves all have the same dimension. For a regular foliation, the tangent spaces $T_x\F$, $x \in M$ form a subbundle of $TM$ and $\F$ is equal to the compactly-supported sections of this subbundle. See  \cite{AS[2007]}, Example~1.3~(2) for further details.

Every regular foliation is almost regular. As explained in Proposition~\ref{locallyfree}, if $\F$ is almost regular,  $A\F =\bigsqcup_{x \in M} A_x\F$ is  a  vector bundle. Indeed, transferring the bracket of $\F$ to $C_c^\infty(M;A\F)$ makes $A\F$ into Lie algebroid whose anchor map is moreover injective on a dense open subset of $M$. In fact, one may equivalently define  almost regular foliations  as precisely the ones arising from  a Lie algebroid with an almost injective anchor map. See also the discussion in \cite{AS[2007]},  Section~3.2.

\subsection{Pullbacks and automorphisms}

Foliations can be pulled back by submersions (or, more generally, by maps satisfying an appropriate transversality assumption; see \cite{AS[2007]}, Definition~1.9).

\begin{defn}If $(N,\mathcal{E})$ is a foliation and $p:M \to N$ is a submersion, the \textbf{pullback foliation} $p^{-1}(\mathcal{E})$ is the foliation of $M$ consisting of $C_c^\infty(M)$-linear combinations of vector fields on $M$ which are $p$-projectable and project to elements of $\mathcal{E}$. In particular, if $\iota$ is the inclusion of an open set $U$ into $M$, we write $\iota^{-1}(\F)=\F_U$ and call $\F_U$ the \textbf{restriction} of $\F$ to $U$. 
\end{defn}

It is clear that a foliation can be pushed forward or pulled back by a diffeomorphism, simply by pushing forward or pulling back its constituent vector fields. Indeed, this may be considered a special case of pullback by a submersion. We use the following terminology and notations.

\begin{defn}
Let $(M,\F)$ be a foliation.
\begin{itemize}
\item An \textbf{$\mathbf{\F}$-automorphism} is a diffeomorphism $\theta:M \to M$ satisfying $\theta_*(\F)=\F$. We denote the group of $\F$-automorphisms  by $\Aut(\F)$. 
\item A \textbf{local $\mathbf{\F}$-automorphism} is a diffeomorphism $\theta:U \to V$, where $U$ and $V$ are open subsets of $M$, satisfying $\theta_*(\F_U)=\F_V$. The collection of all local $\F$-automorphisms is a pseudogroup. 
\item We write $\GermAut(\F)$ for the groupoid over the base $M$  consisting of germs of local $\F$-automorphisms.
\end{itemize}
\end{defn}

 One has that $\exp(\F)$ (Definition~\ref{leafdef}) is a normal subgroup of $\Aut(\F)$ (see \cite{AS[2007]}, Proposition~1.6).

\subsection{Gluing foliations}

 Even though we never view \emph{individual} foliations as sheaves, in Section~\ref{sec:linbun} it will be useful for us to know that one can compare or construct foliations on the same manifold using a sheaf property.
 The proof, which we omit, is a routine verification using partitions of unity.

 \begin{propn}\label{foliatedgluing}
Let $M$ be a smooth manifold and let $(U_i)_{i \in I}$ be an open cover of $M$. 
 \begin{enumerate}
 \item If $\mathcal{E}$ and $\mathcal{F}$ are foliations of $M$ and $\mathcal{E}_{U_i} = \mathcal{F}_{U_i}$ for all $i \in I$, then $\mathcal{E} = \mathcal{F}$. 
 \item Suppose $\mathcal{F}_i$ is a foliation of $U_i$ for each $i \in I$. If $(\F_i)_{U_i\cap U_j} =  (\F_j)_{U_i\cap U_j}$ is satisfied for all $i,j \in I$, then there exists a (unique by (1)) foliation $\F$ of $M$ such that $\F_{U_i}=\F_i$ for all $i \in I$. \qed
 \end{enumerate} 
 \end{propn}

\subsection{Bisubmersions}

The following definition is the basic ingredient in   Androulidakis and Skandalis's construction of the holonomy groupoid.
\begin{defn}[\cite{AS[2007]}, Definition~2.1]
An \textbf{$\mathbf{\F}$-bisubmersion} of a foliation $(M,\F)$ is a triple $(W,t,s)$ where $W$ is a smooth manifold and $s,t : W \to M$ are submersions satisfying $s^{-1}(\F) = t^{-1}(\F) = C_c^\infty(W;\ker(ds)) + C_c^\infty(W;\ker(dt))$.
\end{defn}

We sometimes abuse notation and denote an $\F$-bisubmersion $(W,t,s)$ simply by $W$. It is easy to see that, if $U \subset W$ is open, then $(U,t|_U,s|_U)$ is also an  $\F$-submersion.

\begin{defn}
A \textbf{morphism} of $\F$-bisubmersions  $(W_1,t_1,s_1)$, $(W_2,t_2,s_2)$ is a smooth map $f:W_1 \to W_2$ such that $s_1 = s_2\circ f$ and $t_1=t_2\circ f$. A \textbf{local morphism} from $W_1$ to $W_2$ is a morphism from an open subset of  $W_1$ to $W_2$.
\end{defn}
It is clear that morphisms of bisubmersions can be composed and that the identity is always a morphism of bisubmersions. In \cite{AS[2007]}, Corollary~2.11(c), it is shown that, if there is a local morphism of bisubmersions $W_1 \to W_2$ sending $w_1 \mapsto w_2$, then there is also a morphism of bisubmersions $W_2 \to W_1$ sending $w_2 \mapsto w_1$. The following definition is therefore justified. 
\begin{defn}\label{simdef}
Let $(M,\F)$ be a foliation, and  $(W_i)_{i \in I}$ be a collection of $\F$-bisubmerions. Then, we denote by $\sim$ the equivalence relation on $\bigsqcup_{i \in I}W_i $ given by $W_i \ni w_i \sim w_j \in W_j$ if and only if there exists a local morphism from $W_i$ to $W_j$ sending $w_i$ to $w_j$. We denote the quotient map $\bigsqcup_{i \in I} W_i \to \left(\bigsqcup_{i \in I} W_i\right)/\sim$  by $Q = (Q_i)_{i \in I}$. 
\end{defn}

\begin{defn}
Let $(M,\F)$ be a foliation and let $\mathcal{U} = (U_i)_{i \in I}$ and $\mathcal{V} = (V_j)_{j \in J}$ be collections of $\F$-bisubmersions. Put $U = \bigsqcup_{i \in I} U_i$ and $V = \bigsqcup_{j \in J} V_j$. We say that $\mathcal{U}$ is \textbf{adapted} to $\mathcal{V}$ if every $u \in U$ is $\sim$ to some $v \in V$. We say that $\mathcal{U}$ and $\mathcal{V}$ are \textbf{equivalent} if they are adapted to each other.  
\end{defn}

The following simple proposition is helpful in clarifying certain issues relating to forming the quotient by $\sim$.
\begin{propn}\label{openmap}
The quotient map $Q$ of the above definition is an open map.
\end{propn}
\begin{proof}
Let $W \coloneqq  \bigsqcup_{i \in I}W_i$. We need to prove  that the $\sim$-saturation of any open set  $U \subset W$ is open. It suffices to consider the case where $U \subset W_i$ for some $i \in I$. To this end, suppose $w \in U$, $w' \in W_j$ for some $j \in I$ and $w \sim w'$. Therefore, there exists a local morphism $f$ from $W_j$ to $W_i$  with $f(w')=w$. Then, $f^{-1}(U)$ is a neighbourhood of $w'$ with the property that each point in $f^{-1}(U)$ is $\sim$-equivalent (by way of $f$)   to a point  in $U$.
\end{proof}

\begin{cor}
Let $\mathcal{U} = (U_i)_{i \in I}$ and $\mathcal{V} = (V_j)_{j \in J}$ be collections of bisubmersions of $(M,\F)$. Put $U = \bigsqcup_{i \in I} U_i$ and $V = \bigsqcup_{j \in J} V_j$. If $\mathcal{U}$ is adapted to $\mathcal{V}$, then the map $U/\sim \to V/ \sim$ sending $[u] \mapsto [v]$ whenever $u \sim v$ is an open embedding. If  $\mathcal{U}$ and $\mathcal{V}$ are equivalent, this map $U/ \sim \to V / \sim$ is a homeomorphism.
\end{cor}
\begin{proof}
Note that the restriction of an open mapping to an open set is an open mapping. In particular, this can be applied to $U$ or $V$ sitting in $U \sqcup V$. It follows that both $U/\sim$ and $V/\sim$ sit as open subsets in $(U \sqcup V)/\sim$. If $\mathcal{U}$ is adapted to $\mathcal{V}$, then $V$ meets every equivalence class in $U \sqcup V$, so that $V/ \sim$ is homeomorphic to $(U \sqcup V) / \sim$. If $\mathcal{U}$ and $\mathcal{V}$ are equivalent, then $U /\sim$ is homeomorphic to $(U \sqcup V)/\sim$ as well.
\end{proof}

\begin{rmk}
Note that, in general, the restriction of a quotient map to an open   set which meets every equivalence class is not a quotient map. For example, the map $[0,1] \to S^1 : t \mapsto e^{2 \pi it}$ is a quotient map, but its restriction to $[0,1)$ is not.
\end{rmk}

\subsection{Construction of the holonomy groupoid}\label{sec:grpconstr}

There are  natural notions of inverse and composition for bisubmersions. See \cite{AS[2007]}, Proposition~2.4.

\begin{defn}
Let $(M,\F)$ be a foliation.
\begin{itemize}
\item The \textbf{composition} of two $\F$-bisubmersions $W_1$, $W_2$ is the $\F$-bisubmersion $W_2 \circ W_1$  whose  underlying manifold is the fiber product $W_2  \mathbin{{}_{s_2}{\times}_{t_1}} 
W_1$ and whose source and target maps are given by $s(w_2, w_1) = s_1(w_1)$ and $t(w_2 , w_1)  =t_2(w_2)$, in an obvious notation.
\item The \textbf{inverse} of an $\F$-bisubmersion $W$ is the $\F$-bisubmersion $W^{-1}$ obtained by keeping the same underlying manifold, but interchanging the source and target maps. 
\end{itemize}
\end{defn}

\begin{defn}[\cite{AS[2007]}, Definition~3.1]
Let $(M,\F)$ be a foliation. A collection   $\mathcal{W} =(W_i)_{i \in I}$ of $\F$-bisubmersions  is called an \textbf{holonomy atlas}  provided that:
\begin{enumerate}[(i)]
\item $\bigcup_{i \in I} s_i(W_i) =M$.
\item If $W \in \mathcal{W}$ and $w \in W$, then there exists $W' \in \mathcal{W}$ and $w' \in W'$ such that $w \in W^{-1}$ is $\sim$ to $w' \in W'$.
\item If $W_1,W_2 \in \mathcal{W}$, $w_1 \in W_1$, $w_2 \in W_2$, and $s_1(w_1) = t_2(w_2)$, then exists $W' \in \mathcal{W}$ and $w' \in W'$ such that $(w_1 \circ w_2 )\in W_1 \circ W_2$ is $\sim$ to $w' \in W'$.
\end{enumerate}
\end{defn}

Items (ii) and (iii) amount to saying $\mathcal{W}$ is closed under inverse and composition, if we work up to $\sim$.

\begin{thm}[\cite{AS[2007]}, Proposition~3.2]\label{holgrpddef}
Suppose  $\mathcal{W} =(W_i)_{i \in I}$ is a holonomy atlas for a foliation $(M,\F)$. Let $G(\mathcal{W}) \coloneqq \left(\bigsqcup_{i\in I} W_i\right)/\sim$. Then there is a groupoid structure on $G(\mathcal{W})$ such that
\[ Q_{W_2}(w_2) Q_{W_1}(w_1) = Q_{W_2 \circ W_1}(w_2,w_1) \]
whenever $W_1$ and $W_2$ are bisubmersions adapted to $\mathcal{W}$ and $w_1 \in W_1$, $w_2 \in W_2$ are such that $s_2(w_2)=s_1(w_1)$. \qed
\end{thm}

Every foliation  $(M,\F)$ admits a path holonomy atlas $\mathcal{W}_\textup{path}$ and a full holonomy atlas $\mathcal{W}_\textup{full}$ such that, if $\mathcal{W}$ is any holonomy atlas for $(M,\F)$, then $\mathcal{W}_\textup{path}$ is adapted to $\mathcal{W}$  and $\mathcal{W}$ is adapted to $\mathcal{W}_\textup{full}$. 
\begin{defn}[\cite{AS[2007]}, Definition~3.5, Example~3.4~(1)]
Let $(M,\F)$ be a foliation with   path holonomy atlas  $\mathcal{W}_\textup{path}$ and  full holonomy atlas $\mathcal{W}_\textup{full}$. Then, the \textbf{path holonomy groupoid}, or simply the \textbf{holonomy groupoid}, of $\F$ is $G(\F) \coloneqq G(\mathcal{W}_\textup{path})$. Similarly, the \textbf{full holonomy groupoid} of $\F$ is $G_\textup{full}(\F) \coloneqq G(\mathcal{W}_\textup{full})$.
\end{defn}
By definition, given  any  holonomy atlas  $\mathcal{W}$ for $\F$, there are canonical open inclusions $G(\F) \subset G(\mathcal{W}) \subset G_\textup{full}(\F)$.

\subsection{Bisections}

\begin{defn}
A \textbf{bisection} of a bisubmersion $(W,t,s)$ of a foliation $(M,\F)$ is a locally closed submanifold $N \subset W$ such that the restrictions of $s$ and $t$ to $N$ are diffeomorphisms onto open subsets of $M$. 
\end{defn}

\begin{propn}
Suppose $N$ is a bisection of a bisubmersion $(W,t,s)$ of a foliation $(M,\F)$. Then, $t|_N \circ (s|_N)^{-1}$ is a local $\F$-automorphism
\end{propn}
\begin{proof}
See Proposition~2.9 of \cite{AS[2007]}.
\end{proof}

\begin{defn}
Suppose that $(W,t,s)$ is a bisubmersion of a foliation $(M,\F)$. We say that a local $\F$-automorphism $\theta$ is \textbf{carried} by $(W,t,s)$ at a point $w \in W$ if there is a bisection $N$ of $(W,t,s)$ with $w \in N$ such that $\theta$ has the same germ as $t|_N \circ (s|_N)^{-1}$ at $s(w)$.  
\end{defn}

\subsection{Holonomy transformation picture of full holonomy groupoid}

In \cite{AZ[2014]}, it is shown that the holonomy groupoid of \cite{AS[2007]} can be realized as a groupoid of (equivalence classes of)  holonomy transformations of a family of transversal slices. It is pointed out  in \cite{AZ[2014]}, Remark~2.10(c) that the use of slices is essential because of an issue which can arise from nonorientable leaves. In this section, we lay out a rather cheap way to realize the groupoid of \cite{AS[2007]} without introducing slices which, though not very different from the original description in terms of bisubmersions, still has some of the flavour of a description by holonomy transformations.

\begin{propn}\label{simtoapprox}
Let $W_1$, $W_2$ be bisubmersions of  $(M,\F)$ and fix $w_i \in W_i$. If $w_1 \sim w_2$, then the set of local automorphisms carried by $W_1$ at $w_1$ is  exactly the equal to the set of local automorphisms carried by $W_2$ at $w_2$. Conversely, if there exists a local automorphism carried at both $w_1$ and $w_2$, then $w_1 \sim w_2$. 
\end{propn}
\begin{proof}
Suppose $w_1 \sim w_2$ and let $f$ be a local morphism with $f(w_1)=w_2$. Let $N_1 \subset W_1$ be a bisection with $w_1 \in N_1$. Since $N_1$ is a section of $s_1$, we have $T_w W_1 = T_w N_1 \oplus \ker(ds_1)_w$ for all $w \in S_1$.  From  $s_1 = s_2 \circ f$, one may deduce that $T_wS_1 \cap \ker(df)_w = \{0\}$, and that $df_w(T_wS) \cap \ker(ds_2) = \{0\}$. Thus, $f|_{N_1}$ is an immersion which is transverse to $s_2$. In the same way, $f|_{N_1}$ is transverse to $r_2$. It follows that the image of a small neighbourhood of $w_1$ in $N_1$ is a bisection $N_2 \subset W_2$ with $w_2 \in N_2$. Clearly the local diffeomorphisms induced by $N_1$ and $N_2$ have the same germ at $s_1(w_1)  = s_2(w_2)$, and we obtain that every local automorphism carried at $w_1$ is also carried at $w_2$. Interchanging the roles of $w_1$ and $w_2$, we get that same local automorphisms are carried at the two points. The converse statement is exactly \cite{AS[2007]}, Corollary~2.11(b). 
\end{proof}

 By the above proposition, the following defines an equivalence relation.

\begin{defn}
Let $(M,\F)$ be a foliation. Let $\theta_1$ and $\theta_2$ be germs at $x \in M$ of local $\F$-automorphisms with $\theta_1(x)=\theta_2(x)$. We write $\theta_1 \approx \theta_2$ if there exists an $\F$-bisubmersion $W$ and a point $w \in W$ such that both $\theta_1$ and $\theta_2$ are carried by $W$ at $w$.
\end{defn}

\begin{defn}\label{nulldef}
Let $(M,\F)$ be a foliation.  A local $\F$-automorphism $\theta_0$ is \textbf{null} at $x \in M$ if there exists an $\F$-bisubmersion $W$ and a point $w \in W$ with $s(w)=x$ such that both $\theta_0$ and $\id_M$ are carried by $W$ at $w$. We denote the group of germs at $x$ of local $\F$-automorphisms which are null at $x$ by $\NullAut(\F)_x$  and put $\NullAut(\F) \coloneqq \bigsqcup_{x \in M} \NullAut(\F)_x$. 
\end{defn}

\begin{thm}\label{G/Nthm}
Let $(M,\F)$ be a foliation. Then $\NullAut(\F)$ is a normal subgroupoid of $\GermAut(\F)$ and there is an abstract groupoid isomorphism 
\[G_\textup{full}(\F) \to \GermAut(\F)/\NullAut(\F) \] 
such that, if $W=(W,t,s)$ is an $\F$-bisubmersion and $w \in W$, then $Q_W(w) \in G_\textup{full}(\F)$ is mapped to the germ at $s(w)$ of any local $\F$-automorphism that is   carried by $W$ at $w$. 
\end{thm}
\begin{proof}
That this map is a bijection follows from Proposition~\ref{simtoapprox}. Multiplication and inversion are preserved by \cite{AS[2007]}, Proposition~2.8.
\end{proof}

\subsection{Smoothness of the holonomy groupoid of an almost regular foliation}\label{sec:almostregsmooth}

The holonomy groupoid constructed in \cite{AS[2007]} can be quite poorly-behaved for general foliations. Almost regular foliations are precisely the foliations  whose holonomy groupoids are Lie groupoids. This case was previously treated by Debord in \cite{Debord[2001b]}, with a different approach than that of \cite{AS[2007]}. The following result follows from \cite{AS[2007]}, Section~3.2 and can also be obtained in a more direct manner by adapting the arguments in  \cite{Debord[2013]}. 

\begin{propn}
Let $M$ be an $n$-dimensional smooth manifold, and $\F$ an almost regular singular foliation of $M$ with constant fiber dimension $k$. Then there is a unique smooth structure on $G_\textup{full}(\F)$ such that, for any $\F$-bisubmersion $W$,  the map $Q_W : W \to G_\textup{full}(\F)$ is smooth. The groupoid operations are smooth with respect to this smooth structure. \qed
\end{propn}

 When $\F$ is almost regular, $\A(\F)$ is isomorphic to  $C_c^\infty(G(\F))$, the smooth convolution algebra of the  groupoid,   and $C^*(\F)$ is isomorphic to $C^*(G(\F))$, the C*-algebra of the groupoid. Here we are implicitly fixing a smooth Haar system on $G(\F)$ in order to make sense of convolution and bypass any discussion of densities.

\section{Groups of jets on the line}\label{jetsec}

In this section we introduce certain groups of jets of diffeomorphisms of the real line which will play an important role.  Let us briefly recall the concept of the jet of a smooth mapping. For more information, one may refer the exposition in \cite{Michor-}, Section~12.

\begin{defn}
Let $M$ be a smooth manifold and $f$ a smooth real-valued function on $M$. Given $x_0 \in M$ and $r$ a positive integer, we say that $f$ \textbf{vanishes to order $\mathbf{r}$} at $x_0$ if $f \in (I_{x_0})^r$, where $I_{x_0} \subset C^\infty(M,\R)$ denotes the ideal of functions which vanish at $x_0$. 
\end{defn}

\begin{lemma}
Let $M, N$ be smooth manifolds, let $r$ be a nonnegative integer and let  $x_0 \in M$, $y_0 \in N$. Choose a diffeomorphism $\phi=(\phi_1,\ldots,\phi_n)$  from an open neighbourhood $V \subset N$ of $y_0$ onto an open set in $\R^n$  and define an equivalence relation  $\sim_{r,x_0}$ on the set of smooth functions $M\to N$ that send $x_0 \mapsto y_0$ by $f \sim_{r,x_0} g$ if and only if $\phi_i \circ f - \phi_i \circ g$ vanishes to order $r+1$ at $x_0$ for $i=1,\ldots,n$. Then, the equivalence relation $\sim_{r,x_0}$ does not depend on the choice of chart $\phi$. \qed
\end{lemma}

\begin{defn}
Let $M, N$ be smooth manifolds and $f:M \to N$ a smooth function. Given $x_0 \in M$ and $r$ a nonnegative integer, the \textbf{$\mathbf{r}$-jet}  of $f$ at $x_0$ is the equivalence class $j^r_{x_0}(f)$ of $f$ under the relation $\sim_{r,x_0}$ of above lemma.
\end{defn}

If $M=\R^n$ and $N=\R$ in the above definition, then $f \sim_{r,x_0} g$ if and only if the $r$th order Taylor polynomials of $f$ and $g$ at $x_0$ are the same. For this reason,  it makes sense   to identify $j^r_{x_0}(f)$ with the $r$th order Taylor polynomial of $f$. We shall make such identifications without comment.

There is a well-defined composition operation on jets. If $f,f' : M_1 \to M_2$ have the same $r$-jet at $x_0 \in M$ and $g,g':  M_2\to M_3$ have the same $r$-jet at $y_0 \coloneqq f(x_0)=f'(x_0) \in M_2$, then $g \circ f$ and $g' \circ f'$ have the same $r$-jet at $x_0$.  It therefore makes sense to define $j^r_{y_0}(g) \circ j^r_{x_0}(f) \coloneqq j^r_{x_0}(g \circ f)$. When working in coordinates, this  operation on jets is the usual ``compose and truncate'' on  $r$th order Taylor polynomials.

The following groups of jets will play an important role for us.

\begin{defn}
For $r$ a positive integer, $J^r$ denotes the group of $r$-jets at $0$ of diffeomorphisms of $\R$ fixing $0$. 
\end{defn}
The group $J^r$ has a canonical $r$-dimensional Lie group structure coming from its identification  with the group of real polynomials of the form $a_1y+\ldots +a_ry^r$ where $a_1 \neq 0$ with respect to the  ``compose and truncate'' operation. 

For each $k \geq 2$, there is canonical exact sequence of Lie groups
$$
\begin{tikzcd}
0 \ar{r}  & \R \ar[r] & J^k \ar{r} & J^{k-1} \ar{r} & 0 
\end{tikzcd}$$
where the projection map $J^k \to J^{k-1}$ is given by deleting the order $k$ term and the inclusion map $\R \to J^k$ is defined by  $t \mapsto y + ty^k$. 

We  will frequently want to equip these jet groups with a nonstandard topology.

\begin{defn}\label{weirdgroup}
For $r \geq 1$, we write $J^r_d$ for $J^r$ considered as an (uncountable) discrete group. For $k \geq 2$ we write $J^k_\R$ for $J^k$ equipped with the  one-dimensional Lie group structure arising from its partition into the cosets of $\R$ in $J^k$ (of which there are uncountably many).
\end{defn}
We then have an extension of (non-second-countable) Lie groups
$$ 
\begin{tikzcd}
0 \ar{r}  & \R \ar[r] & J^k_\R  \ar{r} & J^{k-1}_d \ar{r} & 0,
\end{tikzcd}$$
where $\R$ has its standard smooth structure.

The following proposition is intended to show that, from the perspective of abstract group theory, these jet groups are quite tame.

\begin{propn}\label{solvable}
For every $k \geq 2$, the group $J^k$ is solvable. Indeed, we may express $J^k$ as the semidirect product of a nilpotent group by an abelian group.
\end{propn}
\begin{proof}
For every $k \geq 2$, there is an evident exact sequence:
\[ 0 \to J^{2,k} \to J^k \to J^1 \to 0 \]
where $J^{2,k} \coloneqq \{ y + a_2y^2 + \ldots + a_k y^k : a_i \in \R\}$. This sequence splits on the right via $ay \mapsto ay$, so $J^k$ is the semidirect product of $J^{2,k}$ by the abelian group $J^1$.  For $k=2$, we have $J^{2,k} \cong \R$. For $k \geq 3$, we have a central extension:
\[ 0 \to \R \to  J^{2,k} \to J^{2,k-1} \to 0. \] 
By induction, $J^{2,k}$ is nilpotent for all $k \geq 2$.
\end{proof}

The groups $J^k$ for  $k \geq 2$ are not themselves nilpotent.  In fact, the center of $J^k$ is trivial. It is interesting to notice that  $J^2$ is isomorphic to the ``ax+b group'' of affine bijections of the real line. An example of an isomorphism is  $ay+b \mapsto a^{-1}y+ba^{-2}y^2$. This is conceptually related to the fact that the inversion map  $y \mapsto 1/y$ conjugates $y^2 \frac{d}{dy}$ to $-\frac{d}{dy}$.

\begin{rmk}
It will later be relevant to take a countable subgroup $\Gamma \subset J^{k-1}_d$ and consider its preimage $\Gamma_\R \subset J^k_\R$ under the projection $J^k_\R \to J^{k-1}_d$. Because subgroups of solvable groups are solvable, the above proposition gives that the one-dimensional group $\Gamma_\R$ is second-countable and amenable. One then has that extensions of $C^*(\Gamma_\R)$ can be described in terms of K-theoretic data (by the universal coefficient theorem) and that the K-theory of $C^*(\Gamma_\R)$ can be geometrically computed (by the Baum-Connes conjecture). 
\end{rmk}

\section{The full holonomy groupoid of \texorpdfstring{$\F^k_\R$}{FkR}}\label{sec:fullFkR}

In this section, we describe the full holonomy groupoid (Definition~\ref{holgrpddef}) $G_\textup{full}(\F^k_\R)$  
of $\F^k_\R$ and discuss its point-set topological properties.  When $k=1$, it is simple to see that $G_\textup{full}(\F^1_\R) \cong  \operatorname{GL}(1,\R) \ltimes \R$, so we concentrate our discussion on the case $k \geq 2$ where the groupoid is larger.  The minimal holonomy groupoid of $\F^k_\R$ was discussed in \cite{Francis[1dim]}; there is a unique Lie groupoid isomorphism
\[ G(\F^k_\R) \cong \R \ltimes_\phi \R, \]
where $\phi$ denotes the flow of any complete vector field  $X$ generating $\F^k_\R$. It will be more convenient, however,  to replace  $\R \ltimes_\phi \R$ by an isomorphic bisubmersion with polynomial structure maps.

\begin{defn}\label{omegadef}
Let $\Omega=\{ (t,y) \in \R^2 : 1+ty^{k-1} > 0 \}$ and define $\sigma,\tau : \Omega \to \R$ by $\sigma(t,y) = y$ and $\tau(t,y)= y + ty^k$.
\end{defn}

Note the inequality $1+ty^{k-1}>0$ is  precisely the  condition guaranteeing $\sigma(t,y)$ and $\tau(t,y)=y+ty^k$ have the same sign. 

\begin{lemma}\label{omega}
The triple $\Omega=(\Omega, \tau,\sigma)$ is an $\F^k_\R$-bisubmersion adapted to the path holonomy atlas and the natural map $Q_\Omega: \Omega \to G(\F^k_\R)$ is a diffeomorphism.
\end{lemma}
\begin{proof}
Identify $G(\F^k_\R)$ with $\R \ltimes_\phi \R$, where $\phi$ is the flow of some complete vector field $X$ on $\R$ generating $\F^k_\R$. We have $X 
=f(y) y^k \frac{d}{dy}$ where $f$ is smooth and nonvanishing. From this, it follows that we can write 
\begin{align}\label{h} \phi_t(y) = y + 
h(t,y)y^k  && t,y \in \R,
\end{align}
where $h : \R^2 \to \R$ is smooth. We claim that
\[ (t,y) \mapsto (h(t,y),y) : \R \ltimes_\phi \R \to \Omega \]
defines an isomorphism of bisubmersions. When $y \neq 0$, note that $\{ t \in \R : 1+ty^{k-1} > 0\}$ equals the set of all $t \in \R$ such that $y+ty^k$ has the same sign as $y$. Since $\phi$ is free and transitive on the positive and negative half lines, it follows that $(t,y) \mapsto (h(t,y),y)$ is a bijection from $\R^2 \setminus (\R \times \{0\})$ to $\Omega \setminus (\R \times\{0\})$. 

Differentiating \eqref{h} with respect to $t$ and rearranging gives
\begin{align*}
\tfrac{\partial h}{\partial t}(t,y) = \left(\frac{ \phi_t(y)}{y}\right)^k  > 0 && t,y \in \R, y \neq 0.
\end{align*}
from which one may deduce that $(t,y) \mapsto (h(t,y),y)$ is a diffeomorphism from $\R^2 \setminus (\R \times \{0\}) \to \Omega \setminus (\R \times\{0\})$.

Finally, for the sake of simplicity, choose $X$ to  coincide with $y^k \frac{d}{dy}$ on a neighbourhood of $0$. Then, $\phi_t(y)=\frac{y}{\sqrt[k-1]{1-(k-1)ty^{k-1}}}$ holds on a neighbourhood of $\R \times \{0\}$. In particular, the   Taylor series of $\phi_t$ begins $\phi_t(y) \sim y + ty^k + \tfrac{1}{2} t^2 y^{2k-1} +  \ldots$ and we have $h(t,0) = t$ for all $t \in \R$. The rest follows.
\end{proof}

By Theorem~\ref{G/Nthm}, the full holonomy groupoid of a foliation $(M,\F)$  is isomorphic to the groupoid $\GermAut(\F)$  of germs of local $\F$-automorphisms modulo  the normal subgroupoid $\NullAut(\F)$ of germs of null automorphisms (Definition~\ref{nulldef}). Therefore, to determine $G_\textup{full}(\F^k_\R)$ as an abstract groupoid, we need only determine the automorphisms and null automorphisms of $\F^k_\R$.

\begin{lemma}\label{omegacarry}
Let $\theta$ be a diffeomorphism of $\R$  defined on a neighbourhood of $0$.
\begin{enumerate}
\item $\theta$ preserves $\F^k_\R$ if and only if $\theta(0)=0$. 
\item $\theta$ is  null at $0$ if and only if $j^k_0(\theta)=y$.
\end{enumerate}
\end{lemma}
\begin{proof}
If $f \in C_c^\infty(\R)$, we have $\theta_*(f \frac{d}{dy})= (f \circ \theta^{-1}) \theta_*(\frac{d}{dy})$. Since $f \mapsto \theta^{-1}$ preserves the ideal of functions that  vanish to order $k$ at $0$ and $\theta_*(\frac{d}{dy})$ is a positive, smooth function-multiple of $\frac{d}{dy}$, assertion (1) follows. If $j^k_0(\theta) =y$, we may write write $\theta(y) = y + f(y)y^k$ where $f$ is a smooth function with $f(0)=0$. Then, for  $\epsilon>0$ appropriately small,  $\{(f(y),y) : y \in (-\epsilon,\epsilon) \}$ is a bisection of $\Omega$ containing $(0,0)$ which induces $\theta$. Conversely, any bisection passing through $(0,0)$ is locally of the form $\{(f(y),y) : y \in (-\epsilon,\epsilon) \}$ for some smooth $f$, and (2) follows.
\end{proof}

\begin{propn}\label{abstractiso}
There is a unique isomorphism of abstract groupoids 
\[ G_\textup{full}(\F^k_\R) \to (\R\setminus\{0\})^2 \cup J^k\]
such that, if $(W,t,s)$ is any $\F^k_\R$-bisubmersion   and $w \in W$ has $s(w)=0$, then $Q_W(w) \mapsto j^k_0(\theta)$ where $\theta$ is any diffeomorphism of $\R$ carried by $W$ at $w$. 
\end{propn}
\begin{proof}
Clearly the restriction of $G_\textup{full}(\F^k_\R)$ to $\R\setminus\{0\}$ is uniquely isomorphic to the pair groupoid $(\R\setminus\{0\})^2$. By Lemma~\ref{omegacarry}, the group $\GermAut(\F^k_\R)_0$ of germs at $0$ of local $\F^k_\R$ automorphisms is the group of germs of diffeomorphisms $\theta$ of  $\R$ with $\theta(0)=0$ and the normal subgroup $\NullAut(\F^k_\R)_0$ of germs of null automorphisms at $0$ is the group of germs of diffeomorphisms $\theta$ of $\R$ whose $k$-jet at $0$ is $y$. Thus, by Theorem~\ref{G/Nthm},  $G_\textup{full}(\F^k_\R)_0 \cong \GermAut(\F^k_\R)_0/\NullAut(\F^k_\R)_0 = J^k$.
\end{proof}

\begin{rmk}The orbits of the minimal holonomy groupoid $G(\F)$ of a foliation $(M,\F)$ are exactly the leaves of $\F$. Similarly, if $\mathcal{W}$ is any holonomy atlas for $\F$, then the orbits of the holonomy groupoid $G(\mathcal{W})$ are unions of leaves related by $\mathcal{W}$. This explains why $G(\F^k_\R)$ is a blow up of $(\R_-)^2 \cup (\R_+)^2 \cup \{(0,0)\}$ and $G_\textup{full}(\F^k_\R)$ is a blowup of $(\R\setminus\{0\})^2 \cup \{(0,0)\}$.
\end{rmk}

It is quite easy to see that the above identification  of $G_\textup{full}(\F^k_\R)_{\R\setminus\{0\}}$ with $(\R\setminus\{0\})^2$ is also a diffeomorphism. However, we shall see that the isotropy group $G_\textup{full}(\F^k_\R)_0$ is in fact diffeomorphic to the one-dimensional Lie group $J^k_\R$, rather than the $k$-dimensional Lie group $J^k$.  In order to  illuminate  the topological structure of  $G_\textup{full}(\F^k_\R)$ near its isotropy group  $G_\textup{full}(\F^k_\R)_0$, we need to introduce  an explicit  full holonomy atlas for $\F^k_\R$.  We use freely  the results and terminology of Section~\ref{sec:grpconstr}.

\begin{defn}\label{omegatheta}
For each $\theta \in \Diff_0(\R)$, put $\Omega_\theta \coloneqq (\Omega,\theta \circ \tau, \sigma)$.
\end{defn}

By construction, $\Omega_\theta$ is an $\F^k_\R$-bisubmersion carrying $\theta$ at the point $(0,0)$.

\begin{propn}\label{omegaatlas}
\text{ }
\begin{enumerate}
\item $\{\Omega_\theta : \theta \in \mathrm{Diff}_0(\R)\}$ is  a full holonomy atlas for $\F^k_\R$. That is, any $\F^k_\R$-bisubmersion is adapted to this holonomy atlas (see Section~\ref{sec:grpconstr}).
\item For each $\theta \in \Diff_0(\R)$, the canonical map $Q_{\Omega_\theta} : \Omega_\theta \to G_\textup{full}(\F^k_\R)$ is a diffeomorphism onto its image.
\end{enumerate}
\end{propn}
\begin{proof}
If $\theta$ is orientation-preserving, then $\Omega_\theta$ restricted to  $\R\setminus\{0\}$ is isomorphic to $(\R_-)^2 \cup (\R_+)^2$. If $\theta$ is orientation-reversing, then $\theta$  restricted to  $\R\setminus\{0\}$ is isomorphic to $(\R_- \times\R_+) \cup  (\R_+\times \R_-)$. This implies that, for $y \in \R\setminus\{0\}$, every $\F^k_\R$-automorphism germ at $y$ is carried by some $\Omega_\theta$. By construction, if $\theta \in \Diff_0(\R)$, then $\theta$ is carried by $\Omega_\theta$ at $(0,0)$. We have shown that every local automorphism of $\F^k_\R$ is carried by some $\Omega_\theta$ so, by Proposition~\ref{simtoapprox}, the $\Omega_\theta$ form a full holonomy atlas.

For (2), note that the dimension of $\Omega_\theta$ equals the dimension of $\R$ plus the fiber dimension of $\F^k_\R$, so the map of $\Omega_\theta$ to $G_\textup{full}(\F^k_\R)$ is a local diffeomorphism (see Proposition~3.11(b) in \cite{AS[2007]}). Since the set of points in $\Omega_\theta$ with trivial isotropy is dense, every local morphism $\Omega_\theta \to \Omega_\theta$ is the identity and the map $\Omega_\theta \to G_\textup{full}(\F^k_\R)$ is injective.
\end{proof}

One rudimentary property of the topology of $G_\textup{full}(\F^k_\R)$ is that it has two connected components. The basic idea of the proof, which we omit,  already appeared in the first  paragraph of the proof of Proposition~\ref{omegaatlas}.

\begin{propn}\label{2comp}
The holonomy groupoid $G_\textup{full}(\F^k_\R)$ has two connected components\footnote{Or, equivalently, path components, since $G_\textup{full}(\F^k_\R)$ is locally Euclidean.}. Under the isomorphism of Proposition~\ref{abstractiso}, the components map to $J^k_+ \cup (\R_-)^2 \cup (\R_+)^2$ and $J^k_- \cup (\R_-\times\R_+)\cup(\R_+\times\R_-)$, where $J^k_+$ and $J^k_-$ denote the $k$-jets of the orientation-preserving and orientation-reversing diffeomorphisms, respectively. \qed
\end{propn}

Next we describe the smooth structure of the isotropy group $G_\textup{full}(\F^k_\R)$.

\begin{propn}\label{Jkdescrip}
Giving $J^k$ its one-dimensional Lie group structure $J^k_\R$ (Definition~\ref{weirdgroup}) makes the group isomorphism $G(\F^k_\R)_0 \to J^k_\R$ provided by  Proposition~\ref{abstractiso} into a Lie group isomorphism. 
\end{propn}
\begin{proof}
If $\theta \in \Diff_0(\R)$ and $t\in \R$, then $(\{t\}\times\R) \cap\Omega_\theta$ is a bisection of $\Omega_\theta$ carrying the local diffeomorphism $y \mapsto \theta(y+ty^k)$ at the point $(t,0)$. The composition
\[\Omega_\theta \overset{\Omega_\theta}{\to} G_\textup{full}(\F^k_\R) \to (\R\setminus\{0\})^2 \cup J^k\]
therefore sends
\[ \Omega_\theta \ni (t,0) \mapsto j^k_0(\theta) \circ (y+ty^k) \in J^k \]
so that $\Omega_\theta \cap (\R \times \{0\})$ is carried diffeomorphically onto the coset of $\R \subset J^k_\R$ containing $j^k_0(\theta)$.
\end{proof}

Proposition~\ref{Jkdescrip} shows that $G_\textup{full}(\F^k_\R)$ is necessarily a somewhat strange space; it is a blowup of the singular equivalence relation $(\R\setminus\{0\})^2 \cup\{(0,0)\}\subset \R^2$ for which the singular point at the origin is replaced by continuum-may copies of the real line. This already shows it is not a manifold (even in our relaxed sense of the word, see Subsection~\ref{smoothspaceterm}) because it has only two components (Proposition~\ref{2comp}), but is not second countable.

The remainder of this section is devoted to investigating the separation properties of $G_\textup{full}(\F^k_\R)$. The basic observation is as follows: if $\theta_1, \theta_2 \in \Diff_0(\R)$ have different Taylor series at $0$, then the intersection of their graphs with some  punctured neighbourhood of the origin are disjoint and, therefore, can   be separated open subsets of the punctured plane. The following lemma shows that, by choosing  neighbourhoods carefully, we can separate the different cosets  of $\R$ in $J^k_\R$ from each other by open sets in  $G_\textup{full}(\F^k_\R)$.

\begin{figure}[ht]
     \centering
    \def\svgwidth{.3\textwidth}
    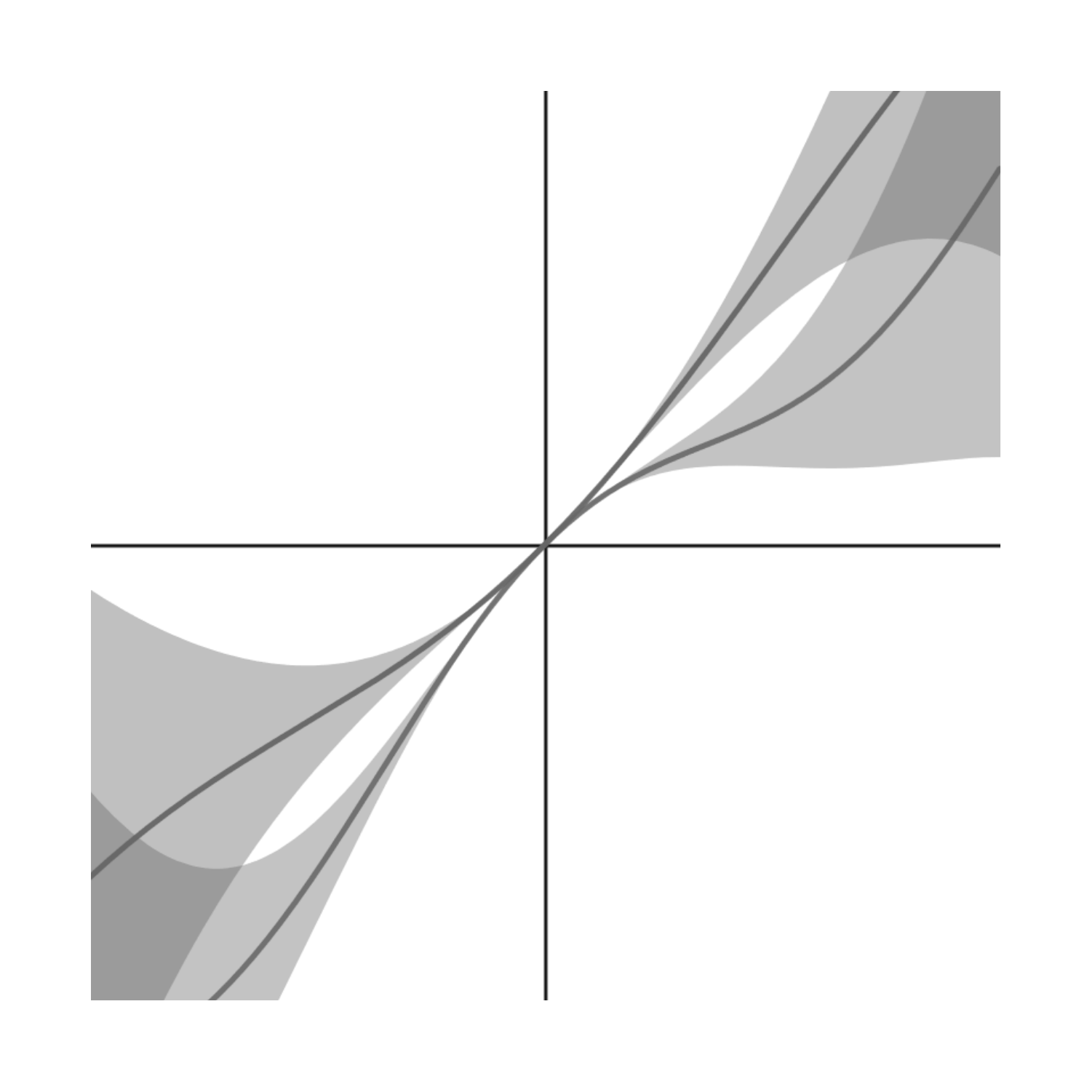
    \caption{The geometry  behind Lemma~\ref{graphnbhd}.}
\end{figure}

\begin{lemma}\label{graphnbhd}
For any $\theta \in \Diff_0(\R)$, there is a unique smooth function $f_\theta$ on $\R^2$ such that $\theta(y+ty^k) = \theta(y) + f_\theta(t,y)y^k$ for all $(t,y) \in \R^2$. Define
\[ U_\theta \coloneqq \{ (t,y) \in \Omega_\theta : |y|^{1/2} |f_\theta(t,y)| < 1 \}, \]
so that $U_\theta$ is an open subset of $\Omega_\theta$ containing $\R \times \{0\}$. Then, given   $\theta_1, \theta_2 \in \Diff_0(\R)$ with different $(k-1)$-jets at $0$, there exists $\epsilon > 0$ such that $U_{\theta_1}$ and $U_{\theta_2} \cap (\R \times (-\epsilon,\epsilon))$, viewed as $\F^k_\R$-bisubmersions, have disjoint images in $G_\textup{full}(\F^k_\R)$. 
\end{lemma}
\begin{proof}
We identify $G_\textup{full}(\F^k_\R)$ with $(\R\setminus\{0\})^2\cup J^k_\R$ without comment. It is straightforward to deduce the existence of $f_\theta$  from Taylor's theorem. 

Suppose $\theta_1,\theta_2 \in \Diff_0(\R)$ and $j^{k-1}_0(\theta_1) \neq j^{k-1}_0(\theta_2)$. The images of $U_{\theta_1} \cap (\R \times \{0\})$ and $U_{\theta_2} \cap (\R \times \{0\})$ are the (disjoint) cosets of $\R$  in $J^k_\R$ which contain $j^k_0(\theta_1)$ and $j^k_0(\theta_2)$, respectively. We need therefore only need to check that there is some $\epsilon > 0$ such that, whenever $(t_1,y) \in U_{\theta_1}$, $(t_2,y)\in U_{\theta_2}$ and $0<|y|<\epsilon$, we have $\theta_1(y+t_1y^k)\neq\theta_2(y+t_2y^k)$. For any $(t_1,y) \in U_{\theta_1}$ and $(t_2,y)\in U_{\theta_2}$, we have
\begin{align*} | \theta_1(y+t_1y^k) - \theta_2(y+t_2y^k)| &\geq |\theta_1(y)-\theta_2(y)| - |y^k f_{\theta_1}(t_1,y)|-|y^k f_{\theta_2}(t_2,y) \\ 
&\geq |\theta_1(y)-\theta_2(y)|  -  2|y|^{k-\frac{1}{2}}
\end{align*}
Because the $(k-1)$-jet of $\theta_1-\theta_2$ at $0$ is nonzero, $2|y|^{k-\frac{1}{2}}$ vanishes more quickly than $|\theta_1(y)-\theta_2(y)|$ at $y=0$. It follows that there exists an $\epsilon >0$ such that, for $0<|y|<\epsilon$, the right hand side of the above inequality is strictly positive.
\end{proof}

\begin{cor}\label{c1}
Let $A$ be any coset of $\R$ in $J^k_\R$ and put $B = J^k_\R \setminus A$. Then, there are disjoint open sets $U, V \subset G_\textup{full}(\F^k_\R) \cong (\R\setminus\{0\})^2 \cup J^k_\R$ such that $A \subset U$ and $B \subset V$ (see Figure~\ref{fig:hdrf}).
\end{cor}
\begin{figure}[ht]
     \centering
    \def\svgwidth{.8\textwidth}
    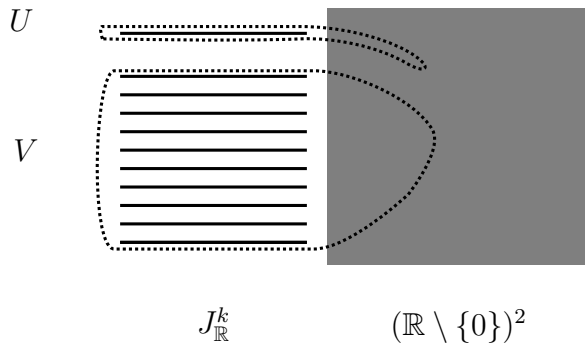
    \caption{Separating one component of $J^k_\R$ from the rest inside $G_\textup{full}(\F^k_\R)$.}
    \label{fig:hdrf}
\end{figure}
\begin{proof}
Fix $\theta_0 \in \Diff_0(\R)$. In the notation of  Lemma~\ref{graphnbhd} above, for each $\theta \in \Diff_0(\R)$ with $j^{k-1}_0(\theta)\neq j^{k-1}_0(\theta_0)$, there exists an $\epsilon_\theta$ such that $U_{\theta_0}$ and $U_{\theta} \cap (\R \times (-\epsilon_\theta,\epsilon_\theta))$ have disjoint images in $G_\textup{full}(\F^k_\R)$. Let $U$ be the image of $U_\theta$ and let $V$ be the union of the images of the $U_{\theta} \cap (\R \times (-\epsilon_\theta,\epsilon_\theta))$, ranging over $\theta \in \Diff_0(\R)$ with $j^{k-1}_0(\theta)\neq j^{k-1}_0(\theta_0)$.
\end{proof}
\begin{cor}\label{hdrffreg}
The topology of $G_\textup{full}(\F^k_\R)$ is Hausdorff and regular.
\end{cor}

\begin{propn}\label{notnormal}
For every $k \geq 2$, the topology of $G_\textup{full}(\F^k_\R))$ is not normal.
\end{propn}
\begin{proof}
For notational convenience we take $k=2$; the argument for $k>2$ is essentially the same. Let $A \subset J^2_\R$ consist of all $a_1y+a_2y^2 \in J^2$ with $a_1$ rational.  Put $B = J^2_\R \setminus A$. Then $A$ and $B$ are disjoint closed sets partitioning $J^2_\R$ (each is a union of cosets). In a similar spirit to Niemytzki's Tangent Disc (see \cite{Steen-Seebach}, pp.~100), one can use the Baire category theorem to show that $A$ and $B$ cannot be separated by disjoint open sets in $G_\textup{full}(\F^2_\R)$. 
\end{proof}

We summarize various topological  properties  of $G_\textup{full}(\F^k_\R)$ discussed above in the following theorem. 

\begin{thm}\label{thm:topprop}
For every $k\geq 2$, the full holonomy groupoid $G_\textup{full}(\F^k_\R)$:
\begin{enumerate}
\item is equipped with a smooth atlas (i.e. is a smooth space, in the sense of Section~\ref{smoothspaceterm}),
\item is Hausdorff and regular, but not normal,
\item is separable, but not second countable,
\item has two connected components.
\end{enumerate}
\end{thm}
\begin{proof}
(1) holds for the full holonomy groupoid of any almost regular foliation (see Section~\ref{sec:almostregsmooth}). (2) is a repetition of Corollary~\ref{hdrffreg} and Corollary~\ref{notnormal}.  $G_\textup{full}(\F^k_\R)$ is separable because it contains  $(\R\setminus\{0\})^2$ as a dense open subset. It is not second countable because it contains $J^k_\R$ and $J^k_\R$ is not second-countable. (4) is a repetition of Proposition~\ref{2comp}. 
\end{proof}

\section{Transverse order \texorpdfstring{$k$}{k} foliations}\label{sec:transordk}

We now define this articles's main objects of study and give some of their basic properties. Recall (Definition~\ref{FkRdef}) that $\F^k_\R \coloneqq \F\{y^k \frac{d}{dy}\}$. This is the foliation of $\R$ consisting of all  compactly-supported vector fields on $\R$ which vanish to order $k$ or more at $0$.

\begin{defn}\label{def:transordk}
Let $M$ be a connected, smooth manifold. A foliation $(M,\F)$ is a \textbf{transverse order $\mathbf{k}$} foliation if:
\begin{enumerate}
\item for each $x \in M$, there  exists an open set $U \subset M$  with $x \in U$ and a local submersion $p : U \to \R$ such that $p^{-1}(\F^k_\R) = \F_U$.
\item $\F$ has exactly one (connected) singular leaf $L$.
\end{enumerate}
Given a transverse order $k$ foliation $(M,\F)$, we  call a submersion such as the one in (1) a \textbf{local $\mathbf{\F}$-$\F^k_\R$-submersions}.  
\end{defn}

\begin{rmk}
The important assumption above is (1). Assumption (2)  
is included mainly for convenience. 
\end{rmk}

The prototypical example of a transverse order $k$ foliation is the following.

\begin{ex}\label{prototype}
Let $\ell$ be a positive integer and $n \coloneqq \ell+1$.   Equip $\R^n = \R^\ell \times \R$ with coordinates $(x,y) = (x_1,\ldots, x_\ell, y)$. Then 
\[ \F^k_{\R^n} \coloneqq \F \left\{y^k\tfrac{\partial}{\partial y},\tfrac{\partial}{\partial x_1}, \ldots, \tfrac{\partial}{\partial x_\ell} \right\}\] 
is a transverse order $k$ foliation of $\R^n$. Indeed, $\F^k_{\R^n} = {\mathrm{pr}_2}^{-1}( \F^k_\R)$, where $\mathrm{pr}_2$ is the final coordinate projection $(x,y) \mapsto y$. The singular leaf of $\F^k_{\R^n}$ is the horizontal hyperplane $\R^\ell \times \{0\}$. 
\end{ex}

Every transverse order $k$ foliation of an $n$-dimensional manifold is locally isomorphic to the foliation in Example~\ref{prototype}.

\begin{propn}\label{zoom}
Let $(M,\F)$ be a transverse order $k$ foliation with singular leaf $L$ and $n \coloneqq \dim(M)\geq 2$. Then, for any $x_0 \in L$, there exists a diffeomorphism $\theta : U \to V$, where $U \subset M$ is an open neighbourhood of $x_0$ and $V \subset \R^n = \R^\ell \times \R$ is an open neighbourhood of $(0,0)$, such that $\theta(x_0)=(0,0)$ and   $\theta_*(\F_U) = (\F^k_{\R^n})_V$.  Moreover, given any local $\F$-$\F^k_\R$-submersions $p$ and any local retraction $\pi$ onto $L$, both defined near $x_0$, there exists a diffeomorphism $\theta$ under which  $p$ becomes $\mathrm{pr}_2$ and $\pi$ becomes $\mathrm{pr}_1$. 
\end{propn}
\begin{proof}
Let $p:U_0 \to \R$ be a local $\F$-$\F^k_\R$-submersions with $x_0 \in U_0$. Any submersion is locally a Euclidean projection so, shrinking $U_0$ and choosing coordinates appropriately, we may identify $U_0$ with a ball in $\R^n$ centred at $x_0=(0,0)$ in such a way that $p= \mathrm{pr}_2|_{U_0}$. We 
then have $\F_{U_0} = (\mathrm{pr}_2|_{U_0})^{-1}(\F^k_\R) = (\F^k_{\R^n})_{U_0}$. In particular, $L \cap U_0 = (\R^\ell \times \{0\})\cap U_0$. Now, possibly shrinking $U_0$ further, let $\pi :U_0 \to \R^\ell$ be a submersion satisfying $\pi(x,0)=x$ for all $(x,0) \in U_0 \cap (\R^\ell \times \{0\})$. By 
the inverse function theorem, there is a smaller ball $U \subset U_0$ centred on the origin such that $(x,y) \mapsto (\pi(x,y),y)$ defines a diffeomorphism $\theta : U \to V$, where $V = \theta(U) \subset U_0$ as well. Since, by construction, $\mathrm{pr}_2|_V \circ \theta = \mathrm{pr}_2|_U$, we have $\theta_*(\F_U) = \theta_*((\mathrm{pr}_2|_U)^{-1}(\F^k_\R)) =(\mathrm{pr}_2|_V)^{-1}(\F^k_\R) = (\F^k_{\R^m})_V$. 
\end{proof}

\begin{cor}
Let $(M,\F)$ be a transverse order $k$ foliation with singular leaf $L$. Then, $L$ is a closed submanifold of codimension-$1$.
\end{cor}
\begin{proof}
The above proposition shows that $L$ is a codimension-$1$ embedded submanifold of $M$. Since the other leaves of $\F$ are open, it follows that $L$ is closed.
\end{proof}

\section{Local results on transverse order \texorpdfstring{$k$}{k} foliations}\label{localresultssection}

In this section, we study the prototypical transverse order $k$ foliation $\F^k_{\R^n}$ of Example~\ref{prototype}. We assume throughout that $k \geq 2$. The main results for $\F^k_{\R^n}$ are Theorem~\ref{local}, which characterizes $\F^k_{\R^n}$-$\F^k_\R$-submersions (Definition~\ref{def:transordk}) in terms of their infinitesimal  behaviour along the singular hyperplane, and Theorem~\ref{fullGforRn}, which computes the restriction of the full holonomy groupoid to the singular hyperplane. Also of importance is  Theorem~\ref{germ=nullorbequiv} which shows that, for any transverse order $k$ foliation $(M,\F)$ and $x$ a point in the singular leaf, $k$-jet equivalence of $\F$-$\F^k_\R$-submersions at $x$ is the same as orbit equivalence under the action of the group of null $\F$-automorphisms at $x$.

Let $n \geq 2$ be an integer and put  $\ell \coloneqq n-1$.  We equip $\R^n = \R^\ell \times \R$ with coordinates $(x,y) = (x_1,\ldots, x_\ell, y)$. The transverse order $k$ foliation under discussion is:  
\[ \F^k_{\R^n} \coloneqq \F \left\{y^k\tfrac{\partial}{\partial y},\tfrac{\partial}{\partial x_1}, \ldots, \tfrac{\partial}{\partial x_\ell} \right\} =\mathrm{pr}_2^{-1}(\F^k_\R).  \] 
The singular leaf of $\F^k_{\R^n}$ is the horizontal hyperplane:
\[ L \coloneqq \R^\ell \times \{0\}. \]
The following   result shows that $\F^k_{\R^n}$-automorphisms and $\F^k_{\R^n}$-$\F^k_\R$-submersions are closely related.

\begin{propn}\label{difftosub}
Let $\theta : \R^n \to \R^n$ be a diffeomorphism. Then, $\theta$ preserves $\F^k_{\R^n}$ if and only  if $p^{-1}(\F^k_\R) = \F^k_{\R^n}$, where $p \coloneqq \mathrm{pr}_2 \circ \theta :  \R^n \to \R$. 
\end{propn}
\begin{proof}
We have $\theta^*(\F^k_{\R^n}) = \theta^{-1}( \mathrm{pr}_2 ^{-1}(\F^k_\R)) = p^{-1}(\F^k_\R)$. 
\end{proof}

The following terminology will be convenient.

\begin{defn}
Let $\theta$ be a diffeomorphism of $\R^n$. We say that $\theta$ is \emph{vertical} if it has the form $\theta(x,y) = (x,\theta_x(y))$  where $\theta_x$, $x \in \R^\ell$ is a smoothly-varying diffeomorphism of $\R$. Similarly, we say that $\theta$ is \emph{horizontal} if it has the form $\theta(x,y) = (\theta_y(x), y)$ where $\theta_y$, $y \in \R$ is a smoothly varying diffeomorphism of $\R^\ell$. 
\end{defn}
Since a horizontal diffeomorphism $\theta$ satisfies $\mathrm{pr}_2\circ\theta=\theta$, the following is an immediate consequence of Proposition~\ref{difftosub}.
\begin{lemma}
Any horizontal diffeomorphism preserves $\F^k_{\R^n}$.\qed 
\end{lemma}
On the other hand, an inverse function theorem argument gives the following.
\begin{lemma}\label{verhorizdecomp}
Let $\theta$ be any diffeomorphism of $\R^n$ which preserves $L$. Then, locally near any point of $L$, we can write $\theta= \theta_h \circ \theta_v$ where $\theta_v$ is vertical and $\theta_h$ is horizontal. \qed
\end{lemma}
The crux, therefore, is to understand which vertical diffeomorphisms $\theta$ preserve $\F^k_{\R^n}$. It is obviously necessary that $\theta$ preserve $L$, but this does not suffice. For example, $(x,y)\mapsto (x,e^xy)$ does not preserve $\F^2_{\R^2}$ (see Example~\ref{planeex}).
\begin{lemma}\label{theta'}
Let $\theta(x,y)=(x,\theta_x(y)$ be a  vertical diffeomorphism of $\R^n$ which preserves $L$. Then $\theta$ preserves $\F^k_{\R^n}$ if and only if $\theta_x'(0), \theta_x''(0), \ldots, \theta_x^{(k-1)}(0)$ are independent of $x$. 
\end{lemma}
\begin{proof}
Suppose that $\theta$ is vertical. Firstly, even without the assumptions on the derivatives, one has that $\theta$ preserves the  foliation singly-generated by $y^k \frac{\partial}{\partial y}$. Indeed, for any $f \in C_c^\infty(\R^n)$, we have $\theta^*(f \frac{d}{dy}) = (f \circ \theta^{-1}) \theta^*(\frac{\partial}{\partial y})$. Since $f \mapsto f \circ \theta^{-1}$ preserves the ideal $I^k_L \subset C_c^\infty(\R^n)$ of functions vanishing to order $k$ on $L$, and since $\theta_*(\frac{\partial}{\partial y}) = \theta_x'(y) \frac{d}{dy}$ where $\theta_x'(y)$ is nowhere vanishing, the claim follows.

Next,  $\theta_*(\tfrac{\partial}{\partial x_i}) = \tfrac{\partial}{\partial x_i} +  (f_i \circ \theta^{-1}) \tfrac{\partial}{\partial y}$, where $f_i(x,y) = \tfrac{\partial}{\partial x_i} \theta_x(y)$. The foliation $\F^k_{\R^n}$ is preserved by $\theta$ if and only if $f_i \circ \theta^{-1}$, or equivalently $f_i$, belongs to  the ideal $I^k_L$ for  $i=1,\ldots,\ell$. In other words, we need $\frac{\partial}{\partial x_i} \theta_x^{(r)}(0) =0$ for all $x \in \R^\ell$, $i=1,\ldots, \ell$, $r=1,\ldots, k-1$, proving (2). 
\end{proof}

We now translate Lemma~\ref{theta'} into the following result about submersions which will play an important role.

\begin{thm}\label{local}
Let $U \subset \R^n$ be a convex open set containing $(0,0)$. Let  $p : U \to \R$ be a submersion with $p^{-1}(0) = L \cap U$. Then, the following are equivalent:
\begin{enumerate}
\item $p^{-1}(\F^k_\R) = (\F^k_{\R^n})_U$, i.e. $p$ is a local $\F^k_{\R^n}$-$\F^k_\R$-submersion.
\item $\frac{\partial^rp}{\partial y^r}$ is constant on $L \cap U$ for $r=1,\ldots,k-1$. 
\item There exist constants $a_1,\ldots,a_{k-1} \in \R$ with $a_1 \neq 0$ and a smooth function $f$ on $L\cap U$  such that $j^k_{(x_0,0)}(p) = a_1 y + \ldots a_{k-1} y^{k-1} + f(x_0) y^k$ for all $(x_0,0) \in L \cap U$.
\item There exist constants $a_1,\ldots,a_{k-1} \in \R$ with $a_1 \neq 0$ and a smooth function $f$ on $U$ such that $p(x,y)=a_1 y + \ldots a_{k-1} y^{k-1} + f(x,y) y^k$ for all $(x,y) \in U$. 
\end{enumerate}
\end{thm}
\begin{proof}
Define $\theta:U \to \R^n$ by $\theta(x,y) = (x,p(x,y))$. By the inverse function theorem,  $\theta$ is a diffeomorphism in a neighbourhood of $L$. Shrinking $U$, we may assume that $\theta$ maps $U$ diffeomorphically onto $\theta(U)$.  By Proposition~\ref{difftosub}, $\theta$ preserves $\F^k_{\R^m}$ if and only if $p^{-1}(\F^k_\R) =(\F^k_{\R^n})_U$. The equivalence of statements (1) and (2) then follows from  Lemma~\ref{theta'}. Obviously (3) implies (2). Conversely, (2) and the fact that $p$ itself vanishes on $L$ imply that 
\[ 
\left( \tfrac{\partial}{\partial x_1} \right)^{\alpha_1}
\cdots
\left( \tfrac{\partial}{\partial x_\ell} \right)^{\alpha_\ell}
\left( \tfrac{\partial}{\partial y} \right)^\beta
p  (x,0) = 0\]
whenever $\beta \leq k-1$ and at least one of $\alpha_1,\ldots,\alpha_\ell$ is nonzero. Statement (3) follows. Clearly (4) implies (3) and, by a Taylor series argument, (4) implies (3) as well. 
\end{proof}

One may interpret the above theorem as saying that the infinitesimal structure of local $\F^k_{\R^n}$-$\F^k_\R$-submersions is very rigid along the singular leaf $L$:  their $k$th order Taylor expansions involve only  the variable $y$, and none of the variables $x_1,\ldots,x_\ell$. Furthermore, the coefficients of $y,\ldots, y^{k-1}$ remain constant as the basepoint of the Taylor expansion varies in $L$.  As a further demonstration of this rigidity principle, the following corollary says that, working locally and up to order $k-1$, any two local $\F$-$\F^k_\R$-submersions are related by a unique polynomial.

\begin{cor}
Let $(M,\F)$ be any transverse order $k$ foliation with singular leaf $L$. Let $p$ and $q$ be local $\F$-submersions defined at a point $x_0 \in L$. Then, on some neighbourhood $U$ of $x_0$, there exist unique constants $a_1,\ldots, a_{k-1} \in \R$ and a unique smooth, real-valued function $f:U \to \R$ such that 
\[ q = a_1 p + \ldots a_{k-1} p^{k-1}  + f p^k \]
holds on $U$. Necessarily, $a_1 \neq 0$. 
\end{cor}
\begin{proof}
By Proposition~\ref{zoom}, we may assume $M=\R^n$, $\F=\F^k_{\R^n}$, $x_0=(0,0)$ and $p = \mathrm{pr}_2$, whence the claim follows from Theorem~\ref{local}~(4).
\end{proof}

Taking Proposition~\ref{difftosub} and Theorem~\ref{local} together gives a good understanding of  the structure of $\F^k_{\R^n}$-automorphisms and $\F^k_{\R^n}$-$\F^k_\R$-submersions. Our next task is to mod out by null automorphisms  (Definition~\ref{nulldef}). It is therefore necessary to involve some $\F^k_{\R^n}$-bisubmersions. We can easily  convert an $\F^k_\R$-bisubmersion into an $\F^k_{\R^n}$-bisubmersion by taking the product with the pair groupoid $L^2$. 
\begin{lemma}
If $(W,t,s)$ is an $\F^k_\R$-bisubmersion, then $(L^2\times W, \mathrm{pr}_1\times t,\mathrm{pr}_2\times s)$ 
is an $\F^k_{\R^n}$-bisubmersion.
\end{lemma}
\begin{proof}
This conclusion follows from consideration of the commutative diagrams:
\[ \begin{tikzcd}
L^2 \times W \ar[r] \ar[d,"{\mathrm{pr}_2\times\sigma}"]  & W \ar[d,"\sigma"]  & L^2 \times W \ar[r] \ar[d,"{\mathrm{pr}_1\times\tau}"]   & W \ar[d,"\tau"] \\
\R^n \ar[r] & \R & \R^n \ar[r] & \R\nospacepunct{.} 
\end{tikzcd}\]
\end{proof}
In particular, recall (Definition~\ref{omegadef}) that
\begin{align*}
\Omega=\{ (t,y) \in \R^2 : 1+ty^{k-1} > 0 \} && \sigma(t,y) =  y   && \tau(t,y)= y + ty^y
\end{align*}
defines an $\F^k_\R$-bisubmersion $\Omega\coloneqq(\Omega,\tau,\sigma)$. More generally (Definition~\ref{omegatheta}), $\Omega_\theta \coloneqq (\Omega, \theta \circ \tau, \sigma)$ is an $\F^k_\R$-bisubmersions for any $\theta \in \mathrm{Diff}_0(\R)$.

\begin{defn}\label{tildeath}
Let $\widetilde\Omega$ denote the $\F^k_{\R^n}$-bisubmersion $L^2 \times \Omega$. More generally, for any $\theta \in \Diff_0(\R)$, let $\widetilde\Omega_\theta$  denote the $\F^k_{\R^n}$-bisubmersion $L^2 \times \Omega_\theta$. 
\end{defn} 

Note that $N_0\coloneqq \{ (x,x,0,y) : (x,y) \in \R^n\}$ is a bisection of $\widetilde \Omega_\theta$ inducing the constant vertical diffeomorphism $(x,y)\mapsto (x,\theta(y))$. In particular, $\Omega=\Omega_{\id}$ carries the identity map and we may characterize the local  $\F^k_{\R^n}$-automorphisms which are null at a point $x \in L$ as the ones which are carried by $\widetilde\Omega$ at at the point $(x,x,0,0)$. For example, we have the following:
 
 \begin{lemma}\label{horiznull}
Let $\theta$ be a local $\F^k_{\R^n}$-automorphism with $\theta(0,0)=(0,0)$. If $\theta$ is horizontal, then $\theta$ is null (Definition~\ref{nulldef}) at $(0,0)$. 
\end{lemma}
\begin{proof}
Near $(0,0)$, since $\theta$ is horizontal, we may write $\theta(x,y) = (\theta_y(x),y)$ where $\theta_y$ depends smoothly on $y$. Let $N = \{ (\theta_y(x),x, 0 , y) \in \widetilde \Omega: (x,y) \in U\}$ , where $U \subset \R^n$ is an appropriately chosen neighbourhood of $(0,0)$. Then, $N$ is a bisection of $\widetilde \Omega$ through the point $(0,0,0,0)$ inducing the germ of $\theta$ at $(0,0)$. 
\end{proof}
 
 In a similar spirit, we have the following  generalization of  Lemma~\ref{omegacarry}~(2).

\begin{lemma}\label{constvertnull}
Let $\theta_0 \in \Diff_0(\R)$ be a diffeomorphism of $\R$ defined near $0$ and define  a constant, vertical diffeomorphism $\theta$ of $\R^n$  by  $(x,y) \mapsto (x,\theta(y))$. Then, $\theta$ preserves $\F^k_{\R^n}$  and is  null at $(0,0)$ if and only if $j^k_0(\theta)=y$. \qed
\end{lemma}

We state  the next result  for general foliations of transverse order $k$, though we quickly reduce to coordinates in the proof.

\begin{thm}\label{germ=nullorbequiv}
Let $(M,\F)$ be a transverse order $k \geq 2$ foliation with singular leaf $L$. Let $p$ and $q$ be local $\F$-$\F^k_\R$-submersions defined at $x_0 \in L$. Then, $j^k_{x_0}(p)=j^k_{x_0}(q)$ if and only if there exists  a local $\F$-automorphism $\theta$ which is null at $x_0$ such that $q=p\circ \theta$ on a neighbourhood of $x_0$.
\end{thm} 
\begin{proof}
Using Proposition~\ref{zoom}, we may suppose without loss of generality that $(M,\F) = (\R^n, \F^k_{\R^n})$, $x_0=(0,0)$  and $p = \mathrm{pr}_2$. By the inverse function theorem, $\eta(x,y)=(x,q(x,y))$  defines a diffeomorphism nearby to $(0,0)$. By definition, $q = p \circ \eta$ holds near $(0,0)$. It remains to confirm that $\eta$ is null at $(0,0)$. Since $j^k_{(0,0)}(q)=y$, Theorem~\ref{local}~(4) implies that we can write $q(x,y) = y + f(x,y)y^k$ near $(0,0)$ for a smooth function $f$  satisfying $f(0,0)=0$.  Then, the bisection
\[ N_f = \{ (x,x,f(x,y),y) : (x,y) \in U \}, \]
where $U$ is an appropriate neighbourhood of $(0,0) \in \R^n$, induces $\eta$. Since  $(0,0,0,0) \in N_f$ and the identity is also carried at $(0,0,0,0)\in\widetilde\Omega$, we have that $\eta$ is null at $(0,0)$. 
\end{proof}

Along similar lines, we now show that all local $\F^k_{\R^n}$-automorphisms can be brought, modulo a null automorphism, into a simple form.

\begin{propn}\label{perturbtoconstvert}
Let $x_1,x_2 \in L$ and let $\theta$ be a local $\F^k_{\R^n}$-automorphism with $\theta(x_1,0)=(x_2,0)$. Then, there is a local $\F^k_{\R^n}$ automorphism $\eta$ which is null at $(x_1,0)$ and a diffeomorphism  $\theta_0 \in \Diff_0(\R)$ such that $(\theta \circ \eta)(x,y) = (x-x_1+x_2, \theta_0(y))$ holds on a neighbourhood of $(x_1,0)$. 
\end{propn}
\begin{proof}
By translating, we may reduce to the case $x_1=x_2=0$. 
Using Lemmas~\ref{verhorizdecomp} and \ref{horiznull}, we furthermore reduce to the case where $\theta$ is a vertical $\F^k_{\R^n}$-automorphism $(x,y) \mapsto (x,\theta_x(y))$. Composing with the constant vertical diffeomorphism $(x,y) \mapsto (x,\theta_0^{-1}(y))$, we may furthermore reduce to the case where $\theta_0(y)=y$. The result then follows from Theorem~\ref{germ=nullorbequiv}, taking $p(x,y)=y$ and $q(x,y)=\theta_x(y)$. 
\end{proof}

\begin{propn}\label{tildeatlas}\leavevmode
\begin{enumerate}
\item $\left\{ \widetilde \Omega_\theta : \theta \in \Diff_0(\R)\right\}$ is  a full holonomy atlas for $\F^k_{\R^n}$. 
\item For each $\theta \in \Diff_0(\R)$, the canonical map $Q_{\widetilde\Omega_\theta} : \widetilde\Omega_\theta \to G_\textup{full}(\F^k_{\R^n})$ is a diffeomorphism onto its image.
\end{enumerate}
\end{propn}
\begin{proof}
Denote the open upper and lower half spaces of $\R^n$ by  $\R^n_+ \coloneqq \R^\ell \times \R_+$ and $\R^n_-\coloneqq \R^\ell \times \R_-$. It is easy to see that the restriction of $\widetilde\Omega_\theta$ to $\R^n\setminus L$ is isomorphic to $(\R^n_-)^2 \cup (\R_+^n)^2$ if $\theta$ is orientation-preserving and $(\R^n_-\times\R^n_+)\cup(\R^n_+\times\R^n_-)$ if $\theta$ is orientation-reversing. Thus, every  local diffeomorphism of $\R^n\setminus L$ is already carried by $\widetilde\Omega_\id$ and $\widetilde\Omega_{-\id}$. It remains to show that, given any point $x_0 \in L$ and  any local $\F^k_{\R^n}$-automorphism $\theta$ defined at $x_0$, the germ of $\theta$ at $x_0$ can by induced by a bisection of one of  $\{ \widetilde \Omega_\theta : \theta \in \Diff_0(\R)\}$. After an easy reduction, we may take $x_0=(0,0)$ and assume $\theta(0,0)=(0,0)$. By Proposition~\ref{perturbtoconstvert}, after composing $\theta$ with a local $\F^k_{\R^n}$-automorphism that is null at $(0,0)$, we may assume that $\theta(x,y)=(x,\theta_0(y))$ holds on a neighbourhood $U$ of $(0,0)$ for some  $\theta_0 \in \Diff_0(\R)$. Then, $N=\{(x,x,0,y) : (x,y) \in U\}$ is a bisection of $\widetilde \Omega_\theta$ which induces $\theta|_U$, proving (1). The proof of (2) is the same as in Proposition~\ref{omegaatlas}.
\end{proof}

 We can now show that the restriction of $G_\textup{full}(\F^k_{\R^n})$ is isomorphic, as a Lie groupoid,  to $L^2 \times J^k_\R$, the product of the pair groupoid $L^2$ and the 1-dimensional Lie group $J^k_\R$ (see Definition~\ref{weirdgroup}).

\begin{thm}\label{fullGforRn}
There is a unique isomorphism of Lie groupoids 
\[ G_\textup{full}(\F^k_{\R^n})_L \to L^2 \times J^k_\R \]
 such that
\[ Q_{\widetilde \Omega_\theta}(x_2,x_1,t,0) \mapsto (x_2,x_1,j^k_0(\theta) \circ(y+ty^k)) \]
for all $\theta\in\Diff_0(\R)$, $x_1,x_2 \in L$ and $t \in \R$.
\end{thm}
\begin{proof}
Theorem~\ref{G/Nthm} gives an isomorphism $G_\textup{full}(\F^k_\R)_L \to \mathcal{G}/\mathcal{N}$   where, for the sake of brevity, we write $\mathcal{G}\coloneqq \mathcal{G}(\F^k_{\R^n})_L$ and $\mathcal{N}\coloneqq \mathcal{N}(\F^k_{\R^n})_L$.  Given $x_1,x_2 \in L$, $\theta \in \Diff_0(\R)$, let $T_{x_2,x_1,\theta}$ denote the germ at $x_1$ of the diffeomorphism $(x,y) \mapsto (x+x_2-x_1,\theta(x))$. The set of all $T_{x_2,x_1,\theta}$ constitute a subgroupoid $\mathcal{H} \subset \mathcal{G}$. It is easy to see that $T_{x_2,x_1,\theta} \to (x_2,x_1,j^k_0(\theta)) : \mathcal{H} \to L^2 \times J^k$ is a surjective groupoid homomorphism. By Lemma~\ref{constvertnull}, the kernel of the latter homomorphism is exactly $\mathcal{H} \cap \mathcal{N}$, so that
\[ \frac{\mathcal{H}\mathcal{N}}{\mathcal{N}}\cong \frac{\mathcal{H}}{\mathcal{H} \cap \mathcal{N}} \cong L^2 \times J^k. \]
By Proposition~\ref{perturbtoconstvert}, $\mathcal{H}\mathcal{N}=\mathcal{G}$, so we have an isomorphism
\[ \mathcal{G}/\mathcal{N} \ni [T_{x_2,x_1,\theta}]  \mapsto (x_2,x_1,j^k_0(\theta) \to L^2 \times J^k.\]
Finally, note that the  bisection $N_{x_1,x_2,t} \coloneqq  \{(x+x_2-x_1,x,t,y) \in \widetilde\Omega_\theta : (x,y) \in \R^n \}$ induces the diffeomorphism $(x,y)\mapsto (x+x_2-x_1,\theta(y+ty^k)$ so that the resulting isomorphism $G_\textup{full}(\F^k_{\R^n})\to L^2 \times J^k$ indeed sends \[Q_{\widetilde\Omega_\theta}(x_2,x_1,t,0) \mapsto (x_2,x_1,j^k_0(\theta)\circ(y+ty^k)) \in L^2 \times J^k. \] 
By Proposition~\ref{tildeatlas}, the
$\widetilde\Omega_\theta$ constitute a full holonomy atlas for $\F^k_{\R^n}$ and each $Q_{\widetilde\Omega_\theta}$ is a diffeomorphism onto its image. It follows that this groupoid isomorphism  becomes a Lie groupoid isomorphism when $J^k$ is replaced by $J^k_\R$. 
\end{proof}

\section{Principal bundles of a transverse order \texorpdfstring{$k$}{k}  foliation}\label{principalbundles}

In what follows, $(M,\F)$ denotes a transverse order $k$ foliation with singular leaf $L$ and $k \geq 2$. The purpose of this section is to leverage the rigidity phenomena encountered in the preceding section  to construct certain natural principal bundles over $L$ whose elements are jets of local $\F$-$\F^k_\R$-submersions. The main result is Theorem~\ref{bunthm}.

Because the group  $\mathrm{Diff}_0(\R)$ of diffeomorphisms of $\R$ fixing $0$ preserves $\F^k_\R$, the group $\mathrm{Diff}_0(\R)$ acts by composition from the left on the set 
of local $\F$-$\F^k_\R$-submersions defined at $x \in M$. Since composition of jets is well-defined, the following definition makes sense.

\begin{defn}
Given $x \in L$ and $r \in \{1,\ldots,k\}$, we write $P^r_x(\F)$ for the set of $r$-jets at $x$ of local $\F$-$\F^k_\R$-submersions endowed with the left  action of $J^r$ that descends from the left action of $\mathrm{Diff}_0(\R)$.
Put also $P^r(\F) \coloneqq \bigsqcup_{x\in L} P^r_x(\F)$, so that $P^r(\F)$ is a bundle of sets over $L$ with an action of $J^r$ on each fiber.
\end{defn}

Our goal is to equip $P^r(\F)$ with the structure of a smooth principal bundle over $L$. Our constructions will rely on the  following elementary lemma which provides a mechanism by which a smooth principal bundle structure may be induced from  an appropriate family of sections. In essence, this is the construction of a principal bundle from a 1-cocycle.

\begin{lemma}\label{bundlemaker}
Let $H$ be a Lie group and $M$ a smooth manifold. Suppose that $P$ is a set equipped with a map $\pi : P \to M$ and an action of $H$ that is free and transitive on each fiber of $\pi$.  Let $\{U_i\}_{i \in I}$ be an open cover of $M$ and let  $\{ s_i : U_i \to P\}_{i \in I}$ be a collection of local sections of $\pi$ such that the transition maps $\{h_{ij} : U_i \cap U_j \to H\}_{i,j \in I}$, uniquely defined by   $s_i(x) = h_{ij}(x)  s_j(x)$ for all $x \in U_i \cap U_j$, are smooth. There is a unique smooth structure on $P$ making it a smooth principal $H$-bundle with respect to which each  $s_i$ is a smooth section. \qed
\end{lemma}

Recall (Definition~\ref{weirdgroup}) that $J^r_d$ denotes $J^r$ considered as a discrete group and $J^r_\R$ denotes $J^r$ with a certain  one-dimensional Lie group structure.

\begin{thm}\label{bunthm}
Let $(M,\F)$ be a transverse order $k$ foliation with singular leaf $L$ and $k \geq 2$.  
\begin{enumerate}
\item There is a unique principal $J^k_\R$-bundle structure on $P^k(\F)$ such that, for any local $\F$-$\F^k_\R$-submersions $p:U\to \R$, the map $x \mapsto j^k_x(p) : U \cap L \to P^k(\F)$ is a smooth section. 
\item If $1 \leq r \leq k-1$, there is a unique principal $J^r_d$-bundle structure on $P^r(\F)$ such that, for any local $\F$-$\F^k_\R$-submersions $p:U\to \R$, the map $x \mapsto j^r_x(p) : U \cap L \to P^r(\F)$ is a smooth section. 
\end{enumerate}
\end{thm}
\begin{proof}
First we claim that, for any $r \in \{1,\ldots,k\}$, the action of $J^r$ on $P^r(\F)$ is free and transitive on fibers. By  Proposition~\ref{zoom}, it suffices to consider the case $(M,\F)= (\R^n,\F^k_{\R^n})$ and work near the point $(0,0) \in \R^n$. By Theorem~\ref{local}, given a local $\F^k_{\R^n}$-$\F^k_\R$-submersion  $p$ defined at $(0,0)$, the jet $j^k_{(0,0)}(p)$ (a priori a polynomial of degree $\leq r$ in $x_1,\ldots, x_\ell,y$) has the form $a_1y+\ldots+a_ky^k$ where $a_i \in \R$, $a_1 \neq 0$. The action of $J^k$, meanwhile, is the usual compose and truncate operation. It follows that $J^r$ acts freely and transitively on $P^r_{(0,0)}(\F)$ for $r=1,\ldots, k$.

Next, suppose that $p,q  : U\to \R$ are local $\F^k_{\R^n}$-$\F^k_\R$-submersions, where $U \subset \R^n$ is a convex, open neighbourhood of $(0,0)$. Without loss of generality (Proposition~\ref{zoom}) we may take $p=\mathrm{pr}_2$. By Theorem~\ref{local}, we have \[j^k_{(0,x_0)}(q) = a_1y+\ldots+a_{k-1} y^{k-1} + f(x_0) y^{k-1}\]
for $x_0 \in U \cap \R^\ell$, where $f : U \cap \R^\ell \to \R$ is smooth. The map $h : U \cap \R^\ell \to J^k_\R$ defined by $h(x_0) = a_1 y +\ldots + a_{k-1}y^{k-1} + f(x_0) y^k$ is smooth (the smooth structure on $J^{k-1}_\R$ only permits us to vary the coefficient of $y^k$). We have $j^k_{(x_0,0)}(q) = h \circ j^k_{(x_0,0)}(p)$ for all $x_0 \in U \cap \R^\ell$ and so, applying Lemma~\ref{bundlemaker}, we obtain (1). The proof of (2) is similar, keeping fewer terms of the Taylor expansions.
 \end{proof}

By construction, these bundles are functorial for foliation-preserving diffeomorphisms defined near the singular leaf.

\begin{propn}\label{bunfunc}
For $i=1,2$, let $(M_i,\F_i)$ be transverse order $k$ singular foliations with singular leaves $L_i$. Suppose $U_i \subset M_i$ is an open set containing $L_i$ and $\theta:  U_1 \to U_2$ is a diffeomorphism with $\theta_*((\F_1)_{U_1}) = (\F_2)_{U_2}$. Then  $\theta$ restricts to a diffeomorphism $\theta_0:L_1 \to L_2$ and,  for $r=1,\ldots,n$,  there is a unique isomorphism of principal bundles $P^r(\theta) : P^r(\F_1)\to P^r(\F_2)$ covering $\theta_0$ sending  $j^r_x(p) \mapsto j^r_{\theta(x)} (p \circ \theta^{-1})$, whenever $p$ is a local $\F_1$-$\F^k_\R$-submersions defined at $x \in L$.
\[ \begin{tikzcd}
P^r(\F_1) \ar[d] \ar[r,"P^r(\theta)"] &  P^r(\F_2) \ar[d] \\
L_1 \ar[r,"\theta|_L"] & L_2
\end{tikzcd} \]
\end{propn}
\begin{proof}
We may assume $U_1=M_1$, $U_2=M_2$. Then  $\theta : M_1 \to M_2$ induces a bijection $p \mapsto p \circ \theta^{-1}$ from the set of local $\F_1$-$\F^k_\R$-submersions to the set of local $\F_2$-$\F^k_\R$-submersions. This bijection commutes with the left action of $\mathrm{Diff}_0(\R)$ by composition. Passing to $r$-jets gives the result. 
\end{proof}

Since $P^{k-1}(\F)$ has discrete structure group $J^{k-1}_d$, taking monodromy immediately gives the following.  It is appropriate to think of an  $\F$-$\F^k_\R$-submersion at a point $x\in L$ as a dual version of a transversal. Accordingly,  the statement below may be interpreted as saying that a path in $L$ induces a holonomy transformation between transversals  \emph{at the level of $(k-1)$-jets}, in alignment with the classical notion of  holonomy for regular foliations.

\begin{cor}
Suppose $(M,\F)$ is a transverse order $k$ singular foliation with singular leaf $L$. Then a path  $c:[0,1] \to L$ from a point $x$ to a point $y$ induces a $J^{k-1}_d$-equivariant map  $P^{k-1}_x(\F) \to P^{k-1}_y(\F)$ defined by sending $p_0 \mapsto p(1)$ where $p:[0,1]\to P^{k-1}(\F)$ is the unique lift of $c$ with $p(0)=p_0$.  
\end{cor}

We furthermore use the monodromy of the $J^{k-1}_d$-bundle $P^{k-1}(\F)$ to define the following invariant of a transverse order $k$ foliation. See Section~\ref{sec:monodromy} for notation and terminology.

\begin{defn}\label{invdef}
Let $(M,\F)$ be a transverse order $k$ foliation with singular leaf $L$. The \textbf{holonomy invariant} of $\F$ is
\[ h(\F) \coloneqq h(P^{k-1}(\F)) \in [\pi_1(L),J^{k-1}].\]
That is, $h(\F)$ is the monodromy invariant (Definition~\ref{def:monodromy}) of  $P^{k-1}(\F)$.
\end{defn}

Here,  $[\pi_1(L),J^{k-1}]$ denotes the quotient of the set $\Hom(\pi_1(M,x_0),J^{k-1})$ by the conjugation action of $J^{k-1}$, where  $x_0 \in L$ is any basepoint, as in Section~\ref{sec:monodromy}.

The following proposition shows that $h(\F)$ is indeed an \emph{invariant} of $\F$. More precisely,  $h(\F)$ is an ``$L$-local invariant'' in the sense that it only depends on the  restriction of $\F$  to any neighbourhood of its singular leaf.

\begin{propn}
For $i=1,2$, let $(M_i,\F_i)$ be  transverse order $k$ singular foliations with singular leaves $L_i$. Suppose $U_i \subset M_i$ is an open set containing $L_i$ and $\theta:  U_1 \to U_2$ is a diffeomorphism with $\theta_*((\F_1)_{U_1}) = (\F_2)_{U_2}$. Then  $\theta$ restricts to a diffeomorphism $\theta_0:L_1 \to L_2$  and $(\theta_0)_*(h(\F_1)) = h(\F_2)$. We refer to  Proposition~\ref{[pi,Gamma]induce} for the definition of the induced  map  $(\theta_0)_*: [\pi_1(M_1),J^{k-1}] \to [\pi_1(M_2),J^{k-1}]$.
\end{propn}
 
\begin{proof}
First apply Theorem~\ref{bunfunc} and then Theorem~\ref{thm:monodromy}~(1). 
\end{proof}

If we take the underlying $J^1$-bundle of the $J^1_d$-bundle $P^1(\F) \to L$ and identify $J^1$ with $\operatorname{GL}(1,\R)$ in the obvious way, then we see that, through the usual correspondence between vector bundles and their frame bundles, $P^1(\F)$ determines a flat line bundle over $L$. As one might guess (since a $1$-jet $j^1_x(p)$ of a real-valued map is essentially the same thing as its differential $dp_x$),  this line bundle is canonically isomorphic to $\nu_M^*(L)$, the conormal bundle of $L$ in $M$. For convenience, we identify $\nu_M^*(L)$ with  the subbundle of $T^*M|_L$ which annihilates $TL \subset TM|_L$. We omit the verification  of the following.

\begin{propn}\label{bottproof}
Let $(M,\F)$ be a transverse order $k\geq 2$ foliation with singular leaf $L$. View $P^1(\F)$ as a $\mathrm{G}(1,\R)$-bundle, as described above. Then
\[ \R\underset{\operatorname{GL}(1,\R)}{\times} P^1(\F) \ni[\lambda,j^1_x(p)]  \mapsto \lambda dp_x \in \nu_M^*(L) \]
is an isomorphism of line bundles. \qed
\end{propn}

Thinking of $P^1(\F)$ as a flat $\operatorname{GL}(1,\R)$-bundle, we may therefore make the following definition.

\begin{defn}\label{bottdef}
Let $(M,\F)$ be a transverse order $k \geq 2$ foliation with singular leaf $L$. The \textbf{Bott connection}  of $\F$ is the flat connection $b(\F)$ on the conormal bundle $\nu^*_M(L)$ induced by Proposition~\ref{bottproof} above.
\end{defn}

We conclude this section with the observation that one  really only needs to construct the bundle $P^k(\F)$; the bundles $P^r(\F)$ for $r <k$ can be recovered as quotients of the former. Recall that if $P$ is a principal $H$-bundle and $K$ is a closed normal subgroup of $H$, then $P/K$ is naturally a principal $H/K$ bundle over the same base.

\begin{propn}
Let $r$ be an integer with $2 \leq r \leq k$ and identify $J^{r-1}$ with $J^r/\R$. Then the map  $P^r(\F) \to P^{r-1}(\F)$  given by taking the underlying $(r-1)$-jet of an $r$-jet induces an isomorphism of $J^{r-1}_d$-bundles $P^r(\F)/\R \to P^{r-1}(\F)$. \qed
\end{propn}

\section{Gauge groupoid description of full holonomy groupoid}\label{sec:gaugefull}

Suppose $(M,\F)$ is a transversely order $k$ singular foliation with singular leaf $L$. Since the restriction  $G_\textup{full}(\F)_L$ of the full holonomy groupoid to $L$ is transitive, it must be a gauge groupoid (see Definition~\ref{gaugedef}). In this section, we show that it is isomorphic to  the  gauge groupoid of the principal $J^k_\R$-bundle $P^k(\F)$ constructed in Section~\ref{principalbundles}. The main result is Theorem~\ref{gaugeiso}.

The following lemma indicates the method we use to produce elements of the gauge groupoid. The proof is a simple algebraic verification, which we omit.

\begin{lemma}\label{lambda}
Let $P \to B$ be a (smooth, left) principal $H$-bundle. Given $x,y \in B$, denote by $\Lambda_x^y$ the set of all maps  $\lambda : P_y \times P_x \to H$ that satisfy
\[
\lambda(kq,h p) = k \lambda(q,p) h^{-1} 
\]
for all $h,k \in H$, $p \in P_x$, $q \in P_y$.
\begin{enumerate}
\item  Given  $\lambda \in \Lambda_x^y$, the element 
\[ \overline \lambda \coloneqq [q,\lambda(q,p)p] \in \Gauge(P)_x^y \]
is independent of choice of $p \in P_x$,  $q \in P_y$.
\item $\lambda \mapsto \overline \lambda$ defines a bijection  $\Lambda_x^y \to  \Gauge(P)_x^y$.
\item Given $x,y,z \in B$; $\lambda_1 \in \Lambda_x^y$; $\lambda_2 \in \Lambda_y^z$; $\lambda \in \Lambda_x^z$, the following are equivalent:
\begin{enumerate}
\item $\overline \lambda_2 \overline \lambda_1 = \overline \lambda$ 
\item $\lambda_2(r,q)\lambda_1(q,p) = \lambda(r,p)$ for all $p \in P_x$, $q \in P_y$, $r \in P_z$. \qed
\end{enumerate}
\end{enumerate}
\end{lemma}

We need the following simple lemma concerning  bisubmersions.

\begin{lemma}\label{qplemma}
Let $(M,\F)$ and $(N,\F_N)$ be singular foliations. Let $U,V \subset M$ be open and let $p : U \to N$ and $q : V \to N$ be submersions such that $p^{-1}(\F_N) = \F|_U$ and $q^{-1}(\F_N)=\F|_V$.
\begin{enumerate}
\item If $(W,t,s)$ is an $\F$-bisubmersion, then $W_{q,p} \coloneqq(s^{-1}(U) \cap t^{-1}(V), q \circ t,p \circ s)$ is an $\F_N$-bisubmersion.
\item Suppose $W'$ is another  $\F$-bisubmersion and let $w \in W_{q,p}$, $w' \in  W'_{q,p}$. If there is a local morphism of $\F$-bisubmersions from $W$ to $W'$ sending $w_1 \mapsto w_2$, then there  is also a local morphism of $\F_N$-bisubmersions from $W_{q,p}$ to $W'_{q,p}$ sending $w_1\mapsto w_2$. 
\end{enumerate}
\end{lemma}
\begin{proof}
For brevity, put $s_N \coloneqq p \circ s$, $t_N \coloneqq q \circ t$ and  $\F_W \coloneqq s_N^{-1}(\F_N) =t_N^{-1}(\F_N)$. Since $C_c^\infty(\ker(ds_N)) \subset s_N^{-1}(\F_N)$ and $C_c^\infty(\ker(dt_N)) \subset t_N^{-1}(\F_N)$, we have \[ C_c^\infty(\ker(ds_N)) + C_c^\infty(\ker(dt_N)) \subset \F_W.\] On the other hand, $\ker(ds) \subset \ker(ds_N)$ and $\ker(dt) \subset \ker(dt_N)$, so 
\[ \F_W = C_c^\infty(\ker(ds)) + C_c^\infty(\ker(dt)) \subset C_c^\infty(\ker(ds_N)) + C_c^\infty(\ker(dt_N)), \] 
establishing (1). For (2), one needs  simply to restrict the domain of the given morphism appropriately.
\end{proof}

\begin{rmk}\label{pqmap}
The above lemma shows that $p:U \to N$ and $q: V \to N$ determine a well-defined function (typically not a groupoid morphism) $G_\textup{full}(\F)_U^V \to G_\textup{full}(\F_N)$ sending $Q_W(w) \mapsto Q_{W_{q,p}}(w)$. This function is smooth when  $\F$ and $\F_N$ are almost regular.
\end{rmk}

The  notation in the following definition will facilitate the statement of Theorem~\ref{gaugeiso}.

\begin{defn}\label{lambdadef}
Let $(M,\F)$ be a transverse order $k\geq 2$ foliation with singular leaf $L$. Let $(W,t,s)$ be an $\F$-bisubmersion, let $w \in W$ and assume that $x_1 \coloneqq s(w)$ and $x_2 \coloneqq t(w)$ belong to $L$. Let $p$ and $q$ be $\F$-$\F^k_\R$-submersions defined at $x_1$ and $x_2$ respectively. Then, we put
\[ \lambda_W(w,q,p) \coloneqq j^k_0(\theta) \in J^k, \] 
where $\theta$  is any diffeomorphism of $\R$ carried by the associated $\F^k_\R$-bisubmersion $W_{q,p}$ (see Lemma~\ref{qplemma},~(1)) at $w$. Equivalently, $\lambda_W(w,q,p)$ is the image of $Q_{W_{q,p}}(w) \in G(\F^k_\R)_0$ under the isomorphism $G(\F^k_\R)_0 \to  J^k$ of Proposition~\ref{Jkdescrip}. 
\end{defn}

\begin{lemma}\label{lambdawelldef} 
The $k$-jet $\lambda_W(w,q,p)$ in Definition~\ref{lambdadef} above only depends on:
\begin{enumerate}
\item the groupoid element  $Q_W(w) \in G_\textup{full}(\F)$, and
\item the $k$-jets $j^k_{x_1}(p) \in P^k_{x_1}(\F)$ and $j^k_{x_2}(q) \in P^k_{x_2}(\F)$.
\end{enumerate}
\end{lemma}
\begin{proof}
Lemma~\ref{qplemma},~(2) gives immediately that $\lambda_W(w,q,p)$ only depends on $Q_W(w)$ (this is in the same vein as Remark~\ref{pqmap}). It remains to investigate the dependence of $\lambda_W(w,q,p)$ on $p$ and $q$. To this end, suppose $p'$ and $q'$ are local $\F$-$\F^k_\R$-submersions with $j^k_{x_1}(p)=j^k_{x_1}(p')$ and $j^k_{x_2}(q)=j^k_{x_2}(q')$. Then, by Theorem~\ref{germ=nullorbequiv}, there exists, for $i=1,2$, a  local $\F$-automorphism $\theta_i$ which is null at $x_i$, such that $p' = p \circ \theta_1$ near $x_1$ and $q'=q\circ\theta_2$ near $x_2$. Because $\theta_1,\theta_2$ are null at $w$, the $\F$-bisubmersion $W'\coloneqq W_{\theta_2,\theta_2}$ carries the same local $\F$-automorphisms at $w$ as does $W$. Equivalently, by Lemma~\ref{simtoapprox}, there exists a local morphism from $W$ to $W'$ sending $w \mapsto w$. Thus, applying Lemma~\ref{qplemma},~(2), there is a local morphism from $W_{q,p}$ to $W'_{q,p} =W_{q\circ\theta_2,p\circ\theta_1} = W_{q',p'}$ sending $w\mapsto w$ and it follows that $\lambda_W(w,q,p)=\lambda_W(w,q',p')$. 
\end{proof}

We have come to the main result of this section.

\begin{thm}\label{gaugeiso}
Let $(M,\F)$ be a transverse order $k \geq 2$ foliation with singular leaf $L$. There is a unique isomorphism of Lie groupoids
\[ G_\textup{full}(\F)_L \to \operatorname{Gauge}(P^k(\F)) \]
such that  if $(W,t,s)$ is an $\F$-bisubmersion,
 and if  $w \in W$ is such that $x_1 \coloneqq s(w)$ and $x_2 \coloneqq t(w)$ belong to $L$, then
\[ Q_W(w) \mapsto 
 [ j^k_{x_2}(q), \lambda_W(w,q,p) \cdot j^k_{x_1}(p)],\]
where $p$ and $q$ are any  local $\F$-$\F^k_\R$-submersions defined at $x_1$ and $x_2$, respectively.
\end{thm}
 \begin{proof}
Suppose that $W$ is a $\F$-bisubmersion and $w\in W$ has $s(w) \in L$. Let $p$ be an $\F$-$\F^k_\R$-submersions defined at $s(w)$ and $q$ and $\F$-$\F^k_\R$-submersions defined at $t(w)$. If  $W_{q,p}$ carries $\theta$ at $w$ and  $\theta_1, \theta_2 \in \Diff_0(\R)$ are given, then it is easy to see $W_{q\circ\theta_2,p\circ\theta_1}$ carries $\theta_2 \circ \theta \circ \theta_1^{-1}$ at $w$. Therefore, combining Lemma~\ref{lambdawelldef} and Lemma~\ref{lambda}~(2), the procedure in the theorem statement determines a well-defined map $G_\textup{full}(\F)_L \to \operatorname{Gauge}(P^k(\F))$.

Next we claim the map $G_\textup{full}(\F)_L \to \operatorname{Gauge}(P^k(\F))$ under discussion preserves groupoid multiplication. Let $(W,t,s)$ and $(W',t',s')$ be $\F$-bisubmersions and suppose $w \in W$ and $w'\in W'$ have $s'(w')=t(w) \in L$. Put $x_1 \coloneqq s(w)$, $x_2 \coloneqq s'(w') = t(w)$, $x_3 \coloneqq t'(w')$. Let $p_i$ be an $\F$-$\F^k_\R$-submersions defined at $x_i$ for $i=1,2,3$. Observe that $W'_{p_3,p_2} \circ W_{p_2,p_1}$ is a  closed submanifold of  $(W' \circ W)_{p_3,p_1}$. The latter observation implies that, if  $\theta$ is carried by $W_{p_2,p_1}$ at $w$ and $\theta'$ is carried by $W'_{p_3,p_2}$ at $w'$, then $\theta' \circ \theta$ is carried by $(W' \circ W)_{p_3,p_1}$ at $(w',w)$. Thus, 
\[ \lambda_{W'}(w',p_3,p_2)\lambda_W(w,p_2,p_1) = \lambda_{W' \circ W}((w',w),p_3,p_1) \]
so that, by Lemma~\ref{lambda}~(3), the map $G_\textup{full}(\F)_L \to \operatorname{Gauge}(P^k(\F))$ preserves the groupoid multiplication.

It remains to check  the groupoid morphism $ G_\textup{full}(\F)_L \to \operatorname{Gauge}(P^k(\F))$ is a diffeomorphism. Because of the local nature of the problem, it suffices (by Proposition~\ref{zoom})  to consider the case $(M,\F)=(\R^n,\F^k_{\R^n})$, $L=\R^\ell\times\{0\}$ which was studied in detail in Section~\ref{localresultssection}. We do so for the remainder of the proof. In this case, the principal $J^k_\R$-bundle $P^k(\F)\to L$ is isomorphic to $L \times J^k_\R$, courtesy of  the global section induced by the global $\F$-$\F^k_\R$-submersion $\mathrm{pr}_2 : \R^n \to \R$. Correspondingly (Example~\ref{trivalgauge}), there is an isomorphism  of Lie groupoids:
\[ L^2 \times J^k_\R\ni (x_2,x_1,\alpha) \mapsto [j^k_{x_2}(\mathrm{pr}_2), \alpha \cdot  j^k_{x_2}(\mathrm{pr}_2)] \in \Gauge(P^k(\F)).  
\]
We also have already a Lie groupoid isomorphism 
\[ G_\textup{full}(\F) \to (\R^\ell)^2 \times J^k_\R \] 
by Theorem~\ref{fullGforRn}. It remains, therefore, to check that the diagram 
\[\begin{tikzcd}
G_\textup{full}(\F) \ar[rd] \ar[rr] && \Gauge(P^k(\F)) \\
&L^2 \times J^k_\R \ar[ru] & 
\end{tikzcd} \]
is commutative. For this purpose, recall the full holonomy atlas $\{\widetilde\Omega_\theta:\theta \in \R\}$ (Definition~\ref{tildeath}). Fix $\theta \in \Diff_0(\R)$; $x_1,x_2 \in \R^\ell$; $t \in \R$ and put $w \coloneqq (x_2,x_1,t,0) \in \widetilde\Omega_\theta$. It is  straightforward to chase $Q_{\widetilde\Omega_\theta}(w) \in G_\textup{full}(\F)$ both ways around the above diagram and confirm that result is the same.
 \end{proof}

\section{Extracting the holonomy groupoid from the full holonomy groupoid}\label{sec:gaugemin}

The holonomy groupoid  of a singular foliation is contained in the full holonomy groupoid as an open subgroupoid. To be precise, $G(\F)$ is the $s$-connected component $G_\textup{full}(\F)$ (c.f. Proposition~\ref{openmap} and \cite{AS[2007]}, Theorem~0.1).

\begin{defn}
The \textbf{$\mathbf{s}$-connected component} of a Lie groupoid $G$ is the subgroupoid $G_0 \subset G$ consisting of all $g \in G$ that can be connected to $s(g)$ by a path in $G_{s(g)}$. 
\end{defn}

In the preceding section, we obtained  a description
\[ G_\textup{full}(\F) = (M\setminus L)^2 \cup \Gauge(P^k(\F)) \]
for the full holonomy groupoid (Theorem~\ref{gaugeiso}). As a corollary, we have:

\begin{thm}\label{thm:extract}
Let $(M,\F)$ be a transverse order $k \geq 2$ foliation with singular leaf $L$. If $M\setminus L$ is connected, then 
\[ G(\F)\cong (M\setminus L)^2 \cup \Gauge(P^k(\F))_0. \]
If $M\setminus L$ has two connected components $U$ and $V$, then 
\[ G(\F)\cong U^2 \cup V^2 \cup \Gauge(P^k(\F))_0. \]
Here, $\Gauge(P^k(\F))_0$ denotes the $s$-connected component of $\Gauge(P^k(\F))$. The smooth structure of $G(\F)$ is the one inherited from $G_\textup{full}(\F)$. \qed
\end{thm}

Although the bundle $P^k(\F)$ is in fact a manifold in our sense (i.e. is metrizable), it is still very large; both it and its structure group have continuum-many components. One may therefore desire a more concrete description of $G(\F)_L$ which avoids this bundle. The following proposition, whose proof we omit, shows that we in fact only need to deal with one connected component of $P_0 \subset P^k(\F)$, which is automatically a  second-countable manifold.

\begin{propn}\label{subgauge}
Let $P \to B$ be a (smooth, left) principal $H$-bundle, where $H$ is a (possibly disconnected) Lie group. 
\begin{enumerate}
\item If $p \in P$, then the set of $h \in H$ for which $hp$ belongs to the same  component of $P$ as $p$ is a closed and open subgroup $H_p \subset H$. 
\item If $p$ and $q$ belong to the same component of $P$, then $H_p = H_q$. 
\item Let $P_0$ be a component of $P$. Then $P_0$ is a principal $H_0$-bundle, where $H_0$ is the stabilizer of $P_0$ or, equivalently, any particular point of $P_0$. 
\item The inclusion  $P_0 \times P_0 \to P \times P$ descends to a Lie groupoid isomorphism from  $\Gauge(P_0)$ onto the $s$-connected component $\Gauge(P)_0 \subset \Gauge(P)$. \qed
\end{enumerate}
\end{propn}

\begin{cor}
Let $(M,\F)$ be a transverse order $k \geq 2$ foliation with singular leaf $L$. Then $G(\F^K)_L$ is isomorphic to $\Gauge(P_0)$, where $P_0$ is any connected component of the principal $J^k$-bundle $P^k(\F)$. 
\end{cor}

The following result  gives additional insight into the  holonomy \emph{groups} of a transverse order $k$ foliation at points of its singular leaf.

\begin{thm}\label{isotrope}
Let $(M,\F)$ be a transverse order $k$  foliation with singular leaf $L$. Fix $x_0 \in L$ and $q_0 \in P^{k-1}(\F)_{x_0}$. Let $\gamma : \pi_1(L,x_0) \to J^{k-1}$ be the  monodromy homomorphism determined by $q_0$ and let $\Gamma \subset J^{k-1}$ be the range of $\gamma$. Finally, let $\Gamma_\R \subset J^k$ be the preimage of $\Gamma$ under the projection $P^k(\F) \to P^{k-1}(\F)$. Then, the isotropy group of $G(\F)$ at any point of $L$ is isomorphic, as a Lie group, to the one-dimensional group $\Gamma_\R$. For convenience, the following diagram summarizes the relationship between the various groups at hand:
\[ \begin{tikzcd} 
\R \ar[r] & J^k_\R \ar[r] & J^{k-1}_d \\
\R \ar[u,equals] \ar[r] & \Gamma_\R \ar[u] \ar[r] & \Gamma \ar[u] \\
 & & \pi_1(L,x_0)\nospacepunct{.} \ar[u,"\gamma"] 
\end{tikzcd} \] 
\end{thm}
\begin{proof} 
Let $Q_0$ denote the connected component of $P^{k-1}(\F)$ containing $q_0$. By standard covering space theory, the subgroup $\Gamma$ equals the stabilizer of $Q_0$ so that, by Proposition~\ref{subgauge}, $Q_0$  becomes a principal $\Gamma$-bundle. Let $P_0$ be the preimage of $Q_0$ under the projection $P^k(\F)\to P^{k-1}(\F)$. Since the fibers of this projection are copies of $\R$, it follows that $P_0$ is a connected component of $P^k(\F)$. Thus, by Proposition~\ref{subgauge}~(4), 
\[ \Gauge(P_0)\cong \Gauge(P^k(\F))_0 \cong G(\F)_L. \] 
The subgroup of $J^k_\R$ stabilizing $P_0$ is $\Gamma_\R$. By Proposition~\ref{subgauge}~(3), $P_0$ is a principal $\Gamma_\R$-bundle,  whence every isotropy group of $\Gauge(P_0)$ is isomorphic to $\Gamma_\R$. 
\end{proof}

\begin{cor}\label{C*L}
With notation as in Theorem~\ref{isotrope}, we have 
\[ C^*(G(\F)_L) \cong C^*(\Gamma_\R)\otimes \K,\] 
where $\K$ denotes the compact operators on the canonical $L^2$ space of $L$. 
\end{cor}
\begin{proof}
From \cite{MRW}, Theorem~3.1 it follows that, if $G\rightrightarrows M$ is any transitive Lie groupoid (hence a gauge groupoid), then $C^*(G)\cong C^*(H)\otimes \K( L^2(M))$, where $H$ denotes any isotropy group of $G$. 
\end{proof}
 
 \begin{rmk}
The homomorphisms $\gamma$ in Theorem~\ref{isotrope} are exactly the same  ones used to define the holonomy invariant $h(\F)$  (Definition~\ref{invdef}), so Corollary~\ref{C*L}  shows that knowledge of the invariant $h(\F)$  is enough to compute  $C^*(G(\F)_L)$.
 \end{rmk}

\section{Transverse order \texorpdfstring{$k$}{k} foliations on line bundles}\label{sec:linbun}

Throughout this section, $L$ is a connected, smooth manifold. If $E \to L$ is a smooth line bundle, we always identity  $L$ with the zero section of $E$, so that $L \subset E$.

\begin{defn}\label{linebundef}
Fix a positive integer $r$. 
\begin{itemize}
\item If $E\to L$ is a line bundle, we write  $J^r(E,\R)\to L$ for the   smooth, left principal $J^r$-bundle whose fiber over $x \in L$ consists of all $r$-jets at $x$ of local diffeomorphisms $E_x \to \R$ sending $x \mapsto 0$.
\item If $U\subset E$ is open and $p : U \to \R$ is a submersion with $p^{-1}(0)=U\cap L$, we write $j^r(p) : U\cap L \to J^r(E,\R)$ for the smooth section defined by $j^r(p)(x) \coloneqq j_x^r(p|_{E_x \cap U})$.
\item If $\pi_i:E_i \to L$ is a line bundle for $i=1,2$,  we write $J^r(E_1,E_2)\to L$ for the smooth fiber bundle whose fiber over $x \in L$ consists of $r$-jets at $x$ of (local) diffeomorphisms $(E_1)_x\to (E_2)_x$ sending $x \mapsto x$.
\item Let $\theta : U_1\to U_2$ be a diffeomorphism, where $U_i\subset E_i$ is open for $i=1,2$. We say that $\theta$ is \textbf{fiberwise} if   $\pi_2(\theta(e))=\pi_1(e)$ for all $e\in U_1$ and $\theta(L \cap U_1)=L\cap  U_2$. If $\theta$ is fiberwise, we write $j^r(\theta) : U \cap L \to J^r(E_1,E_2)$ for the   smooth section defined by $j^r(\theta)(x) = j^r_x(\theta|_{E_x \cap U})$.
\end{itemize}
\end{defn}

The following proposition shows that all sections of these jet bundles can be lifted to honest maps.

\begin{propn}\label{choosepoly}
Let $r$ be a positive integer. 
\begin{enumerate}
\item Let $E$ be a line bundle over $L$. Then, every local section of $J^r(E,\R)$ is equal to $j^r(p)$ for some local submersion $p:U \to L$ with $p^{-1}(0)=L \cap U$. 
\item Let $E_1$ and $E_2$ be line bundles over $L$.  Then, every local section of $J^r(E_1, E_2)$ is equal to $j^r(\theta)$ for some local fiberwise diffeomorphism $\theta$. 
\end{enumerate}
\end{propn}
\begin{proof}
Both statements follow from the fact that  each $r$-jet between a pair of one-dimensional vector spaces preserving $0$ is represented by a unique polynomial mapping of degree $\leq r$.
\end{proof}

If $P\to M$ and $Q\to M$ are smooth principal $H$-bundles, we write $\Hom_H(P,Q)$ for the usual smooth fiber bundle over $M$ whose fiber over $x \in M$ is the the set of $H$-equivariant maps $P_x\to Q_x$. We note the following without proof.

\begin{lemma}\label{J(1,2)}
If $E_1$ and $E_2$ are line bundles over $L$, then there is a fiber bundle isomorphism  $\alpha \mapsto \alpha_* : J^r(E_1,E_2) \to \Hom_{J^r}(J^r(E_1,\R),J^r(E_2,\R))$ defined by $\alpha_*(\beta) = \beta \circ \alpha^{-1}$.  \qed
\end{lemma}

In particular,  if $\theta$ is a fiberwise diffeomorphism from $E_1$ to $E_2$ defined on a neighbourhood of $L$,   the corresponding global section $j^r(\theta):L\to J^r(E_1,E_2)$ determines a principal bundle isomorphism $j^r(\theta)_* : J^r(E_1,\R)\to J^r(E_2,\R)$. Thus,  $E \mapsto J^r(E,\R)$ is functorial for fiberwise diffeomorphisms. In fact, by Proposition~\ref{choosepoly}, \emph{every} isomorphism $J^r(E_1,\R) \to J^r(E_2,\R)$ is induced by a fiberwise diffeomorphism.

The following two results are the main findings of this section.

\begin{thm}\label{flatcorresp}
Let $\pi : E \to L$ be a smooth line bundle with connected base manifold $L$. There is a unique bijection $\F \mapsto \nabla_\F$ from
\begin{enumerate}
\item the set of transverse order $k$ foliations of $E$ with singular leaf $L$  to
\item the set of (smooth) flat connections on $J^{k-1}(E,\R)$
\end{enumerate}
such that a submersion $p : U \to \R$, where $U \subset E$ is open, satisfying $p^{-1}(0)=U \cap L$ is a local $\F$-$\F^k_\R$-submersion  if and only if the associated section  $j^{k-1}(p) : U\cap L\to J^{k-1}(E,\R)$  is parallel for $\nabla_\F$.
\end{thm}
\begin{proof}
Suppose $U \subset E$ is open and $p,q :U \to \R$ are submersions with $p^{-1}(0)=q^{-1}(0)=U\cap L$. Up to shrinking $U$ around $U\cap L$, we may assume that $p_x \coloneqq p|_{U\cap E_x}$ and $q_x \coloneqq q|_{U \cap E_x}$ are diffeomorphisms onto their images for all $x \in U\cap L$. Define $\theta_x \coloneqq q_x \circ (p_x)^{-1}$ for all $x \in U\cap L$ so that, by construction,  $\theta_x$ is a smooth  family of local diffeomorphisms of $\R$ satisfying  $\theta_{\pi(e)}(p(e))=q(e)$ on a neighbourhood of $U \cap L$. Define $h_{qp}:U\cap L \to J^{k-1}$ by  $h_{qp}(x)=j^{k-1}_0(\theta_x) \in J^{k-1}$. By construction, $h_{qp}$ is the transition function for the sections $j^{k-1}(p)$ and $j^{k-1}(q)$ (Definition~\ref{linebundef}).

Now suppose $\F$ is a transverse order $k$ foliation of $E$ with singular leaf $L$ and $p,q:U\to \R$  are local $\F$-$\F^k_\R$-submersions. We claim that the transition map $h_{qp} : U\cap L \to J^{k-1}$ is locally constant. By Proposition~\ref{zoom}, given any point $x_0 \in U\cap L$, we may choose coordinates so that $x_0$ is the point $(0,0)$ in $\R^n = \R^\ell \times \R$ and the foliation is $\F^k_{\R^n}$. Moreover, we may perform this coordinate change in such a way that $p:E\to \R$ and $\pi:E\to L$ are given by $p(x,y)=y$ and $\pi(x,y)=x$ (the linear structure of the fibers is not likely to be preserved under this coordinate change, but that is irrelevant). By Theorem~\ref{local}, there are constants $a_1,\ldots,a_{k-1}$ and a smooth function $f$ such that 
\[ \theta_x(y) =  q(x,y) = a_1y+\ldots + a_{k-1} y^{k-1} + f(x,y)y^k \]
holds on a  neighbourhood of $(0,0)$, whence $h_{qp}(x)=j_0^{k-1}(\theta_x)$ is constant on a neighbourhood of any $x_0 \in U \cap L$. Therefore, applying Lemma~\ref{bundlemaker}, there is a unique $J^{k-1}_d$-bundle structure on  $J^{k-1}(E,\R)$ for which $j^{k-1}(p)$ is a smooth section for every local $\F$-$\F^k_{\R}$ submersion $p$. This shows the map $\F \mapsto \nabla_\F$ is well-defined.

Conversely, suppose that $\nabla$ is a flat connection on $J^{k-1}(E,\R)$. For each local submersion $p : U_p \to \R$ with $p^{-1}(0)=L\cap U_p$ satisfying $j^{k-1}(p)=U_p \cap L$, define $\F_p = p^{-1}(\F^k_\R)$. Thus, $\F_p$ is a foliation of $U_p$. We use Proposition~\ref{foliatedgluing} to glue these foliations together into a foliation $\F$ of $M$. To this end, it suffices to consider $p$ and $q$ with $U_p=U_q=U$. It is easy to see that $\F_p$ and $\F_q$ both restrict to the trivial one leaf foliation on $U \setminus L$. It remains only to check the foliations agree on a neighbourhood of each point  $x_0 \in L$. Using that $j^{k-1}(p)$ and $j^{k-1}(q)$ are $\nabla$-parallel and shrinking $U$ about $x_0$, there exists $h \in J^{k-1}$ with 
\[ j^{k-1}(q)=h \cdot j^{k-1}(p). \] 
We may pass again to coordinates such that $x_0=(0,0) \in \R^\ell \times \R$ and $\pi(x,y)=x$, $p(x,y)=y$ near $(0,0)$. Writing $h=a_1y + \ldots +a_{k-1}y^{k-1}$,  the equation above then says that $\frac{\partial^rq}{\partial y^r}(x,0)=a_r$ holds near $(0,0)$ for $r=1,\ldots,k-1$. Thus, applying Theorem~\ref{local} again, we have that, in these coordinates, $p^{-1}(\F^k_\R)$ and $q^{-1}(\F^k_\R)$ agree with $\F^k_{\R^n}$ in a neighbourhood of $(0,0)$. 

It is easy to see that these mappings between foliations and flat connections are inverses of each other.
\end{proof}

Next we show that the correspondence of the above theorem is natural with respect to fiberwise diffeomorphisms.

\begin{thm}\label{flatleafwise}
Let $L$ be a connected, smooth manifold. For $i=1,2$, let $E_i$ be a smooth line bundle over $L$ equipped with a transverse order $k$ foliation $\F_i$ with singular leaf $L$.  Suppose that $\theta : U_1 \to U_2$ is a fiberwise diffeomorphism, where $U_i$ is an open subset of $E_i$ containing $L_i$. Then, the following are equivalent:
\begin{enumerate}
\item $\theta_*((\F_1)_{U_1})=(\F_2)_{U_2}$. 
\item The induced $J^{k-1}$-bundle isomorphism $j^{k-1}(\theta)_* : J^{k-1}(E_1,\R) \to J^{k-1}(E_2,\R)$ pushes forward   $\nabla_{\F_1}$ to $\nabla_{\F_2}$. 
\end{enumerate}
\end{thm}
\begin{proof}
Assume for convenience that $U_1=E_1$, $U_2=E_2$. Let $U \subset L$ be open and let $p:U\to L$ be a submersion satisfying $p^{-1}(0)=U \cap L$. Clearly (1) is equivalent to the assertion:  $p$ is a local $\F_1$-$\F^k_\R$-submersion if and only if $p\circ \theta^{-1}$ is a local $\F_2$-$\F^k_\R$-submersion. Meanwhile (applying Proposition~\ref{choosepoly}~(2)), (2) is equivalent to the assertion: $j^{k-1}(p)$ is $\nabla_{\F_1}$ parallel if and only if $j^{k-1}(\theta)_*(j^{k-1}(p))$ is $\nabla_{\F_2}$ parallel. Since 
\[ j^{k-1}(\theta)_*(j^{k-1})(p) = j^{k-1}(p) \circ j^{k-1}(\theta)^{-1} = j^{k-1}(p \circ \theta^{-1}), \] 
the desired conclusion is an immediate consequence of Theorem~\ref{flatcorresp}. 
\end{proof}

\section{Completeness of the holonomy invariant}\label{sec:complete}

Let $(M,\F)$ be a transverse order $k$ foliation with singular leaf $L$. In this section, we show that the holonomy invariant (Definition~\ref{invdef}) \[ h(\F) \in [\pi_1(L),J^{k-1}]\] 
together with the diffeomorphism type of $L$ is a complete invariant for the structure of $\F$ nearby to $L$.

Given a line bundle equipped with a transverse order $k$ foliation whose singular leaf is the zero section, there are ostensibly two principal $J^{k-1}_d$-bundles at play: the one from Section~\ref{principalbundles} and the one from Section~\ref{sec:linbun}. As one would probably suspect, these two bundles are in fact the same.

\begin{propn}\label{PtoJ}
Let $E\to L$ be a smooth line bundle and let $\F$ be transverse order $k$ foliation of $E$ with singular leaf $L$. There is an isomorphism of principal $J^{k-1}_d$-bundles $P^{k-1}(\F) \to J^{k-1}(E,\R)$ such that, if $p : U \to \R$ is a local $\F$-$\F^k_\R$-submersion and $x \in U \cap L$, then $j^k_x(p) \mapsto j^k_x(p|_{E_x\cap U})$. Here, $J^k(E,\R)$ is given the structure of a $J^{k-1}_d$-bundle using the flat connection $\nabla_\F$ of Theorem~\ref{principalbundles}.
\end{propn} 
\begin{proof}
The map is clearly well-defined and $J^{k-1}$-equivariant. Given a local $\F$-$\F^k_\R$-submersion $p:U\to \R$, the sections 
\begin{align*} 
x \mapsto j^{k-1}_x(p) : U \cap L\to P^{k-1}(\F) && x \mapsto j^{k-1}_x(p|_{E_x \cap U}) : U \cap L \to J^{k-1}(E,\R)
\end{align*}
are smooth as part of the definition of those bundles. The first of these  smooth local sections is clearly mapped to the second, and from this it follows that the map $P^{k-1}(\F) \to J^{k-1}(E,\R)$ is smooth.
\end{proof}

We now state and prove the main result of the section. One may refer to Subsection~\ref{sec:monodromy} and Definition~\ref{invdef} for context.

\begin{thm}\label{holcompl}
For $i=1,2$, let $(M_i,\F_i)$ be  transverse order $k$ foliations with singular leaf $L_i$ . Suppose there is a diffeomorphism $\theta_0 : L_1 \to L_2$ whose induced map  $[\pi_1(L_1),J^{k-1}] \to [\pi_1(L_2),J^{k-1}]$, given by pushing forward loops, sends $h(\F_1)$ to $h(\F_2)$. Then, we can extend $\theta_0$ to a diffeomorphism $\theta : U_1 \to U_2$, where $U_i \subset M_i$ are open neighbourhoods of $L_i$, so as to have $\theta_*(\F_1|_{U_1}) = \F_2|_{U_2}$. 
\end{thm}
\begin{proof}
Without loss of generality, $L=L_1=L_2$ and $\theta_0 = \id_L$.  Thus $h(\F_1) = h(\F_2)$. By Definition~\ref{invdef}, this means exactly that the principal $J^{k-1}_d$-bundles $P^{k-1}(\F_1)$ and $P^{k-1}(\F_2)$ have the same monodromy invariant so, by Proposition~\ref{thm:monodromy}, there is an isomorphism of principal  $J^{k-1}_d$-bundles $P^{k-1}(\F_1) \to P^{k-1}(\F_2)$ covering the identity map on $L$.

By passing to tubular neighbourhoods, we may assume that $M_i$ is the total space of a line bundle $\pi_i : E_i \to L$, with $L$ contained in $E_i$ as the zero section. Then, applying Proposition~\ref{PtoJ}, there is a  $J^{k-1}$-bundle isomorphism $J^{k-1}(E_1,\R) \to J^{k-1}(E_2,\R)$, equivalently a section of $J^{k-1}(E_1,E_2)$, which moreover  maps $\nabla_{\F_1}$ to $\nabla_{\F_2}$.  By Proposition~\ref{choosepoly}, we can take this section of the form $j^{k-1}(\theta)$ for $\theta$ some fiberwise diffeomorphism $E_1 \to E_2$ preserving $L$. Then, by Theorem~\ref{flatleafwise}, $\theta_*(\F_1|_{U_1})=\F_2|_{U_2}$, as desired. 
\end{proof}

As a corollary, when $L$ is simply connected, there is only one ``$L$-local'' isomorphism class of transverse order $k$ foliation with singular leaf $L$.

\begin{cor}
Let $(M,\F)$ be a transverse order $k$ foliation with singular leaf $L$. If $L$ is simply connected, then there exists a diffeomorphism $\theta:U \to L \times \R$, where $U \subset M$ is an open neighbourhood of $L$, such that $\theta|_L=\id_L$ and $\theta_*(\F_U)$ is the product the trivial one-leaf foliation of $L$ with $\F^k_\R$. 
\end{cor}

More generally, the same conclusion  holds when there is no nontrivial homomorphism $\pi_1(L)\to J^{k-1}$. For example, this is the case when  $\pi_1(L)$ is a simple, nonsolvable group (see Proposition~\ref{solvable}).

\section{Range of the holonomy invariant}\label{sec:ran}

In this section we show that the range of the holonomy invariant (Definition~\ref{invdef}) is as large as one could hope it to be. Specifically, we prove the following:

\begin{thm}\label{holrange}
Let $L$ be a connected, smooth manifold, let $x_0 \in L$ and let $\gamma: \pi_1(L,x_0) \to J^{k-1}$ be a group homomorphism. There exists a line bundle $E \to L$ and a transverse order $k$ foliation $\F$ on $E$ with singular leaf $L$ (identified with the zero section of $E$) whose holonomy invariant $h(\F)$ is equal to the class of $\gamma$ in $[\pi_1(L),J^{k-1}]$. 
\end{thm}

Together, Theorems~\ref{holcompl} and \ref{holrange} imply that the problem of enumerating the ``$L$-local'' isomorphism classes of  transverse order $k$-foliations with a fixed singular leaf $L$ is equivalent to the problem of enumerating homomorphisms $\gamma : \pi_1(L)\to J^{k-1}$ modulo a certain equivalence relation. Specifically:
\[ \gamma \sim \alpha \circ \gamma \circ \beta \]
as $\alpha$ ranges over inner automorphisms of $J^{k-1}$ and $\beta$ ranges over automorphisms of $\pi_1(L)$ that can be induced by a diffeomorphism of $L$. 
Under a variety of assumptions on $L$, one is assured of being able to implement any automorphism of $\pi_1(L)$ with a diffeomorphism of $L$. For example, if $L$ is a surface, then the homomorphism from the mapping class group to the group of outer automorphisms of $\pi_1(L)$ is an isomorphism (the Dehn-Nielsen-Baer theorem). Another example: for $L$ a  hyperbolic 3-manifold of dimension $\geq 3$, one has that $\mathrm{Isom}(L) \to \mathrm{Out}(\pi_1(L))$ is an isomorphism by the Mostow rigidity theorem. In  cases such as these, the $L$-local isomorphism classes of transverse order $k$ foliations with singular leaf $L$ can be put into correspondence with the quotient set
$$\mathrm{Inn}(J^{k-1}) \setminus \operatorname{Hom}(\pi_1(L),J^{k-1})/ \Aut(\pi_1(L)).$$

The proof of Theorem \ref{holrange} will be an exercise in manipulating principal bundles and will have little to do with foliations per se. Let us first recall the associated principal bundle construction:

\begin{defn}
If $\varphi:G \to H$ is a homomorphism of Lie groups and $P$ is a principal $G$-bundle, then $H \times_\varphi P$ denotes  the principal $H$-bundle whose total space is the quotient  of $H \times P$ by the equivalence relation defined by $(h,p)\sim(hg,g^{-1}p)$ for all $g\in G$ and equipped with the $H$-action defined by $h[h',p]=[hh',p]$. Here we denote the class of $(h,p)$ by $[h,p]$.
\end{defn}

See for example \cite{greub-halperin}, Chapter~5. The following instances of this construction will be relevant:

\begin{ex}\label{assoc-quotbun}
If $1 \to N \to G \overset{\pi}{\to} H \to 1$ is an exact sequence of Lie groups  and if $P$ is a principal $G$-bundle, then there is a principal $H$-bundle isomorphism $P/N \to H \times_\pi P$ sending $[p]\mapsto[1,p]$. 
\end{ex}

\begin{defn}
Given a line bundle $E$, the \textbf{coframe bundle} $F^*(E)$ of $E$  is the principal $\operatorname{GL}(1)\coloneqq \operatorname{GL}(1,\R)$-bundle whose fiber at $x$ is the collection of nonzero linear functionals $E_x \to \R$. 
\end{defn}

\begin{ex}\label{silly}
Let $E$ be a  smooth line bundle and $r$ a positive integer. Denote by   $\iota$ the natural inclusion $a \mapsto ay : \operatorname{GL}(1)\to J^r$.  Then there is a  principal $J^r$-bundle isomorphism $J^r \times_\iota F^*(E) \to J^r(E,\R)$  sending $[j^r_0(\theta), \varphi] \mapsto j^r_x(\theta \circ \varphi)$ 
for all $j^r_0(\theta) \in J^r$ and  all $\varphi:E_x \to \R$.
\end{ex}

\begin{ex}\label{ex:univcover}
Let $M$ be a connected smooth manifold with basepoint $x_0$. Let $\widetilde M$ be the universal cover of $M$ with corresponding basepoint $\widetilde x_0$. Let $\Gamma$ be a discrete group and $\varphi : \pi_1(M,x_0) \to \Gamma$  a homomorphism. View $\widetilde M$ as a principal $\pi_1(M,x_0)$-bundle in the usual way using deck transformations. Put $P=\Gamma \times_\varphi \widetilde M$ and $p_0 = [1,\widetilde x_0]$. Then, $P$ is a principal $\Gamma$-bundle over $M$ for which the monodromy homomorphism $\pi_1(M,x_0) \to \Gamma$ determined by the point $p_0$ is exactly $\varphi$. 
\end{ex}

Suppose we have an exact sequence of Lie groups $1 \to N \to G \overset{\pi}{\to} H \to 1$ that is split by a homomorphism $\iota : H \to G$. In other words, suppose $G$ is presented as a semidirect product  $N \rtimes H$. Then, using the fact that $\pi \circ \iota = \id_H$, there is a natural isomorphism
\begin{align}\label{invertnat} 
P \cong H \times_\pi (G \times_\iota P) 
\end{align}
for every principal $H$-bundle $P$. On the other hand,  $\iota \circ \pi\neq\id_G$  (unless $\pi$ is an isomorphism), so there is no  reason to expect the opposite relation 
\begin{align}
\label{invert} Q \cong G\times_\iota(H \times_\pi Q) 
\end{align}
to hold for every principal  $G$-bundle $Q$. With some knowledge of classifying space theory , however,  one may prove the following. 

\begin{propn}
Let $1 \to N \to G \overset{\pi}{\to} H \to 1$ be an exact sequence of Lie groups and let  $\iota : H \to G$ satisfy  $\pi \circ \iota = \id_H$. If  either $\iota$ or $\pi$ is a homotopy equivalence, then $Q \cong G\times_\iota(H \times_\pi Q)$   for every principal $G$-bundle $Q$. 
\end{propn}
Note that, unlike the isomorphism \eqref{invertnat}, the isomorphism given by the above proposition is not natural.
\begin{proof}[Proof sketch]
A homomorphism between $G$ and $H$ which is also a homotopy equivalence induces a homotopy equivalence at the level of classifying spaces. It follows that the associated bundle construction for such a homomorphism induces a bijection between isomorphism classes of $G$-bundles and isomorphism classes of $H$-bundles. The desired claim then follows from the principle  that a one-sided inverse for an invertible map is the two-sided inverse. We refer the reader to \cite{Husemoller}, Chapter~4 or \cite{Steenrod}, Part II for relevant aspects of classifying space theory. 
\end{proof}

In the following two lemmas, we give an alternative proof of a special case of the above proposition which avoids the technology of classifying space theory.

\begin{lemma}\label{buniso1}
Let $1 \to N \to G \to H \to 1$ be short exact exact sequence of Lie groups. Let $P\to B$ and $Q\to B$ be smooth principal  $G$-bundles. If  $P/N\cong Q/N$   as principal $H$-bundles and $N$ is diffeomorphic to a finite-dimensional Euclidean space,  then $P \cong Q$    as principal $G$-bundles. 
\end{lemma}
\begin{proof}
Recall $\Hom_G(P,Q)$ is defined as the fiber bundle over $B$ whose fiber at $x \in B$ is the set of $G$-equivariant maps from $P_x$ to $Q_x$. Reformulated slightly, the statement to be proven is that $\Hom_G(P,Q) \to B$ admits a (smooth) global section. We are given that the fiber bundle $\Hom_H(P/N,Q/N) \to B$ has a global section. We may instead view $\Hom_G(P,Q)$ as a fiber bundle over $\Hom_H(P/N,Q/N)$ with typical fiber $N$. Since $N$ is locally Euclidean, it follows that the latter fiber bundle has a global section. Indeed, the Tietze extension theorem can be used to show that every fiber bundle over a manifold with Euclidean space fibers has a continuous section  (see \cite{Steenrod}, Theorem~12.2) and routine smoothing arguments can then be applied. Composing the section $B \to \Hom_H(P/N,Q/N)$ with the  section $\Hom_H(P/N,Q/N) \to \Hom_G(P,Q)$ gives the result.
\end{proof}

\begin{lemma}\label{buniso2}
Let $1 \to N \to G \to H \to 1$ be a short exact exact sequence of Lie groups and let  $\iota : H \to G$ satisfy  $\pi \circ \iota = \id_H$. Let $P \to B$ be a smooth principal $G$ bundle. If is $N$ diffeomorphic to a finite-dimensional Euclidean space, then there is a smooth isomorphism of principal $G$-bundles from $P$ to $G\times_\iota( H\times_\pi P )$. 
\end{lemma}
\begin{proof}
Let $Q \coloneqq H\times_\pi P$. Since $\pi \circ \iota = \id_H$, we have $Q\cong H  \times_\pi(G\times_\iota Q )$ as principal $H$-bundles. Noting that $Q \cong P/N$ and $H  \times_\pi(G\times_\iota Q ) \cong (G\times_\iota Q )/N$ (see Example~\ref{assoc-quotbun}), we have $P/N \cong (G\times_\iota Q )/N$. Thus, applying Lemma~\ref{buniso1}, $P$ is isomorphic to $G \times_\iota Q$, as desired. 
\end{proof}

In particular, the above lemma applies to the extension
\[ 1 \to N \to J^r \overset{\pi}{\to} \operatorname{GL}(1) \to 1,\]
where the quotient map is $a_1y+\ldots+a_ry^r \mapsto a_1$, $N$ is its kernel and $\iota(a)=ay$ provides a splitting on the right. From this we obtain:

\begin{propn}
Let $P \to B$ be any principal $J^r$ bundle, where $r$ is a positive integer. Then, there exists a line bundle $E \to B$ such that $P \cong J^r(E,\R)$ as principal $J^r$-bundles.
\end{propn}
\begin{proof}
Applying Lemma~\ref{buniso2} to $P$ with the extension described above, we obtain $P\cong J^r \times_\iota Q$ where $Q \coloneqq \operatorname{GL}(1) \times_{\pi} P$. Any $\operatorname{GL}(1)$-bundle is isomorphic to the coframe bundle of its associated vector bundle, so we have $Q \cong F^*(E)$ where $E \coloneqq \R \times_{\operatorname{GL}(1)} Q$. Thus, $P \cong J^r \times_\iota F^*(E) \cong J^r(E,\R)$ (see Example~\ref{silly}).
\end{proof}

We conclude with the proof of the main result of this section.

\begin{proof}[Proof of Theorem~\ref{holrange}]
Let $\gamma:\pi_1(L,x_0)\to J^{k-1}$ be any homomorphism. Let $P \to L$ be any  principal $J^{k-1}_d$-bundle such that the monodromy mapping associated to some  point $p_0 \in P_{x_0}$ is $\gamma$. For instance, we may use $P=J^{k-1}\times_\gamma \widetilde L$, as in Example~\ref{ex:univcover}. We may equivalently view $P$ as a principal $J^{k-1}$-bundle equipped with a flat connection.   
According to the preceding lemma, there exists a line bundle $E \to L$ and an isomorphism $P \to J^{k-1}(E,\R)$ of principal $J^{k-1}$-bundles. Push forward the flat connection on $P$ to $J^{k-1}(E,\R)$ through this isomorphism, and then use  Theorem~\ref{flatcorresp} to produce a corresponding transverse order $k$ foliation $\F$ of $E$ with singular leaf $L$. By Proposition~\ref{PtoJ}, $J^{k-1}(E,\R)$ and $P^{k-1}(\F)$ are isomorphic as $J^{k-1}_d$-bundles, so the holonomy invariant of $\F$ is as desired.
\end{proof}

\printbibliography


\end{document}